\documentclass[12pt]{amsart}
\usepackage{xcolor}
\usepackage[normalem]{ulem}

\usepackage[margin=3cm]{geometry}
\usepackage{amsthm, bbm}
\usepackage{amsmath}
\usepackage{hyperref}

\numberwithin{equation}{section} 

\theoremstyle{definition}
\newtheorem{definition}{Definition}[section]
\newtheorem{theorem}[definition]{Theorem}
\newtheorem{lemma}[definition]{Lemma}
\newtheorem{corollary}[definition]{Corollary}
\newtheorem{proposition}[definition]{Proposition}
\newtheorem{remark}[definition]{Remark}

\newtheorem{assumption}[definition]{Assumption}

\newcommand{\ba}{{\bf a}}
\newcommand{\bb}{{\bf b}}
\newcommand{\bc}{{\bf c}}
\newcommand{\bd}{{\bf d}}
\newcommand{\bx}{{\bf x}}
\newcommand{\by}{{\bf y}}
\newcommand{\bu}{{\bf u}}
\newcommand{\bv}{{\bf v}}

\newcommand{\cA}{{\mathcal A}}
\newcommand{\cB}{{\mathcal B}}
\newcommand{\cC}{{\mathcal C}}

\newcommand{\cH}{{\mathcal H}}

\newcommand{\cJ}{{\mathcal J}}

\newcommand{\cL}{{\mathcal L}}

\newcommand{\cP}{{\mathcal P}}

\newcommand{\cR}{{\mathcal R}}

\newcommand{\cZ}{{\mathcal Z}}

\newcommand{\bbC}{{\mathbb C}}
\newcommand{\bbE}{{\mathbb E}}

\newcommand{\bbN}{{\mathbb N}}
\newcommand{\bbP}{{\mathbb P}}
\newcommand{\bbR}{{\mathbb R}}
\newcommand{\bbT}{{\mathbb T}}
\newcommand{\bbZ}{{\mathbb Z}}
\newcommand{\bbI}{{\mathbb I}}

\newcommand{\brF}{{\bar F}}

\newcommand{\brell}{{\bar\ell}}

\newcommand{\fa}{{\mathfrak{a}}}
\newcommand{\fb}{{\mathfrak{b}}}
\newcommand{\fc}{{\mathfrak{c}}}
\newcommand{\fR}{{\mathfrak{R}}}

\newcommand{\Var}{\text{Var}}
\newcommand{\R}{{\mathbb{R}}}
\newcommand{\T}{{\mathbb{T}}}
\newcommand{\Z}{{\mathbb{Z}}}

\newcommand{\fA}{\mathfrak{A}}
\newcommand{\fB}{\mathfrak{B}}
\newcommand{\fC}{\mathfrak{C}}
\newcommand{\fd}{\mathsf{d}}

\newcommand{\te}{{\theta}}

\newcommand{\Om}{{\Omega}}
\newcommand{\om}{{\omega}}
\newcommand{\ve}{{\varepsilon}}
\newcommand{\del}{{\delta}}

\newcommand{\sig}{{\sigma}}
\newcommand{\al}{{\alpha}}
\newcommand{\be}{{\beta}}
\newcommand{\ka}{{\kappa}}

\newcommand{\la}{{\lambda}}
\newcommand{\bQ}{{\bf Q}}

\definecolor{OliveGreen}{rgb}{0,0.6,0}

\def\DS{\displaystyle}

\newcommand{\bA}{{\mathbf{A}}}

\newcommand{\hell}{{\hat\ell}}

\newcommand{\eps}{\epsilon}

\newcommand{\brl}{\bar l}
\newcommand{\integers}{\mathbb Z}
\newcommand{\bry}{\bar y}
\newcommand{\brry}{\bar{\bar y}}

\newcommand{\fm}{\mathfrak{m}}
\newcommand{\fq}{\mathfrak{q}}
\newcommand{\sC}{\mathsf{C}}

\newcommand{\ta}{{\tilde a}}
\newcommand{\tb}{{\tilde b}}
\newcommand{\tc}{{\tilde c}}
\newcommand{\td}{{\tilde d}}

\title{Local limit theorems for expanding maps}

\begin{document}

\maketitle
 \author{\hskip4cm Dmitry Dolgopyat and Yeor Hafouta}

\tableofcontents

\begin{abstract} 
We prove local central limit theorems  for partial sums of the form \newline
$\DS S_n=\sum_{j=0}^{n-1}f_j\circ T_{j-1}\circ\cdots\circ T_1\circ T_0$ where $f_j$ are uniformly H\"older functions and $T_j$ are expanding maps. 
Using a symbolic representation a similar result follows for maps $T_j$ in a small $C^1$ neighborhood of an Axiom A map and H\"older continuous functions $f_j$. All of our results are already new when 
 all maps are the same $T_j=T$ 
but observables $(f_j)$ are different. The current paper compliments \cite{DH2} where Berry--Esseen theorems are obtained. An important  step in the proof is developing an appropriate reduction theory in the sequential case.
\end{abstract}




\renewcommand{\theequation}{\arabic{section}.\arabic{equation}}


\section{Preliminaries and main results}\label{LLT sec}
\subsection{Introduction}
\label{SSIntro}
A great discovery of the last century is that deterministic system could exhibit random behavior.
The hallmark of stochasticity is the fact that ergodic averages of deterministic systems could satisfy the 
Central Limit Theorem (CLT). The CLT states that the probability that an ergodic sum at time $N$
 belongs to an interval
of size $\sqrt{N}$ is asymptotically
the same as for the normal random variable with the same mean and variance.
In many problems one needs to see if the same conclusion holds for unit size intervals.  Such results follow from the local (central) limit theorem (LLT), which has applications to various areas of mathematics including mathematical physics
\cite{BLRB, DN16, DN-Mech, DT77, EM22, Kh49, LPRS, PS23}, number theory \cite{ADDS, Beck10, DS17}, 
geometry (\cite{LL15}), PDEs \cite{HKN}, and combinatorics (\cite{Can76,BR83, GR92, Hw98}).
Applications to dynamics include  abelian covers (\cite{BFRT, CaRa, DNP, OP19}),
suspensions flows (\cite{DN-llt}), skew products (\cite{Br05, DSL16, DDKN1, DDKN2, LB06}), and homogeneous dynamics (\cite{BeQu, BRLLT05, BRLLT23, Guiv2015, Hough2019}).
Our interest in non-autonomous
 local limit theorems is motivated among other things by applications to renormalization
(cf.  \cite{ARHW22, ARHW23, ADDS, BU}) and to random walks in random environment
(\cite{BCR16, CD24, DG13, DG21, DK20}).

Recently there was a significant interest in statistical properties of
 non-autonomous systems. 
In fact, many systems appearing in nature are time dependent due to an interaction with the outside world.
On the other hand, many powerful tools developed for studying autonomous systems are unavailable in non autonomous setting, so often a non trivial work is needed to  handle non autonomous dynamics.
In particular, the CLT for non autonomous hyperbolic systems was investigated in 
\cite{ANV, ALS09, Bk95, CLBR, CR07, DFGTV, YH Adv, YH 23, HNTV, Kifer 1998,NSV12}, see also \cite{Cogburn, Dob, Pelig, SeVa} for the CLT for inhomogeneous Markov chains.
By contrast, the LLT has received much less attention and, it has been only
established for random systems under strong additional assumptions \cite{DPZ, DFGTV, DFGTV1, HK, Nonlin, YH YT}, see also \cite{DS, MPP LLT1, MPP LLT2} for the Markovian case.
In fact, the question of local limit theorem is quite
subtle, and even in the setting of independent identically distributed (iid) random variables the local distribution 
depends on the arithmetic properties of the summands. If the summands do not take values in a lattice then the local distribution is Lebesgue, and otherwise it is the counting measure on the lattice. The case when the local distribution is Lebesgue will be refered to as the non-lattice LLT while the case when the local distribution is an appropriate counting measure will be refered to as the lattice LLT. The exact definitions are postponed to \S \ref{Sec LLT}.

In this paper for certain classes of expanding and hyperbolic maps we prove a general LLT for Birkhoff sums $S_n$ formed by a sequence of H\"older continuous functions. In fact,
we identify the complete set obstructions to the non lattice LLT for a large class of expanding maps,
 see the discussion in \S \ref{SSObstructions}. 
In the autonomous case this is done by using a set of tools called
 Livsic theory.
 In that case (see \cite{HH}) the non-lattice LLT fails only if the underlying function forming the Birkhoff sums is lattice valued, up to a coboundary. 
 In the autonomous case different notions of coboundary (measurable, $L^2$, continuous,
 H\"older, smooth) are equivalent, see \cite{PP, dLMM, Wil}, but this is 
  false in the non stationary setting.
 In the course of the proof of our main results we develop a reduction theory in the non-autonomous case, generalizing the corresponding results in the Markov case \cite{DS}.

We stress that we have no additional assumptions. In particular, we neither assume 
that the maps are random, nor that the variance of $S_n$ grows linearly in $n$. 

\subsection{Locally expanding maps}
Let $(X_j,\mathsf{d}_j)$ be metric spaces with $\text{diam}(X_j)\leq 1$. 

 In what follows we will work with the class of maps $T_j:X_j\to X_{j+1}$ considered in \cite{Nonlin} 
 (see also \cite{HK, MSU} and \cite[\S 5.2]{DH2}), which is described as follows. 
 
  \begin{assumption}\label{AssPairing} 
  {\bf (Pairing).}
 There are constants $0<\xi\leq 1$ and $\gamma>1$ such that for every two points $x,x'\in X_{j+1}$ with $\mathsf{d}_{j+1}(x,x')\leq \xi$ we can write 
 $$
T_j^{-1}\{x\}=\{y_i(x): 0\leq i\leq k_j(x)\}, \quad T_j^{-1}(x')=\{y_i(x'): 0\leq i\leq k_j(x)\}
 $$
with 
$$
\mathsf{d}_j(y_i(x),y_i(x'))\leq \gamma^{-1}\mathsf{d}_{j+1}(x,x').
$$ 
 Moreover $ \DS \sup_j \deg (T_j)<\infty$, where   $\deg(T)$ is the largest number of preimgaes that a point $x$ can have under the map $T$. 
Furthermore, the Lipschitz constant of the map $T_j$ does not exceed some constant $K_0$ which does  not depend on $j$. 
\end{assumption}

Next, for every $j$ and $n$ let 
$
T_j^n=T_{j+n-1}\circ\cdots\circ T_{j+1}\circ T_j.
$
Denote by $B_j(x,r)$ the open ball of radius $r$ in $X_j$ around  a point $x\in X_j$.

 \begin{assumption}\label{Ass n 0}
 {\bf (Covering).}
 There exists $n_0\in\bbN$ such that
 \vskip0.2cm
(i)  For every $j$ and $x\in X_j$ we have 
\begin{equation}\label{CoverLLT}
T_j^{n_0}\left(B_j(x,\xi)\right)=X_{j+n_0}.
\end{equation} 
\vskip0.2cm

(ii) For all $j$ and $y\in X_j$ there is a function $W_{j,y}:X_{j+n_0}\to B_{j}(y,\xi)$
such that 
$$
T_j^{n_0}\circ W_{j,y}=\text{id},
$$
where $\xi$ is from Assumption \ref{AssPairing}.
Moreover, the functions $W_{j,y}$ are uniformly Lipschitz. 
 \end{assumption}

 \noindent
Fix  $\beta\!\!\in\!\!(0,1]$ and let $f_j:X_j\!\!\to\!\! \bbR$ be such that 
$\DS \sup_j\|f_j\|_\beta\!\!<\!\!\infty$ where 
 $\|f_j\|_\beta\!\!=\!\sup|f_j|\!+\!
 G_{j,\beta}(f_j)$  and $G_{j,\beta}(f_j)$ is the 
H\"older constant of $f_j$ corresponding to the exponent $\be$. The main goal in this paper is to prove local limit theorems for the sequences of functions
$$
S_nf=\sum_{j=0}^{n-1}f_j\circ T_0^j
$$
considered as random variables on  $(X_0,\cB_0,\ka_0)$. Here $\cB_0$ is the Borel $\sig$-algebra of $X_0$ and $\ka_0$ belongs to a  suitable class of measures. For instance, when each $X_j$ are equipped with a reference probability measures $m_j$ such that $(T_j)_*m_j\ll m_{j+1}$ with logarithmically H\"older continuous Radon-Nikodym  derivatives then we can take $\ka_0$ to be any probability measure of the form $d\ka_0=q_0 dm_0$ with H\"older continuous density $q_0$. In the more general setting we consider two sided\footnote{Note that if there is an expanding map $T_{-1}:X_0\to X_0$ such that Assumptions \ref{AssPairing} and \ref{Ass n 0} hold with the constant sequence $(T_{-1})_{j\geq0}$ then we can always extend $(T_j)_{j\geq0}$ to a two sided sequence by setting $T_j=T_{-1}$ for $j<0$. Another example when we can extend
 dynamics to negative times is a non-stationary subshift of finite type, see \S \ref{SFT sec}.} sequences $T_j, j\in\bbZ$ such that Assumptions \ref{AssPairing} and \ref{Ass n 0} holds for all $j\in\bbZ$ and then we can take $\ka_0$ to be any measure which is absolutely continuous with respect to the time zero sequential Gibbs measure $m_0$ (see \S \ref{TPo}) with H\"older continuous density.

 \subsection{The CLT and growth of variance}\label{CLT} 
Denote $\sigma_n(\ka_0)=\sqrt{\text{Var}_{\ka_0}(S_n f)}$. If $\sigma_n=\sigma_n(\ka_0)\to\infty$ then (see \cite{CR07, DH2}) $\sigma_n^{-1}(S_nf-\ka_0(S_nf))$ converges in law to the standard normal distribution. 
By \cite[Lemma 6.3]{DH2}
 \begin{equation}
\label{EqReductionH}
f_j-\ka_0(f\circ T_0^j)=A_j+B_j-B_{j+1}\circ T_j
\end{equation}
where $A_j,B_j: X_j\to \mathbb{R}$  satisfy
$\DS \sup_j
\max(\|A_j\|_{\beta}, \|B_j\|_{ \beta})\!\!<\!\!\infty$, and
 $A_j\circ T_0^j$ is a zero mean reverse martingale. Thus, $\sigma_n(\ka_0)\not\to\infty$ if and only if 
 \begin{equation}
\label{EqVarInf}
\sum_j\text{Var}_{\kappa_0}(A_j\circ T_0^j)<\infty. 
\end{equation}  
By \cite[Remark 2.6]{DH2} and \cite[Proposition 3.3]{BulLMS} 
$\DS  \sup_n|\sigma_n^2(\ka_0)-\sigma_n^2(m_0)|<\infty$, 
so the divergence of $\sigma_n(\ka_0)$ and a decomposition \eqref{EqReductionH} with \eqref{EqVarInf} do not depend on the choice of density of $\ka_0$. One of the 
 key ingredients in the proofs of the local limit theorems is to determine when such decomposition exist  modulo a lattice, see Remark \ref{Rem Red}.



\subsection{Local limit theorems}\label{Sec LLT}

Recall that a sequence of  square integrable random variables $W_n$ with $\sig_n=\|W_n-\bbE[S_n]\|_{L^2}\to\infty$ obeys the non-lattice local central limit theorem (LLT) if for every continuous function $g:\bbR\to\bbR$ with compact support, or an indicator of a bounded interval we have 
$$
\sup_{u\in\bbR}\left|\sqrt{2\pi}{\sig_n}
\bbE[g(W_n-u)]-\left(\int g(x)dx\right)e^{-\frac{(u-\bbE[W_n])^2}{2\sig_n^2}}\right|=
 o(1).
$$

\noindent
A sequence of square integrable integer valued  random variables $W_n$  
obeys the lattice LLT if 

$$
\sup_{u\in\bbZ}\left|\sqrt{2\pi}{\sig_n}\bbP(W_n=u)
-\eta e^{-\frac{(u-\bbE[W_n])^2}{2\sig_n^2}}\right|= o(1).
$$
 To compare the two results note that the above equation is equivalent to saying 
that
for every continuous function $g:\bbR\to\bbR$ with compact support, 
$$
\sup_{u\in\bbZ}\left|\sqrt{2\pi}{\sig_n}\bbE[g(W_n-u)]
-\left(\sum_{k}g(k )\right)e^{-\frac{(u-\bbE[W_n])^2}{2\sig_n^2}}\right|= o(1).
$$

 A more general type of LLT which is valid also for reducible $(f_j)$ (defined below)
 is discussed in Section \ref{Red sec}.

\begin{definition}\label{Irr Def}
A sequence of real valued functions $(f_j)$ is reducible (to a lattice valued 
sequence) 
if there is $h\not=0$ and 
 functions $H_j: X_j\to \mathbb{R},$ $Z_j: X_j\to \mathbb{Z}$  
such that $\DS \sup_j
\|H_j\|_{\beta},<\infty$, $\DS \left(S_nH\right)_{n=1}^\infty$ is tight, and
\begin{equation}
\label{EqReduction}
f_j=\ka_0(f_j\circ T_0^j)+H_j+hZ_j.
\end{equation}


A sequence of  $\mathbb{Z}$ valued functions $(f_j)$ is reducible if it admits the
representation \eqref{EqReduction} as above with
 $h>1.$
\vskip1mm

We say that the sequence $(f_j)$ is irreducible if it is not reducible.
\end{definition}

\begin{remark}\label{Rem Red}
In the present paper we shall verify 
the tightness condition in the definition of reducibility by showing that $H_j$ admits a decomposition $H_j=A_j+B_j-B_{j+1}\circ T_j$ like in \eqref{EqReductionH} with \eqref{EqVarInf}.  
This is equivalent to $\DS \sup_n\|S_nH\|_{L^2(\ka_0)}<\infty$. As discussed in \S \ref{CLT} when \eqref{EqVarInf} fails 
 then 
$\|S_nH\|_{L^2(\ka_0)}^{-1}S_nH$ converges in law
to the standard normal distribution, whence  $S_nH$ are not tight. Thus \eqref{EqVarInf} is necessary for tightness.
Using the martingale convergence theorem we conclude that in the setup of this paper Birkhoff sums $S_nf$ of reducible sequences $(f_j)$ can be decomposed into three components: a coboundary $B_0-B_n\circ T_0^n$, a convergent Birkhoff sum $S_n H$ and a lattice valued Birkhoff sum $S_n(hZ)$. We refer to \S \ref{SSObstructions}  for an elaborated discussions on this matter.

Finally,  note that  reducibility of $(f_j)$ does not depend only on the choice of density of $\ka_0$. 
 Indeed as  discussed in \S \ref{CLT}, the divergence of $\|S_nH\|_{L^2(\ka_0)}$ depends only on $m_0$. Also, 
$|\kappa_0(f_j\circ T_0^j )-\mu_j(f_j)|=O(\delta^j)$ for some $\delta\in(0,1)$, see \cite[Remark 2.6]{DH2}, which ensures we can always absorb the change of mean in the coboundary term $B_j-B_{j+1}\circ T_j$. 
\end{remark}

Next, let $R=R(f)$ be the set of all numbers  $h\not=0$ such \eqref{EqReduction} holds with appropriate functions $A_j,B_j, Z_j$. If $(f_j)$ is 
irreducible then $R=\emptyset$. Moreover, by \cite[Theorem 6.5]{DH2}, if $\sigma(\kappa_0)\not\to\infty$ then
$R=\bbR\setminus\{0\}$ since then \eqref{EqReduction} holds  for any $h$ with $Z_j=0$. The following result completes the picture. Let $\mathsf{H}=\{1/r: r\in R\}\cup \{0\}$.

\begin{theorem}\label{Reduce thm}
 If $(f_j)$ is reducible and $\sigma_n(\ka_0)\to \infty$   then 
$$
\mathsf{H}=h_0\bbZ
$$
for some $h_0>0$. As a consequence, the number $r_0=1/h_0$ is the largest positive number such that $(f_j)$ is reducible to an $r_0\bbZ$-valued sequence. Therefore, $f_j$ can be written in the form \eqref{EqReduction} with an irreducible sequence $(Z_j)$. 
\end{theorem}

\begin{theorem}\label{LLT1}
Let $(f_j)$ be an irreducible sequence of $\mathbb{R}$ valued functions.
Under Assumptions \ref{AssPairing}  and \ref{Ass n 0}
 the sequence of random variables
  $W_n=S_nf$ obeys the non-lattice LLT if  $\sig_n(\kappa_0)\to\infty$.    
\end{theorem}

As a byproduct of the arguments of the proof of Theorem \ref{LLT1} we also 
obtain the first order Edgeworth expansions.
\begin{theorem}\label{ThEdge1}
 If $(f_j)$ is irreducible then  with $\bar S_nf=S_n f-\ka_0(S_n f)$ and $\sigma_n=\sigma_n(\kappa_0)$,
   \begin{equation}
   \label{EdgeFirst}
\sup_{t\in\bbR}\left|\ka_0(\bar S_n/\sig_n\leq t)-\Phi(t)-
  \frac{\ka_0(\bar S_n^3)(t^3-3t)}{6\sig_n^3}\varphi(t)\right|=o(\sig_n^{-1})
   \end{equation}
   where
   $\varphi(t)=\frac{1}{\sqrt{2\pi}}e^{-t^2/2}$
   is the the standard Gaussian density and 
   $\DS \Phi(t)=\int_{-\infty}^t \phi(s) ds$ is the Gaussian cumulative distribution function.
\end{theorem}

We note that by  \cite[Proposition 7.1]{DH2}
$\DS \frac{\ka_0(\bar S_n^3)}{\sigma_n^3}=O(\sig_n^{-1})$ and so the correction term $  \frac{\ka_0(\bar S_n^3)(t^3-3t)}{6\sig_n^3}\varphi(t)$ is of order $O(\sig_n^{-1})$.
\\

 Theorem \ref{LLT1} leads naturally to the questions how to check irreducibility.
We obtain several results in this direction, extending the results obtained in
\cite{DH} for Markov chains.

\begin{theorem}\label{LLT2}
If the spaces $X_j$ are connected and $\sig_n\to\infty$ then $(f_j)$ is irreducible. Therefore  
$W_n=S_nf$ obeys the non-lattice LLT. 
\end{theorem}

A partial analogue of this result in the invertible case is presented in
Theorem \ref{ThHyp}(ii) of Section \ref{Hyper Sec}. We also prove the following result.

\begin{theorem}\label{Thm small}
If $\|f_n\|_{L^\infty}\to 0$, $\DS \sup_n\|f_n\|_\beta<\infty$  and $\sig_n\to\infty$ then $(f_n)$ is irreducible and so the non-lattice LLT holds.
\end{theorem}

Next, we consider the lattice case.
\begin{theorem}\label{LLT Latt}
Let $(f_j)$ be an irreducible sequence  of $\mathbb{Z}$ valued functions.
Then the sequence of random variables
  $W_n=S_nf$ obeys the lattice LLT 
   if $\sig_n\to\infty$.    
\end{theorem}

In fact, we prove a  generalized lattice LLT for general sequences of reducible functions (see Theorem \ref{LLT RED}). This result includes Theorem \ref{LLT Latt} as a particular case, and  together with Theorem \ref{LLT1} we get a complete  description  of the local distribution of ergodic sums for the sequential dynamical systems considered in this manuscript. Since the formulation of Theorem \ref{LLT RED} is more complicated it is postponed to Section \ref{Red sec}. The more complicated limiting behavior at the local scale comes from the contributions coming from the coboundary part $B_0-B_n\circ T_0^n$ and the martingale part $\DS \sum_{j=0}^{n-1}A_j\circ T_0^j$, as will be discussed later.

\subsection{Obstructions to LLT}
\label{SSObstructions}
 The results presented above show that  there are only three obstructions for local distribution of ergodic sums to be Lebesgue:
 
 (a) {\em lattice obstruction:} the individual terms could be lattice valued in which case the sums take values in the same lattice;
 
 (b) {\em summability obstruction:} the sums can converge almost surely in which case individual terms are too large 
 to ensure the universal behavior of the sum;

(c) {\em gradient (coboundary) obstruction:} if the observable  is of the form $f_j\!\!=\!\!B_j\!\!-\!\!B_{j+1}\circ T_j$ then
the sum telescopes and the variance does not grow. 

The results presented in the previous section shows that the non-lattice LLT holds unless the observable could be decomposed as a sum of three terms satisfying the conditions (a)--(c) above (i.e. the sequential observable is reducible). Each individual obstruction has been
known for a long time. 
The lattice obstruction appears even in the iid case. In fact, the  classical LLT states that the sum of iid terms with finite 
variance satisfies the non lattice LLT unless the individual terms have lattice distribution in which case the lattice LLT
holds. The coboundary obstruction comes because the terms are not independent. In fact, the martingale coboundary
decomposition developed by Gordin \cite{Gor} allows to show for a large class of weakly dependent variables, 
including  ergodic sums of elliptic Markov chains and subshifts of finite type,
that the
sum satisfies the CLT unless the observable is a coboundary.

The summability obstruction is related to the fact that the process is not stationary. In fact, for independent 
summands combining the classical Kolmogorov three series theorem and Feller--Lindenberg CLT one sees that
if the variances of individual summands are uniformly bounded then either the variance of the sum is bounded and the sum converges almost surely or the variance of sum is unbounded and the CLT holds. The same result
remains valid for additive functionals of uniformly elliptic Markov chains \cite{Dob} and for martingales over mixing 
filtrations \cite{HaHe}. Several papers handle the situation where two obstructions are present. In particular, for
additive functionals of uniformly elliptic Markov chains, if the summands depend on finitely many variables then 
the CLT holds unless the summands can be decomposed as a sum of a gradient and convergent series,
see \cite{SeVa}.
The same result holds for expanding maps \cite{CR07}. 

Concerning the LLT, the spectral approach developed by
Nagaev \cite{Na} and extended to dynamical systems setting by Guivarch and Hardy \cite{GH} shows that 
in the stationary case, both for Markov chains and hyperbolic systems, the non lattice LLT holds unless the
summands are the sums of coboundaries and lattice valued variables. In the independent  setting it was shown in \cite{D16} that if the random variables are bounded then the LLT  holds unless the summands
can be decomposed as a sum of lattice valued and convergent terms. The only work where all three obstructions
are present\footnote{As mentioned in \S \ref{SSIntro}, several papers discuss LLT for 
random systems. However, in the random setting the summability obstruction
does not appear since a stationary non zero series can not converge almost surely.} 
is \cite{DS}, there the analogues to the results of the present paper are obtained for additive
functionals over two step elliptic Markov chains.  Extending the results
of \cite{DS} to deterministic systems requires several new ideas which 
are presented in the next subsection.

\subsection{Outline of the proofs}
In general in order to  prove a non-lattice LLT  it is enough to show that 
\begin{enumerate}
\item[(i)]  $\sig_n\to\infty$ and the CLT holds;
\vskip0.2cm
\item[(ii)] There exists $\del,c,C>0$ such that  $|\bbE[e^{it S_n}]|\leq Ce^{-ct^2\sig_n^2}$ for all $t\in[-\del,\del]$;
\vskip0.2cm
\item[(iii)] For every $T>\del$ we have 
\begin{equation}\label{Charr}
\int_{\del\leq |t|\leq T}|\bbE[e^{it S_n}]|dt=o(\sig_n^{-1}).
\end{equation}
\end{enumerate}
The CLT in our case follows from \cite{DH2}
(the proof of the CLT in the setup of \cite{DH2} follows the arguments of \cite{CR07}, but the setting of
 the present paper
is  slightly different from  \cite{CR07}).
 The second condition will be verified by 
combining \cite[Proposition~7.1]{DH2}  and \cite[Proposition 25]{DH} with $r=1$. 
The main difficulty  is to verify the third condition. 
This condition was  verified in \cite{DS} for irrreducible additive functionals
of uniformly elliptic Markov chains.
 The approach of \cite{DS} relies on the so-called structure constants,
which describe the  stable and unstable holonomies of the associated $\mathbb{R}$ extension.
In the Markov case, the fact that the past and the future are independent given the present allows to combine
the cancellations described by the structure constants at different times. In the present case we only know that 
remote past and remote future are weakly dependent which does not allow us to conclude because the sum of the structure
constants could diverge arbitrary slowly. To overcome this difficulty we use a three step approach 
for proving the non-lattice LLT.

First we show in Section \ref{Sec 5} using ideas of \cite{D98, DS} that if there are no cancelations in the characteristic function
of the sum then the corresponding structure constants are small and so by adding a coboundary the terms
can be reduced to a small neighbourhood of zero. In the case where the terms are small we use 
the complex Ruelle-Perron Frobenius-Theorem 
proved in \cite{HK} using  previous works in \cite{Rug, Du09, Du11} (in fact, in the present setting we find it more convenient to 
use a version presented in \cite[Appendix~D]{DH2}).
 The complex Ruelle-Perron-Frobenius Theorem
replaces the spectral theory of \cite{Na} by allowing perturbative computation of the characteristic function
near zero (see \cite{HK, DH, DH2}). These perturbative expansions allow to handle the case where summands
are small similarly to the iid case considered in \cite{D16}. The proof of the non-lattice LLT is then obtained by
dividing the time axis into blocks (intervals) of two types--the contracting blocks 
 where
the twisted transfer operator decay
and non contracting blocks where the perturbative arguments work.
In the case there are many contracting blocks, the decay of the twisted transfer operators
is sufficient to establish the non-lattice LLT. If the number of contracting blocks is
small, so that most of the variance comes from non contracting blocks, we need an
additional argument showing that the characteristic function can not be  large on a large 
interval. This comes from convexity of sequential pressures (which it turn follows from
Proposition \ref{PrSmVarBlock}) and plays a key role in  our argument. Namely, it shows that if $J$
is a sufficiently small interval and the characteristic function $\Phi(\xi)$ satisfies
$\|\Phi\|_{L^\infty(J)}=o(1)$ then  $\|\Phi\|_{L^1(J)}=o(1/\sqrt{V_N})$   where
$V_N$ is the variance of the sum.  This estimate shows that the local limit theorem 
holds provided that characteristic function tends to zero at all non zero points. 
On the other hand, if the characteristic function does not tend to zero at some $\xi\neq 0$
then 
 we show that the observables $(f_j)$  are reducible to $\frac{2\pi}{\xi}\mathbb{Z}$
 valued random variables, concluding the proof of Theorem \ref{LLT1}, 
 which occupies 
 Section \ref{SSKeyPropZN}.

In Theorem \ref{LLT RED} we will show that in the reducible case an appropriate type of lattice LLT holds true. As noted before, for lattice-valued observables $(f_j)$ Theorem \ref{LLT RED} reduces to Theorem \ref{LLT Latt}, but the more general version takes into account also obstructions (b) and (c).
Our  proof of the generalized lattice LLT includes an additional step, which involves another application of the perturbative arguments 
 around non zero resonant points
(using ideas from \cite{HK, DH, DH2} and \cite{DS}).
Namely,
in the reducible case \eqref{Charr} fails since the characteristic function does not decay in small neighborhoods of the lattice points $\frac{2\pi}{h_0}\bbZ$, 
where $h_0$ is like in Theorem \ref{Reduce thm}. In this case each lattice point contributes a correction term and thus an appropriate  lattice LLT holds. The proof proceeds  by 
expanding $\bbE[e^{itS_n}]$ around the lattice points $\frac{2\pi}{h_0}\bbZ$. 
 Here we again rely on  the complex sequential Perron-Frobenius theorem, 
and we also use 
the ideas from \cite{DH SPA} among other ingredients.



\subsection{Plan of the paper}
The layout of the paper is the following. In Section \ref{Examples} we present several concrete examples of maps satisfying our assumptions.
 We begin with non-stationary subshifts of finite type, and then describe several examples of maps which can be modeled by such shifts, as well as special types of H\"older continuous functionals which have applications, for instance, to  products of (finite valued) random Markov dependent matrices and to random Lyapunov exponents. This class of examples complements the  class of smooth maps considered in \cite[Section 4]{DH2} for which our results also hold. 
Section~\ref{Hyper Sec} explains how to extend our results to two-sided non-stationary shifts and discuss applications to small sequential perturbations of a fixed Axiom A map. 
Section~\ref{Sec4}  provides  background needed for the proof of the local limit theorems (e.g. real and complex transfer operators, equivariant measures, characteristic functions and Lasota-Yorke inequalities). Using these tools we will prove Theorem \ref{Reduce thm}.

Section \ref{Sec 5} is devoted to reduction lemmas, which are  essential in the proof of the LLT in the irreducible case in Section \ref{SSKeyPropZN}. 
Section \ref{Red sec} analyzes the local distribution of $S_N$ in the reducible case.
Section \ref{Sec 9} discusses  two sided SFT.
It 
contains,  in particular, 
several generalizations of standard facts about subshifts of finte type to the non-stationary case, including the conditioning arguments that allow to reduce the LLT from invertible subshifts to non-invertable ones by conditioning on the past.
Finally, Section \ref{Irr Sec}  is devoted to 
checking irreducibility in specific examples.
In \S \ref{PathConn} we show that for  connected spaces, the sequential observables are always irriducible, and hence the non-lattice LLT holds. In \S \ref{SS2SideIrr} we  prove that  close  hyperbolic maps on the tori always satisfy a non-lattice LCLT.

\section{Examples and applications}\label{Examples}
In \cite[\S 4.3.1]{DH2}  a class of smooth expanding maps satisfying  our
assumptions  was described. Below we will focus on 
non autonomous subshifts of finite type and provide several additional
applications.

\subsection{Sequential topologically mixing subshifts of finite type (SFT) and their applications}\label{SFT sec}
Let $\mathcal A_j\!\!=\!\!\{1,\dots,d_j\}$ with $\DS \sup_j d_j<\infty$. Consider matrices $A^{(j)}$ of sizes $d_j\times d_{j+1}$ with $0$--$1$ entries. We suppose that there exists an $M\in\bbN$ such that for every $j$ 
 all entries of the matrix $A^{(j)}\cdot A^{(j+1)}\cdots A^{(j+M)}$ are positive.
Define 
$$
X_j=\left\{(x_{j,k})_{k=0}^\infty: \,x_{j,k}\in \cA_{j+k}, \;\; A^{(j+k)}_{x_{j,k}, x_{j,k+1}}=1\right\}.
$$
Let $T_j:X_j\to X_{j+1}$ be the left shift.  Then $X_j=T_{j-1}\circ\cdots\circ T_1\circ T_0 X_0$, and so, we can identify the $k$-th coordinate $x_{j,k}$ in $X_j$ with $x_{0,j+k}$, which from now on will just be denoted by $x_{j+k}$, and points in $X_j$ will be denoted by $(x_{j+k})_{k=0}^\infty$.

Define a metric $\mathsf{d}_j$ on $X_j$ by 
$$
\mathsf{d}_j(x,y)=2^{-\inf\{k: x_{j+k}\not= y_{j+k}\}}
$$
and we use the convention $2^{-\infty}=0$.
Then, with this metric,  the maps $T_j$ satisfy Assumptions \ref{AssPairing} and \ref{Ass n 0}. 
Thus, all of our results are true  starting with measures of the form $\ka_0=q_0d\mu_0$, where $\mu_0$ is a time zero Gibbs measure (see \S \ref{TPo} and \S 
\ref{ScSFT-Gibbs}
for the background on Gibbs measures).

Next, given a point $x=(x_{j+k})_{k\geq 0}\in X_j$ and  $0\leq r_1\leq r_2$ we denote  the corresponding cylinder by
$$
[x_{j+r_1},x_{j+r_1+1},...,x_{j+r_2}]=\{x'=(x'_{j+k})_{k\geq 0}\in X_j:\,x'_{j+k}=x_{j+k},\,\forall r_1\leq k\leq r_2\}.
$$

Related invertible sequential dynamical systems are two sided subshifts of finite type. Here we assume that the sequences $d_j$ and $A^{(j)}$, and so the  shift space $X_j$, are defined for $j\in\bbZ$ and not only for $j\geq 0$.
Set  
$$
\tilde X_0=\{(y_k)_{k=-\infty}^{\infty}:\,A_{y_k,y_{k+1}}^{(k)}=1, y_k\in\cA_k\}
\quad\text{and}\quad 
\tilde X_j=\{(y_{j+k})_{k=-\infty}^{\infty}: (y_k)_{k=-\infty}^{\infty}\in \tilde X_0\}.
$$
Define a metric on $\tilde X_j$ by setting 
$$
\mathsf{d}_j(x,y)=2^{-\inf\{|k|:\, x_{j+k}\not= y_{j+k}\}}.
$$
Let $\tilde T_j:\tilde X_j\to \tilde X_{j+1}$ be the left shift. Set 
$$
\tilde T_j^n=\tilde T_{j+n-1}\circ\cdots\circ\tilde T_{j+1}\circ\tilde T_{j}.
$$
Similarly to the one sided case, for every point $y=(y_{j+k})_{k\in\bbZ}\in \tilde X_j$ and all $r_1\leq r_2$ we denote 
$$
[y_{j+r_1},y_{j+r_1+1},...,y_{j+r_2}]=\{y'=(y'_{j+k})_{k\in\bbZ}\in Y_j:\,y'_{j+k}=y_{j+k},\,\forall r_1\leq k\leq r_2\}.
$$

The maps $\tilde T_j$ do not satisfy  Assumptions \ref{AssPairing} and \ref{Ass n 0} since they have both expanding and contracting directions.
Nevertheless,
in  Section  \ref{Sec 9}
we will explain how to prove all the results stated in the previous section for these maps with $\mu_0$ being a time zero
Gibbs measure on the two sided shift.

\subsection{Uniformly aperiodic Markov maps on the interval}\label{Markov maps}
In this section we consider maps $\tau_j:[0,1]\to[0,1]$ with the following properties.
Let $d\geq 2$ be an integer.  For each $j$, we assume that there is a collection of disjoint open sub intervals $\{I_{j,1},...,I_{j,d_j}\}$ of $[0,1]$, where $2\leq d_j\leq d$, such that the union of their closures covers $[0,1]$. 
Moreover, each set
$
\tau_j(I_{j,k}), k=1,2,...,d_j$ is a union of some of the intervals in the collection $\{I_{j+1,s}: 1\leq s\leq d_{j+1}\}$. 
We also suppose that the maps  $\tau_j|_{\bar I_{k,j}}$ are 
twice
differentiable and that there are constants $\gamma,b>1$ such that for all $j$ and $1\leq k\leq d_j$,
$$
\gamma\leq \inf_{x\in \bar I_{k,j}}|\tau_j'(x)|\leq \sup_{x\in \bar I_{k,j}}|\tau_j'(x)|\leq b.
$$
Here $\bar I_{k,j}$ is the closure of $I_{k,j}$.
The above lower bound means that the maps are uniformly expanding.

  In these circumstances, 
there exist constants $c>0, \eta\in(0,1)$ such that  for all $j,n$ and indexes $1\leq i_{j+\ell}\leq d_{j+\ell},\, \ell<n$ we have 
\begin{equation}\label{Diam}
\text{diam}\left(\bigcap_{\ell=0}^{n-1}\tau_j^{-\ell}(I_{j+\ell,i_{j+\ell}})\right)\leq c\eta^n
\end{equation}
where $\tau_j^{-\ell}A=(\tau_j^\ell)^{-1}A$ for every $j,\ell$ and a Borel set $A$. 

\begin{assumption}[Adler condition]
There is a constant $C>0$ such that 
$$
\sup_j\sup_{1\leq k\leq d_j}\sup_{x\in I_{j,k}}\frac{|\tau_j''(x)|}{(\tau_j'(x))^2}\leq C.
$$
\end{assumption}  
The following bounded distortion property is a standard 
 consequence of the Adler condition.
\begin{corollary}\label{Cor dis}
There exists a constant $C>0$  such that  for all $j,n$ and indexes $1\leq i_{j+\ell}\leq d_{j+\ell},\, \ell<n$ and all 
$\DS x,y\in \cC_{j,n}(\bar i)=\tau_j^n(\bigcap_{\ell=0}^{n-1}\tau_j^{-\ell}(I_{j+\ell, \,i_{j+\ell}}))$ we have 
$$
\left|\frac{|(\tau_j^n)'(x)|}{|(\tau_j^n)'(y)|}-1\right|\leq C|x-y|.
$$
\end{corollary}
This above distortion estimate implies Renyi's property: there exists $A>0$ such that
\begin{equation}\label{Renyi}
  \frac{1}{|(\tau_j^n)'(x)|}=A^{\pm 1}m(\cC_{j,n}(\bar i))  
\end{equation}
on, $\cC_{j,n}(\bar i)$,
where $m=$Lebesgue. In fact, all we need here is this distortion estimate, but we prefer to describe the setup with the more familiar Adler condition.

Our last assumption is as follows.
\begin{assumption}[Uniform mixing of the partition]
There is a constant $M\in\bbN$ such that for all $j$ and all $1\leq k\leq d_j$ and $1\leq \ell\leq d_{j+M}$ we have 
$$
I_{j,k}\cap \tau_j^{-M}(I_{j+M,\ell})\not=\emptyset.
$$
\end{assumption}



The sequential system $\tau_j$ can be lifted to a  sequential SFT, as explained in what follows.
Let $\cA_j=\{1,2,...,d_j\}$ and consider the $0-1$-valued $d_j\times d_{j+1}$ matrix $A^{(j)}$ such that $A^{(j)}_{k,\ell}=1$ iff $I_{j,k}\cap \tau_j^{-1}(I_{j+1,\ell})\not=\emptyset$ (namely, iff $I_{j+1,\ell}\subset \tau_j(I_{j,k})$). Let $T_j:X_j\to X_{j+1}$ denote the corresponding sequential SFT. Then the SFT is uniformly topologically mixing, the maps $\pi_j:X_j\to [0,1]$ given by 
$$
\pi_j(x_j,x_{j+1},...)=\bigcap_{\ell=0}^{\infty}\tau_j^{-\ell}(\bar I_{j+\ell,x_{j+\ell}})
$$
are  uniformly H\"older  and $\pi_{j+1}\circ T_j\!\!=\!\!\tau_j\circ \pi_j$. 
Consider the function $\phi_j:X_j\!\to\! \bbR$ given by 
\begin{equation}\label{Pot}
\phi_j(x)=-\ln |\tau_j'(\pi_j x)|.
\end{equation}
Then under the above assumptions  $\DS \sup_j\|\phi_j\|_\eta<\infty$. 
Let $(\mu_j)$ be a sequence of Gibbs measures  corresponding to the sequence of functions $(\phi_j)$. Then 
any limit theorem on the shift implies the same result for the system $(\tau_j)$ with respect to a measure which  is absolutely continuous with respect to $\te_0=(\pi_0)_*\mu_0$, with H\"older continuous density.

\begin{remark}
 Let $L_j$ be the transfer operator of $\tau_j$, 
 that is the operator mapping a function $g$ to the density of the measure $(\tau_j)_*(gd m)$, where $m=\text{Lebesgue}$. Then, denoting $\tau_{j,k}^{-1}=(\tau_j|_{I_{j,k}})^{-1}$ 
we have
$$
L_j g|_{I_{j+1,\ell}}=\sum_{k:\, I_{j+1,\ell}\subset \tau_j(I_{j,k})}\frac{g\circ \tau_{j,k}^{-1}}{\tau_j'\circ \tau_{j,k}^{-1}}.
$$
Let also $\cL_j$ be the operator acting on the shift given by 
$$
\cL_jg(x)=\sum_{y: T_jy=x}e^{\phi_j(y)}g(y).
$$
We define $L_j^n=L_{j+n-1}\circ\cdots\circ L_{j+1}\circ L_j$, and we define $\mathcal L_j^n$ similarly. Then
\begin{equation}\label{Relation}
   \cL_j^n(g\circ\pi_j)=(L_j^n g)\circ \pi_{j+n}.
\end{equation}
Using this relation it follows that the measure $\te_0=(\pi_0)_*\mu_0$ is absolutely continuous with respect to Lebesgue. In fact,   $(\pi_j)_*\mu_j$ is the asymptotically unique sequence of absolutely continuous measures $(\te_j)$ such that $(\tau_j)_*\te_j=\te_{j+1}$ (see \cite[Theorem 2.4(ii)]{DH2}). 
 In case $(\tau_j)$ is a two sided sequence (namely $\tau_j$ is defined for $j<0$) the shift extension is defined for all $j\in\bbZ$, as well. In these circumstances Gibbs measures are unique, and it  
 follows that the densities of the unique absolutely continuous measures $\te_j$ is $(\pi_j)_*\mu_j$.
  The idea is that (see \cite[Appendix A]{DH2} and \cite[Remarks 2.5 \& 2.6]{DH2}) both the 
 densities of the asymptotically unique and the unique (in the the two sided case) equivariant  measures $\te_j$ can be expressed by means of the operators 
 $L_j^n$ and that each Gibbs measure on the shift is constructed through a two sided extension, and it can be expressed by means of the transfer operators $\cL_j^n$.
\end{remark}


\subsection{Finite state elliptic Markov chains}
In \cite{DS} the LLT for uniformly bounded additive functionals of uniformly elliptic inhomogeneous Markov chains $(\xi_j)$ was obtained. However, these result do not apply to functions $f_j$ of the entire path $(\xi_j)$. In this section we derive the LLT for H\"older continuous functions of the entire path of finite state Markov chains. Additionally,  
we will give several examples of how such functions naturally arise.

For each $j$, let $\cA_j$ be a finite set of size $d_j$. Suppose $\DS \sup_jd_j<\infty$. Let $(\xi_j)$ be a Markov chain, such that $\xi_j$ takes values in $\cA_j$. To fix the notation, let us focus on one  sided Markov chains sequences $(\xi_j)_{j\geq 0}$. The main idea below will be a reduction to a one sided sequential SFT, and a simple modification of the argument will yield a reduction of the LLT's for two sided chains $(\xi_j)_{j\in\bbZ}$ to a two-sided seqeuntial SFT.
 
For $x\in \cA_j$ and $y\in\cA_{j+1}$ let
$$
p_j(x,y)=\bbP(\xi_{j+1}=y|\xi_j=x)
$$
and suppose that 
\begin{equation}\label{LowerBound}
 \inf_j\min\left\{p_j(x,y): p_{j}(x,y)>0, x\in \cA_j, y\in \cA_{j+1}\right\}>0.
\end{equation}
We also assume that  the chain is uniformly elliptic: that is, there exists $M\in \bbN$ and $\ve_0>0$ such that for all $j$ and all $x\in \cA_j$ and $y\in \cA_{j+M}$ we have 
\begin{equation}\label{UE}
   \bbP(\xi_{j+M}=y|\xi_j=x)\geq \ve_0. 
\end{equation}

Define a metric $\mathsf{d}_j$ on the infinite product 
$\DS \textbf{A}_j=\prod_{k=0}^\infty\cA_{j+k}$ by 
$$
\mathsf{d}_j(\bar x,\bar y)=2^{-\inf\{k\geq 0: x_{j+k}\not=y_{j+k}\}}
$$
where 
$\bar x=(x_j,x_{j+1},\dots)$, $\bar y=(y_j,y_{j+1},\dots)$, 
and we use the convention $2^{-\infty}=0$. 

Let $f_j:\textbf{A}_j\to\bbR$ be a sequence of H\"older continuous functions with respect to some given exponent $\al\in(0,1]$, whose H\"older norms are uniformly bounded in $j$. Consider the random variable $Y_j=f_j(\xi_j,\xi_{j+1},\dots)$.
Set 
$
\DS S_n=\sum_{j=0}^{n-1}Y_j.
$
As it will be explained below, as an application of our main results we get the following result.

\begin{theorem}\label{LCLT MC}
Either $(f_j)$ is reducible or $S_n$ obeys  non-lattice local limit theorem. In the reducible case $S_n$ obeys  generalized non-lattice LLT 
(Theorem \ref{LLT RED} from Section \ref{Red sec}).
\end{theorem}

 Theorem 
\ref{LCLT MC} relies on the two auxiliary facts presented below.

\begin{lemma}\label{Simple Lemma}
  Under \eqref{UE} we have 
\begin{equation}\label{LoweBound 2}
\ve_0'=:\inf_j\min_{x\in \cA_j}\bbP(\xi_j=x)>0. 
\end{equation}
\end{lemma}
\begin{proof}
For $j>M$ we have
$$
\bbP(\xi_j=y)=\sum_{x\in \cA_{j-M}}\bbP(\xi_{j-M}=x)\bbP(\xi_j=y|\xi_{j-M}=x)\geq \ve_0\sum_{x}\bbP(\xi_{j-M}=x)=\ve_0.
$$
For $j\leq M$ we use that the set $\{\bbP(\xi_k=x), k\leq M, x\in\cA_k\}$ contains at most $d_1d_2\cdots d_{M}$ values, which are all positive 
\end{proof}

Next, let $A^{(j)}$ be a $d_j\times d_{j+1}$ matrix with $0\!-\!1$ values such that its $(x,y)$ entry is $1$ if and only if $p_j(x,y)\!\!>\!\!0$. Consider the   sequential one sided subshift of finite type generated by the sets $\cA_j$ and the matrices $A^{(j)}$.
By \eqref{UE},  this sequential SFT is aperiodic.

Next, let $\mu_j$ be the measure on the infinite product $\textbf{A}_j=\cA_j\times \cA_{j+1}\cdots$ induced by the process $(\xi_j,\xi_{j+1},...)$. 
 Then $\mu_j$ is supported on $X_j$.

\begin{proposition}   
\label{PrMarkovGibbs}
The sequence $(\mu_j)_{j\geq0}$ is a sequence of  Gibbs measures on $X_j$
corresponding to the potential 
$\DS
\phi_j(x_j,x_{j+1},...)=\ln p_j(x_j,x_{j+1}).
$
\end{proposition}
\begin{proof}
First, by \eqref{LowerBound} the functions $\phi_j$ are uniformly bounded. Since they depend only on the first two coordinates, they are also uniformly H\"older continuous. Hence, such a Gibbs measure indeed exists. Now, to see why $\mu_j$ is the desired Gibbs measure,  by the definition of Gibbs measures
 (see  \S \ref{ScSFT-Gibbs}), we need to show that  $(T_j)_*\mu_j=\mu_{j+1}$
    and  the Gibbs property holds, that is, there is a constant $C\geq 1$ such that 
    for all $j$,\;\; $(x_j,x_{j+1},...)\!\in\! X_j$ and $n\geq0$,
            $\DS
C^{-1}e^{S_{j,n}\phi(x)}\leq\mu_j([x_j,x_{j+1},...,x_{j+n-1}])\leq Ce^{S_{j,n}\phi(x)}
    $
    where $\DS S_{j,n}\phi=\sum_{k=0}^{n-1}\phi_{j+k}\circ T_j^{k}$.
    
    To see why $(T_j)_*\mu_j=\mu_{j+1}$ we observe that both measure coincide with the law of $(\xi_{j+1},\xi_{j+2},...)$.  To prove the Gibbs property,  we have  
    $$
\mu_j([x_j,x_{j+1},...,x_{j+n-1}])=\bbP(\xi_j=x_j)
\prod_{k=0}^{n-2}p_{j+k}(x_{j+k},x_{j+k+1})
$$
$$
=\bbP(\xi_j=x_j)(p_{j+n-1}(x_{j+n-1},x_{j+n}))^{-1}e^{S_{j,n}\phi(x)}
    $$
where for $n=1$ the above product of $n-1$ terms should be interpreted as $1$.
Combining this with Lemma \ref{Simple Lemma} and taking into account that  
  $p_j(x_j,x_{j+1})\in [\ve_0,1]$ we see that
    $$
\ve_0'e^{S_{j,n}\phi(x)}\leq \mu_j([x_j,x_{j+1},...,x_{j+n-1}])\leq 
\ve_0^{-1}e^{S_{j,n}\phi(x)}
    $$
    and the lemma follows.
\end{proof}

In view of the  Proposition \ref{PrMarkovGibbs},
Theorem \ref{LCLT MC} follows  from Theorem \ref{LCLT two sided} below.

\begin{remark}
In the case of two sided Markov chains $(\xi_j)_{j\in\bbZ}$ the same argument shows that the distribution of the entire path  $(\xi_j)_{j\in\bbZ}$ coincides with the (now unique, see  Theorem \ref{ThSFT-Gibbs} in \S \ref{ScSFT-Gibbs}) Gibbs measure at time $0$ of the two sided sequence of two sided shifts. This point emphasizes the reason there are no unique Gibbs measures in the one sided case: we can always change the initial distribution (i.e. the law of $\xi_0$) without changing its support. This results with a wide range of different Gibbs measures. On the other hand, for two sided chains there is no initial condition, and so the resulting Gibbs measures are unique.
\end{remark}

\newpage

\subsubsection*{\textbf{Examples of H\"older functions}}

\subsubsection{Recursive sequences and series}
\,
Let us suppose that $(\xi_j)$ is  a finite state uniformly elliptic Markov chain with values in $\bbR$ such that $\DS \sup_j\|\xi_j\|_{L^\infty}<\infty$. 

We begin with a specific example of a linear statistic.
Define recursively 
$$
X_{j+1}=aX_j+\xi_{j+1}
$$
where $a\in(0,1)$. Then, for all $j\geq 0$ and $n$ we have
$\DS
X_{j}=a^{j+n}X_{-j-n}+\sum_{k=0}^{j+n}a^k\xi_{j-k}.
$
We thus see that the only bounded solution to this recurrence relation is 
$$
X_j=f_j(\xi_j,\xi_{j+1},...)=\sum_{k=0}^\infty a^k \xi_{j-k}.
$$
Notice that the functions $f_j$ are H\"older continuous with respect to the metric introduced earlier since $\DS \sup_{j}\|\xi_{j}\|_{L^\infty}<\infty$.

More generally, if $\xi_j$ takes values in $\{1,2,...,d_j\}$ for $d_j\in\bbN$ and 
$\DS \sum_{k=0}^\infty a_k<\infty$ is a series with exponential tails
then we can consider
$$
f_j(\xi_j, \xi_{j+1},...)=\sum_{k=0}^\infty a_k\xi_{j-k}
$$
or for two sided exponentially decaying
sequences  $(a_k)$,
$$
f_j(\dots ,\xi_{j-1},\xi_j, \xi_{j+1},...)=\sum_{k=-\infty}^\infty a_k\xi_{j-k}.
$$
 We note that similar examples appear in \cite{PSZ24}, however, our set up is more flexible  since we do not require stationarity and we can also replace
linear statistics $\xi_k$ by nonlinear 
smooth
functions $g_k(\xi_{k-r}, \dots, \xi_k, \dots \xi_{k+r}).$

\subsubsection{Products of random positive matrices}\label{McApp}
 Fix some integer $d>1$
and
let $(\xi_j)$ be a sequence of random  $d\times d$ matrices with positive entries, which are uniformly bounded and bounded away from the origin.
Then the arguments in \cite[Ch. 4]{HK} yield that
for every realization of $(\xi_j)$ the sequential Perron-Frobenius  theorem holds.
Namely, denote $\Xi_{j,n}=\xi_{j+n-1}\cdots\xi_{j+1}\cdot \xi_j$. Then there are  two uniformly bounded sequences of random vectors $\nu_j$ and $h_j$  and a sequence of strictly positive random variables (all three depend on the entire orbit $(\xi_k)$) such that, a.s.
\begin{equation}\label{Mat RPF}
\left\|\frac{\Xi_{j,n}}{\la_{j,n}}-\nu_j\otimes h_{j+n}\right\|\leq C\del^n
\end{equation}
for some constants $C>0$ and $\del\in(0,1)$. Moreover 
$\nu_j\cdot h_j=\nu_j\cdot u_j=1$. Furthermore, 
$\DS
\xi_j h_j=\la_j h_{j+1}
$
and 
$\DS
\xi_j^*\nu_{j+1}=\la_j\nu_j.
$
By \cite[Lemma A.2]{DG18},  $\la_j$ are uniformly H\"older  functions of the path $(\xi_k)$ with respect to the distance  $\mathsf{d}_j$ defined in the previous section.



Next, since $\la_j$ is uniformly bounded and bounded away from the origin we get that the functions $\Pi_j=\ln\la_j$ are also uniformly H\"older continuous. Thus, all the results stated in this paper hold true for 
$$
\ln\|\Xi_{0,n}\|\,\text{ and }\,\ln([\Xi_{0,n}]_{k,s}), \,\,1\leq k,s\leq d.
$$
Indeed, by \eqref{Mat RPF}, 
 studying those expressions  reduces to proving the corresponding results for the Birkhoff sums 
$\DS \sum_{j=1}^n \Pi_j$, 
which is exactly the type of sums studied in this paper.

\subsubsection{Lyapunov exponent of nonstationary sequences of random hyperbolic matrices}

Let $d>1$ and let $A$ be a hyperbolic matrix with distinct eigenvalues $\la_1,...,\la_d$. Suppose that for some $k<d$ we have $\la_1<\la_2<...<\la_k<1<\la_{k+1}<...<\la_{d}$. Let $h_j$ be the corresponding eigenvalues.

Now, let $(A_j)$ be a sequence of matrices such that $\DS \sup_j \|A_j-A\|\leq \ve$. Then, if $\ve$ is small enough there are numbers $\la_{j,1}<\la_{j,2}<...<\la_{j,k}<1<\la_{j, k+1}<...<\la_{j,d}$
and vectors $h_{j,i}$ such that 
$$
A_{j}h_{j,i}=\la_{j,i}h_{j+1,i}.
$$
Moreover, $\DS \sup_j|\la_{j,i}-\la_i|$ and 
$\DS \sup_j\|h_{j,i}-h_{i}\|$ converge to $0$ as $\ve\to 0$.

Now, the sequence $(A_j)$ is uniformly hyperbolic and the sequences 
$(\la_{1,j})_j,...,(\la_{d,j})_j$ can be viewed as its sequential Lyapunov exponents. Moreover,  
the one dimensional  spaces $H_{i,j}=\text{span}\{h_{i,j}\}$ can be viewed as its sequential Lyapunov spaces.
Next,  $\la_{i,j}$ and $h_{i,j}$ 
can be approximated exponentially fast in $n$ by functions of 
$$(A_{j-n},A_{j-n+1},...,A_j,A_{j+1},...,A_{j+n}),$$ uniformly in $j$. 

Finally, let us consider a uniformly elliptic Markov chain $(A_k)$ such that each $A_k$ is a  perturbation of $A$ and it can take at most $L$ values, for some fixed $L$. Then the random variables $\la_{i,j}$ and $h_{i,j}$ are H\"older continuous functions of the whole orbit of the chain $(A_k)$ (uniformly in $k$).

\subsection{Expanding maps on $\bbT ^d$}
Suppose that there is a number $\gamma>1$ such that  for  each  $j$ there is a partition of $[0,1)^d$ into rectangles $R_{j,k}, 1\leq k\leq d_j$ such that for each $k$ the map $T_{j}|R_{j,k}\to [0,1)^d$ is contracting by at least $\gamma$, and has a full image $[0,1)^d$. We also assume that $\sup_j d_j<\infty$. Then Assumption \ref{AssPairing} holds with $\xi=1$ (since we can pair two arbitrary points). Therefore, Assumption \ref{Ass n 0} trivially holds with $n_0=1$.

Thus applying Theorem \ref{LLT2} we obtain
\begin{corollary}
For expanding maps of $\T^d$ if $\sig_n\to\infty$ then $(f_j)$ is irreducible and 
$S_nf$ obeys the non-lattice LLT. 
\end{corollary}

\section{Two sided shifts} 
\label{Hyper Sec}
\subsection{The result}
Let $\psi_j:\tilde X_j\to\bbR$ be a sequence of functions on the two sided shift spaces 
$\tilde X_j$ such that 
$\DS \sup_j\|\psi_j\|_\al<\infty$ for some $\al\in(0,1]$. Let $\gamma_0$ be the Gibbs measure at time $0$ corresponding to the sequence $(\psi_j)$ (see Section \ref{Sec 9}). Let $\ka_0$ be a probability measure on $\tilde X_0$ with H\"older continuous density with respect to the measure $\gamma_0$. Let 
$$
S_n=\sum_{j=0}^{n-1}\psi_j\circ \tilde T_0^j.
$$
\begin{theorem}\label{LCLT two sided}
If $(\psi_j)$ is irreducible then
$S_n$ obeys the non-lattice LLT when considered as a sequence of random variables on the probability space $(\tilde X_0,\text{Borel},\ka_0)$. Moreover, the first order expansions 
 \eqref{EdgeFirst} hold.
If  $(\psi_j)$ is reducible then the non-lattice LLT  of Section \ref{Red sec} holds.
\end{theorem}

\subsection{Applications to small sequential perturbations of a single hyperbolic map}
Let $T$ be a diffeomorphism of a compact connected 
smooth Riemannian manifold $M$. 
We assume that $T$ preserves a locally maximal basic hyperbolic set $\Lambda$
such that $T$ is topologically mixing on $\Lambda.$
Next, consider a sequence of maps $T_j$ such that $d_{C^1}(T_j,T)\leq \ve$ for some $\ve$ small enough.
Then (see e.g. \cite[Appendix C]{DH2})
there are sets $\Lambda_j$ and H\"older homeomorphisms 
$h_j:\! \Lambda\!\to\!\Lambda_j$ which conjugate $(T_j)$ to $T$, that is
  \begin{equation}\label{Cong}
 T_j\Lambda_j=\Lambda_{j+1}\, \text{ and }\, T_j\circ h_j=h_{j+1}\circ T.
 \end{equation}

\noindent
Let $\mu_0$ be a time zero Gibbs measure for the sequence $(T_j)$ corresponding to a sequence of potentials $(\phi_j)$,
see \cite[Appendix C]{DH2}.
Consider a sequence of functions $f_j\!:\!M\!\to\!\bbR$ such that $\DS \sup_j\|f_j\|_\al<\infty$ for some $\al\!\!\in\!\!(0,1]$. 
Let $\DS S_n=S_nf=\sum_{j=0}^{n-1}f_j\circ T_0^j$ and consider $S_n$ as random variables on the space $(M,\text{Borel},\mu_0)$.

\begin{theorem}
\label{ThHyp}
(a) If $\ve$ is small enough then either \eqref{EqReduction} holds almost surely with some $h>0$ and a uniformly bounded sequence $H_j$  such that $\DS (S_nH)_{n=1}^\infty$ is tight, or the non-lattice
LLT and the first order expansions hold.

(b) If $T$ is an Anosov map of a torus, 
and $\text{Var}(S_n)\to\infty$ 
 then the non-lattice LLT and the first order expansions hold.
\end{theorem}

Part (a) follows directly from Theorem \ref{LLT1}.
Indeed due to \eqref{Cong}, we may assume that $T_j=T$ for all $j.$ In this case $\Lambda$
admits a Markov partition $\Pi=\{\Pi_j\}_{j=1}^{m}$ which allows to construct a
Markov coding map
$\pi: \Sigma\to \Lambda$ 
where 
$$ \Sigma=\{\omega: \in \{1,\dots ,m\}^\integers: T \Pi_{\omega_n}\cap \Pi_{\omega_{n+1}}
\neq\emptyset\}\quad\text{and}\quad 
\pi(\omega)=\bigcap_{n\in \integers} T^{-n} \Pi_{\omega_n} .$$
Also by construction $\mu_n(B)=\bar\mu_n(\pi^{-1} B)$ where 
$\bar\mu_n$ are Gibbs measures for potentials $\bar\phi_n=\phi_n\circ \pi.$
Now Theorem \ref{ThHyp}(a) follows from 

\begin{proposition}
\label{PrOneToOne}
$\Pi$ can be constructed in such a way that 
$\pi$ is one-to-one $\bar\mu_0$ almost everywhere.
\end{proposition}

Namely we can use the Markov partitions constructed by Bowen. In the case 
$\phi_n$ do
not depend on $n$, Proposition \ref{PrOneToOne} can be found in \cite[page 64]{Bowen}. 
The result in the non stationary case can be obtained using similar ideas and we provide
it below for completeness.

\begin{proof}
Let $\Pi$ be a Markov partition with sufficiently small diameter that produces the coding, and let $\Pi^s$ and $\Pi^{u}$ be the boundaries in the  stable and unstable directions, respectively. Then, since the map $\pi$ is one to one outside the set 
$\DS \mathcal R=\bigcup_{k\in\mathbb Z}T^k(\Pi^s\cup \Pi^u)$, it is enough to show that 
for all $j$ we have $\mu_j(\Pi^u)=\mu_j(\Pi^s)=0$. By replacing $T$ with $T^{-1}$ it is enough to show that $\mu_j(\Pi^u)=0$.

 We note that in Bowen's constructions the rectangles are closures of their interiors (in the induced topology). Take a point $x$ in the interior of one of the rectangles
and let $\omega$ be a point with $\pi(\omega)=x.$ 
Now take a cylinder $C$ containing $\omega$ such that the diameter of 
$\pi(C)$ is so small that $\pi(C)\cap \Pi^u=\emptyset.$
Since $\Pi^u$ is backward invariant we also have $\pi(C) \cap T^{-k} \Pi^u=\emptyset$
for all $k>0.$


By Lemma \ref{balls}  from \S \ref{SSMnN}
there is a constant $c>0$ such that $\DS \inf_j\mu_j(C)\geq c$. 
Thus,   for every measurable set $W$  and every $n$ we have 
$$
\lim_{k\to\infty}|\mu_{n-k}(T^{-k}W\cap C)-\mu_{n}(W)\mu_{n-k}(C)|=0.
$$
Indeed, 
we can  approximate $W$ by images of cylinders and
use the uniform mixing for H\"older functions on the one sided shift. However, 
 as noted above
$\mu_{n-k}(T^{-k}\Pi^u \cap C)\!\!=\!\!0$. So
$\DS
\lim_{k\to\infty}|\mu_{n}(\Pi^u)\mu_{n-k}(C)|\!\!=\!\!0,
$
but since $\mu_{n-k}(C)\!\geq\! c\!\!>\!\!0$ we conclude that $\mu_{n}(\Pi^u)=0$. 
\end{proof}

 Part (b) of Theorem \ref{ThHyp} is proven in \S \ref{SS2SideIrr}.

\section{Background}\label{Sec4}

\subsection{Transfer operators and Gibbs measures}\label{TPo}
We recall the construction of the classes of Gibbs measures $\mu_j$ with respect to which our theorems hold. 

Suppose first that $X_j$ is equipped with a Borel probability measure $m_j$ such that $(T_j)_*m_j\ll m_{j+1}$. Moreover, we assume that the functions 
$\DS \phi_{j}=-\ln\left(\frac{d (T_j)_*m_j}{dm_{j+1}}\right)$ satisfy 
$\DS \sup_j\|\phi_j\|_\beta\!\!<\!\!\infty$ for some H\"older exponent $\beta$. 
Applying \cite[Theorem 2.4]{DH2} we see that there is a sequence of  H\"older continuous positive functions $h_j:X_j\!\!\to\!\! \bbR$ with exponent $\beta$ such that 
the sequence of measures given by $d\mu_j\!\!=\!\! h_j d m_j$ is  the asymptotically unique sequence of absolutely continuous measures such that $(T_j)_*\mu_j=\mu_{j+1}$ (i.e. it is equivariant). If the sequence $(T_j)$ is  two sided 
(that is $T_j$ is defined for all $j\in\bbZ$) then this sequence is unique and not only asymptotically unique (see \cite[Remarks 2.5 \& 2.6]{DH2}).

When there are no underlying reference measure $m_j$ we need first to construct such measures. For this reason we need to work with two sided sequences of maps $T_j:X_j\to X_{j+1}, j\in\bbZ$ (see Footnote 1). On the other hand, even if reference measures exist one might be interested in proving limit theorems for singular  measures
(e.g. measures of maximal entropy in the autonomous case). This is related to the theory of Gibbs states, and in what follows we will give  a quick remainder of the construction of Gibbs measures in  our setting.
Take a sequence $\phi_j:X_j\to\bbR$ of H\"older continuous functions with exponent $\beta$ such that $\DS \sup_j\|\phi_j\|_\beta<\infty$. 
 Let the operator $L_j$ map a function $h:X_j\to\bbR$ to a function $L_jh:X_{j+1}\to\bbR$ given by the formula 
$$
L_jg(x)=\sum_{y:\, T_jy=x}e^{\phi_j(y)}g(y).
$$
Then as proven in \cite[Theorem 3.3]{Nonlin} (see also \cite[\S 5.2]{DH2}) there is a sequence of functions $h_j:X_j\to \bbR$ such that $\DS \inf_j\min_{x\in X_j} h_j>0$ and $\DS \sup_j\|h_j\|_{\beta}<\infty$, a sequence of probability  measures $\nu_j$ on $X_j$ such that $\nu_j(h_j)=1$ and a sequence of positive numbers $\la_j$ such that  $\DS 0<\inf_j\la_j\leq \sup_j\la_j<\infty$ and the following holds: 
$$
L_j h_j=\la_j h_{j+1}, \quad L_j^*\nu_{j+1}=\la_j \nu_j.
$$
Moreover, there are $C_0>0$ and $\del\in (0,1)$ such that for all $n$ and $j$ and all H\"older continuous functions $g$ with exponent $\beta$,
\begin{equation}\label{ExpConvSTF0}
    \|(\la_{j,n})^{-1}L_j^n g-\nu_j(g)h_{j+n}\|_{\be}\leq C_0\|g\|_{\be} \del^n.
\end{equation}
 Here
$$
L_{j}^n=L_{j+n-1}\circ\cdots\circ L_{j+1}\circ L_j, \la_{j,n}=\la_{j+n-1}\cdots \la_{j+1}\la_j.
$$
Then the  sequential Gibbs measures $(\mu_j)$ corresponding to the sequence of potentials $(\phi_j)_{j\in\bbZ}$ are given by $\mu_j=h_j d\nu_j$. In the case when $T_j$ are subshifts of finite type $\mu_j$ is the unique sequence of measures satisfying the Gibbs property (see Section \ref{Sec 9}). 
Let us define
$$
 \cL_j(h)=\cL_j(h\cdot h_j)/\la_jh_{j+1}=\sum_{y: T_jy=x}e^{g_j(y)}h(y)
 $$
 where  $g_j=\phi_j+\ln h_j-\ln (h_j\circ T_j)-\ln \la_j$. Note that
 $\DS \sup_j \|g_j\|_\beta<\infty$, 
 $\cL_j\textbf{1}=\textbf{1}$, where $\textbf{1}$ denotes the function taking the constant value $1$ (regardless of its domain),  $\cL_j^*\mu_{j+1}=\mu_j$ and that  the following duality relation holds
\begin{equation}\label{dual}
\int_{X_j}(f\circ T_j)\cdot gd\mu_j=\int_{X_{j+1}}f\cdot (\cL_j g) d\mu_{j+1}  
 \end{equation}
 for all bounded measurable functions $f$ and $g$. In fact,  when 
 $(T_j)_*m_j\ll m_{j+1}$, taking the functions 
$\DS \phi_{j}=-\ln\left(\frac{d (T_j)_*m_j}{dm_{j+1}}\right)$ we get that $\nu_j=m_j$ and that $\mu_j$ is the unique sequence of absolutely continuous equivariant measures discussed above (in this case $\la_j=1$ for all $j$). In the one sided non-singular case the proof of \cite[Theorem 2.4]{DH2} was based on proving the above results only with $j\geq 0$, and in that case $\la_j=1$ and $(\mu_j)_{j\geq 0}$ is the asymptotically unique sequence of absolutely continuous equivariant measures. Thus \eqref{ExpConvSTF0} and \eqref{dual} hold in all cases.

\subsection{Maps and Norms} 
\label{SSMnN}
We record some consequences of Assumptions \ref{AssPairing}  and \ref{Ass n 0}.

 \begin{lemma}\label{Dcover}
 \label{InvBranch}
 (cf. \cite[Lemma 2.1]{MSU})
 For all $j$ and $n$ and every $y\in X_j$ there is a function $Z_{j,y,n}:B_{j+n}(T_j^ny,\xi)\to B_{j}(y,\xi \gamma^{-n})$ such that: 
\vskip0.2cm 
(i)
$\DS
\mathsf{d}_{j+k}(T_j^k(Z_{j,y,n}x), T_j^k y)<\xi,\,\forall\,\,0\leq k<n,\,\,x\in B_{j+n}(T_j^ny,\xi);
$
\vskip0.2cm
(ii)
$\DS
T_j^n\circ Z_{j,y,n}=id;
$
\vskip0.2cm
(iii) If $\mathsf{d}_{j+n}(x,x')<\xi$ then 
$\DS
\mathsf{d}_j(Z_{j,y,n}x, Z_{j,y,n}x')\leq \gamma^{-n} d_{j+n}(x,x')
$
(and so the  Lipschitz constant of
$Z_{j,y,n}$ does not exceed $C\gamma^{-n}$ for some constant $C>0$);
\vskip0.2cm
(iv)
$
Z_{j,y,n+m}=Z_{j,y,n}\circ Z_{j+n, T_j^n y,m}.
$
\end{lemma}
\begin{proof}
Define $\cZ_{j,y,1}$ as follows. Let $x=T_j y$. Then there is an index $i$ such that $y_i(x)=y$. Set $\cZ_{j,y,1}=y_i$. Then properties (i)-(iv) hold with 

 $\DS
Z_{j,y,n}=\cZ_{j+n-1,T_j^n y, 1}\circ\cdots\circ \cZ_{j+2, T_j^2 y, 1}\circ \cZ_{j+1, T_j y, 1}\circ \cZ_{j,y,1}.
 $
\end{proof}

We need the following result (c.f. \cite[Lemma 2.1]{MSU}).
 \begin{lemma}\label{Dcover}
 Under Assumption \ref{Ass n 0}(i) we have the following. For every $0<r<\xi$ set $m_r=n_0+n_r$, $n_r=\left[\frac{\ln \xi-\ln r}{\ln \gamma}\right]$.
Then for every $j$ and $y\in X_j$ we have 
\begin{equation}\label{Cover r}
T_j^{m_r}\left(B_j(y,r)\right)=X_{j+m_r}.
\end{equation}
 \end{lemma}

\begin{proof}
Let $Z_{j,y,n}$ be the functions from Lemma \ref{InvBranch}. 
By Lemma \ref{InvBranch}(i) 
$$
Z_{j,y,n}(B_{j+n}(T_j^ny,\xi))\subset B_{j}(y,\xi \gamma^{-n}).
$$
Hence Lemma \ref{InvBranch}(ii) gives
$\DS
B_{j+n}(T_j^ny,\xi)=T_{j}^n\circ Z_{j,y,n}\left(B_{j+n}(T_j^ny,\xi)\right)\subset T_{j}^n\left(B_{j}(y,\xi \gamma^{-n})\right).
$
Now let $n=n_r$ be the smallest positive integer so that $\gamma^{-n_r}\xi\leq r$. Then 
$$
B_{j+n_r}(T_j^{n_r}y,\xi)\subset B_j(y,r).
$$
Thus by Assumption \ref{Ass n 0} (i),
$$
X_{j+n_0+n_r}=T_{j+n_r}^{n_0}\left(B_{j+n_r}(T_j^{n_r}y,\xi)\right)\subset  T_{j+n_r}^{n_0}\circ T_{j}^n \left(B_{j}(y,\xi \gamma^{-n})\right)=T_{j}^{n_0+n_r} \left(B_{j}(y,\xi \gamma^{-n})\right)
$$
proving the desired result.
\end{proof}

\begin{lemma}\label{al be}
Let $u_n:X_n\to\bbR$ be a sequence of functions so that 
$$
\lim_{n\to\infty}\|u_n\|_{L^1(\mu_n)}=0\,\,
\text{ and }\,\,\|u\|_\beta=\sup_n\|u_n\|_\beta<\infty.
$$
Then 
$\DS
\lim_{n\to\infty}\|u_n\|_\al=0$
for all $\al<\beta$.
\end{lemma}

In order to prove the lemma we need the following result,    whose proof proceeds exactly like the proof of \cite[Lemma 5.10.3]{HK}.
\begin{lemma}\label{balls}
For every $r>0$ there exists $\eta_r>0$ such that for every $j$ and all $x\in X_j$ we have 
$$
\mu_j(B_j(x,r))\geq \eta_r
$$
where $B_j(x,r)$ is the open ball of radius $r$ around $x$ in $X_j$. 
\end{lemma}
\noindent
Relying on the above result the proof is elementary, and it is included for 
completeness.
\begin{proof}[Proof of Lemma \ref{al be}]
We first show that if $\|u_n\|_\infty\to 0$ then $G_{n,\al}(u_n)\to 0$.
Let $\ve>0$.   Observe that  
   $$
|u_n(x)-u_n(y)|\leq \|u\|_\beta (\mathsf{d}_n(x,y))^{\al}(d_n(x,y))^{\beta-\al}.
   $$
   Let $\del=\left(\ve/\|u\|_\beta\right)^{\frac1{\beta-\al}}$.
   Thus, if $d_n(x,y)\leq \del$ we have 
   $$
|u_n(x)-u_n(y)|\leq\|u_n\|_\beta(\mathsf{d}(x,y))^\beta \leq \ve (\mathsf{d}_n(x,y))^{\al}.
   $$
   On the other hand, if $\mathsf{d}_n(x,y)>\del$ then 
   $$
|u_n(x)-u_n(y)|\leq 2\|u_n\|_\infty\leq 2\|u_n\|_\infty (\mathsf{d}_n(x,y))^{\al}\del^{-\al}.
   $$
   Hence, if $n$ is large enough to insure that $2\|u_n\|_\infty\leq \ve \del^\al$, then  the estimate
   $$
|u_n(x)-u_n(y)|\leq \ve(\mathsf{d}_n(x,y))^\al
   $$
   holds true  also when $\mathsf{d}_n(x,y)\geq \del$.
   We conclude that 
   $\DS
\lim_{n\to\infty}G_{n,\al}(u_n)=0.
   $

  In order to complete the proof of the lemma it suffices to show that if 
  $\DS  \sup_n \|u_n\|_\be<\infty$ and $\|u_n\|_{L^1(\mu_n)}\to 0$ then $\|u_n\|_\infty\to 0$. Fix $\ve>0$ and $x\in X_n$. Let $B_n(x,\ve)$ be the ball of radius $\ve$ around $x$ in $X_n$.
  By Lemma \ref{balls}  there is a constant $c_\ve>0$ which depends only on $\ve$ such that 
  $\DS
\mu_n(B_n(x,\ve))\geq \eta_\ve.
  $
Since $\|u\|_\be<\infty$ we have
\begin{equation}
\label{AvOfU}
\left|u_n(x)-\frac{1}{\mu_n(B_n(x,\ve))}  \int_{B_n(x, \ve)}
u_n(y)d\mu_n(y)\right|\leq \|u_n\|_\beta \ve^\be. 
\end{equation}
Since $\|u_n\|_{L^1(\mu_n)}\to 0$ and $\DS \inf_{x,n}\mu_n(B_n(x,\ve))>0$, for a fixed $\ve$, the second term 
in the  LHS of \eqref{AvOfU} converges to $0$ as $n\to\infty$.
Letting $n\to\infty$ we see that 
$\DS 
\limsup_{n\to\infty}\|u_n\|_\infty\leq \ve^\be.
$
Since $\ve$ is arbitrary we conclude that $\|u_n\|_\infty\to 0$, 
   and the proof of the lemma is complete.
\end{proof}

\subsection{Lasota Yorke inequalities}\label{SecPrep}
In this section we will make some preparations 
for the proof of Theorem \ref{LLT1}. Denote by $\cL_{j,t}^k$ the operators defined by 
$$
\cL_{j,t}^k h=\cL_{j}^k(e^{itS_{j,k}f}h).
$$
Note that $\cL_{j,t}^k=\cL_{j+k-1,t}\circ\cdots\circ\cL_{j,t}$, where 
$\cL_{\ell,t}:=\cL_{\ell,t}^1$.
Fix $\alpha\in (0,1].$ Let $G(h)=G_\alpha(h)$ be the Holder constant of $h.$ 
The following result was essentially obtained in \cite[Lemma 5.6.1]{HK}. 

\begin{lemma}\label{ll1}
 Given $T>0$
there are $C_1>0$ and $\theta_1\in(0,1)$ such that for $|t|\leq T$ 
$$ G_\al(\cL_{n,t}^k h)\leq C_1 \left[\|h\|_\infty+\theta_1^k G_\alpha(h)\right]. $$
\end{lemma}

Let $\|h\|_{\alpha, T}=\max\left(\|h\|_\infty, \frac{G_\alpha(h)}{2 C_1}\right).$
We shall abbreviate $\|\cdot\|_{\alpha, T}=\|\cdot\|_*.$

We will need the following result.

\begin{lemma}\label{ll2}
(a) For all $j\in\bbZ$ and $t\in\bbR$ we have $\DS \|\cL_{j, t}^k h \|_\infty\leq \| \cL_{j}^k |h|\|_\infty. $
\vskip0.1cm
(b) $\forall \epsilon\in(0,\frac12)$ $\exists k_1=k_1(\eps)$ and $\eta(\eps)>0$ such that for  all $j$ we have the following: if $\|h\|_*\leq 1$ and 
$|h(x)|\leq 1-\eps$  for some $x\in X_j$ then $\|\cL_{j}^{k_1} h\|_\infty \leq 1-\eta(\eps). $
\end{lemma}

\begin{proof}
  (a)   We have 
  $\DS
|\cL_{j,t}^kh|=|\cL_j^k(e^{it S_{j,k}f}h)|\leq \cL_j^k(|h|). 
  $
\vskip0.2cm
  (b)
  Since $\cL_{j}^k\textbf{1}\!\!=\!\!\textbf{1}$ (for all $j\!\!\in\!\!\bbZ$ and $k\!\!\in\!\!\bbN$), 
 for every  $j$ and $k\!\in\!\bbN$, a function $h:X_{j+k}\to\bbR$ with $\|h\|_*\leq 1$ and points $x\in X_{j+k}$ and $y\in X_j$ such that $T_j^{k} y\!\!=\!\!x$ we have 
  \begin{equation}\label{Up}
   |\cL_j^kh(x)|\leq 1-e^{S_{j,k}g(y)}+|h(y)|e^{S_{j,k}g(y)}.
  \end{equation}

  Now, suppose that for some $y_0\in X_{j}$ we have $|h(y_0)|\leq 1-\eps$. Let $k\in\bbN$ be large enough (in a way that will be specified later). Fix some $x\in X_{j+k}$.
  By Lemma \ref{Dcover}, for every $r>0$ there is a point $y_r\in X_j$ such that $\mathsf{d}_j(y_0,y_r)<r$ and $T_j^{m_r}y_r=x$ (where $m_r$ was defined in Lemma \ref{Dcover}). Then, as $\|h\|_*\leq 1$ we have 
  $$
|h(y_r)-h(y_0)|\leq 2C_1\mathsf{d}_j^\al(y_r,y_0)\leq 2C_1r^\al.
  $$
  Now, let us take $k=m_r$ for some $r$ which will be determined soon.
 By \eqref{Up} applied with $y=y_r$ we see that 
  $$
  |\cL_j^{k}(x)|\leq (1-e^{S_{j,k}g(y_r)})+|h(y_r)|\leq 1-e^{S_{j,k}(y_r)}+(|h(y_0)|+2C_1r^\al)e^{S_{j,k}g(y_r)}$$$$\leq 
  1-e^{S_{j,k}g(y_r)}+(1-\eps+2C_1r^\al)e^{S_{j,k}g(y_r)}=1-e^{S_{j,k}(y_r)}(\eps-2C_1r^\al).
  $$
  Next, let us take the largest $r=r(\eps)$ such that $\eps-2C_1r^\al\leq \eps/2$. Let 
  $\DS c=\sup_j \|g_j\|_\infty$. Then $e^{S_{j,k}g(y_r)}\geq e^{-c k}$. We conclude that with $k=m_{r(\eps)}$ we have  
  $$
  \sup_{x\in X_{j+k}}|\cL_j^{k}(x)|\leq 1-C(\eps)\eps
  $$
  with $C(\eps)=\frac12 e^{-cm_{r(\eps)}}$. 
  Thus, the result holds  with $\eta(\eps)=C(\eps)\eps$.
    \end{proof}

\begin{remark}
It follows from the formula for $m_r$ in Lemma \ref{Dcover} that we can take  $\eta(\eps)=\eps C(\eps)$, where
$
C(\eps)=A(C_1,n_0,\gamma,c)\eps^{\frac{c}{\al \ln \gamma}},
$\;
 $\DS c=\sup_j\|g_j\|_\infty$ and $A(C_1,n_0,\gamma,c)$ depends continuously only on $C_1,n_0,\gamma$ and $c$ (and can easily be estimated).
However, we will not use this precise form because we will always work with a fixed $\eps$.
\end{remark}
We will constantly use the following two corollaries of the previous two lemmata.

\begin{corollary}
\label{CorNonExpanding}
Let $k_0=k_0(C_1)$ be the first positive integer $k$ such that $2C_1\del^k\leq 1$. Then
\begin{equation}\label{Norm bound}
\sup_{|t|\leq T}\sup_{j}\sup_{k\geq k_0}\|\cL_{j,t}^k\|_{*}\leq 1
\end{equation}
where $\|\cL_{j,t}^k\|_{*}$ is the operator norm with respect to the norm $\|\cdot\|_*$. 
\end{corollary}

\begin{proof}
The result follows by combining Lemmata \ref{ll1} and \ref{ll2}(a).
\end{proof}

\begin{corollary}\label{MainCor}
Given $\eps\in(0,\frac12)$ there exists $k_2=k_2(\eps)\in \bbN$ with the following properties.
If for some $l$, $m\geq k_0=k_0(C_1)$ and $t\in[-T,T]$ we have 
$\|\cL_{l,t}^{k_2+m}\|_{*}>1- \eta(\eps)$ (where  $\eta(\eps)$ comes from Lemma \ref{ll2}), then there exist a function $h$ with $\|h\|_*\leq 1$ such that 
$\DS {\min_x} |\cL_{l+k_2,t}^{m}h{(x)}|>1-\eps$. 
\end{corollary}
\begin{proof}
Let $k_2^*(\eps)$ be the smallest positive  integer $k$ so that $C_1\te_1^{k}\leq \frac12(1-\eta(\eps))$, where $\eta(\eps)$ comes from Lemma \ref{ll2}. Then by Lemma \ref{ll1} for every function $H$ such that $\|H\|_*\leq1$ and all $j\in\bbZ$, $s\geq  k_2^*(\eps)$ and $t\in[-T,T]$ we have 
\begin{equation}\label{Upppp}
 \frac{G(\cL_{j,t}^{s}H)}{2C_1}\leq 1-\eta(\eps).
\end{equation}
Next,  take $k_2(\eps)\!\!=\!\!\max(k_2^*(\eps), k_1(\eps))$, where $k_1(\eps)$ comes from Lemma \ref{ll2}(b). 
Suppose that $\|\cL_{j,t}^{k_2(\eps)+m}\|_{*}\!>\!1\!\!-\!\! \eta(\eps)$. Then  there is $h$ such that $\|h\|_*\!\leq\!\! 1$ and $\|\cL_{j,t}^{k_2(\eps)+m}h\|_*
\!>\!1\!\!-\!\! \eta(\eps)$. Set $H=\cL_{j,t}^m h$. Then
$$
\|\cL_{j,t}^{k_2(\eps)+m}h\|_*=\|\cL_{j+m,t}^{k_2(\eps)}H\|_*=\max\left(\|\cL_{j+m,t}^{k_2(\eps)}H\|_\infty, \frac{G(\cL_{j+m,t}^{k_2(\eps)}H)}{2C_1}\right)>1-\eta(\eps).
$$
Now, since $\|h\|_*\leq 1$ and $m\geq k_0$, it follows from \eqref{Norm bound} that $\|H\|_*\leq 1$. Thus, since $k_2(\eps)\geq k_2^*(\eps)$ we conclude from \eqref{Upppp} that 
$$
 \frac{G(\cL_{j+m,t}^{k_2(\eps)}H)}{2C_1}\leq 1-\eta(\eps).
$$
Hence 
$\DS
\|\cL_{j+m,t}^{k_2(\eps)}H\|_\infty>1-\eta(\eps),
$
and so
by Lemma \ref{ll2}(a),
$\DS
\|\cL_{j+m}^{k_2(\eps)}H\|_\infty>1-\eta(\eps).
$
Hence, since $k_2(\eps)\geq k_1(\eps)$ by (the contrapositive of)  
Lemma \ref{ll2}(b) we have
$$
\min_{x\in X_{j+m}}|H(x)|=\min_{x\in X_{j+m}}\left|\cL_{j,t}^m h(x)\right|
>1-\eps
$$
and the proof of the corollary is complete.
\end{proof}

\subsection{Integral of characteristic function and LLT}
\label{SSInt-LLT}

Arguing
like in \cite[\S 2.2]{HK}, in order 
to prove a non-lattice LLT starting with a measure of the form $\ka_0=q_0d\mu_0$ with $\|q_0\|_\al<\infty$ it suffices to prove the following:
\vskip0.2cm
(i) there are constants $\delta,c_2,C_2>0$ such that   for all $|t|\leq \delta$  and all $n$
\begin{equation}\label{Small t}
\left\|\mathcal{L}_{0,t}^n\right\|_*\leq C_2 e^{-c_2\sig_n^2 t^2}.
\end{equation}
\vskip0.2cm

(ii) for each $T>\delta$  we have
\begin{equation}\label{Suff}
\int_{\del\leq |t|\leq T}\|\cL_{0, t}^n\|_{*}dt=o(\sig_n^{-1}).   
\end{equation}

\noindent
Similarly, in order 
to prove a lattice LLT for integer valued observables $f_j$ it suffices to  prove (i) and 
 \begin{equation}\label{SuffLat}
\mathrm{(ii)'} \quad \quad \int_{\del\leq t\leq 2\pi-\del}\|\cL_{0, t}^n\|_{*}dt=o(\sig_n^{-1}).   
\end{equation}


We begin with \eqref{Small t}.

\begin{proposition}\label{PrSmVarBlock}
There are positive constants $\del_0,c_1,c_2,C_1,C_2$ such that for every finite sequence of functions    
$(v_j)_{j=n}^{n+m-1}$ with $a:=\max\|v_j\|_{\al}\leq \del_0$ 
  we have the following. 
Set $\cA_{j}(g)=\cL_j(e^{iv_j}g)$ and 
$\DS
\cA_{n}^m=\cA_{n+m-1}\circ \cdots\circ\cA_{n+1}\circ\cA_n.
$
Then
\begin{equation}
\label{AUpperLower}
C_1e^{-c_1\mathrm{Var}(S_{n,m} v)}\leq \|\cA_{n}^{m} \|_\al \leq C_2  e^{ -c_2 \mathrm{Var}(S_{n,m} v)}
\end{equation} 
where $\DS S_{n,m}v=\sum_{k=0}^{m-1}v_{n+k}\circ T_n^k$.
\end{proposition}

\begin{corollary}
There exists $\del>0$ such \eqref{Small t} holds for all $t\in[-\del,\del]$ and all $n$.    
\end{corollary}

\begin{proof}
We apply Proposition \ref{PrSmVarBlock}    with functions of the form $v_j\!\!=\!\!tf_j$. Let $\|f\|\!\!=\!\!\sup_j\|f_j\|_\al$. Now \eqref{Small t} follows from the upper bound in Proposition \ref{PrSmVarBlock} if $|t|\|f\|\leq \del_0$.  
\end{proof}

In the course of the proof of Proposition \ref{PrSmVarBlock} we need the following lemma.

\begin{lemma}\label{LB lemma}
There exists a constant $b>1$ with the following property.
Let $v_j$ for $ n\leq j\leq n+m-1$ be functions such that 
$\DS \max_j\|v_j\|_\al\leq 1$. Let the operators $R_{t}$ be given by
$R_t(h)=\cL_{n}^m(he^{itS_{n,m}v})$. Then for every $t\in\bbR$,
$$
\|R_t\|_\al\geq b^{-m}(1+|t|)^{-1}.
$$
\end{lemma}

\begin{proof}
We have 
$\DS
R_t(e^{-itS_{n,m}v})=\cL_{n}^m\textbf{1}=\textbf{1}.
$
Therefore,
$\DS
1=\|\textbf{1}\|_\al\leq\|R_t\|_\al\|e^{-itS_{n,m}v}\|_\al.
$
To complete the proof, we note that 
$$
\|e^{-itS_{n,m}v}\|_\al\leq 1+|t|G_\al(S_{n,m}v)\leq 1+|t|b^m
$$
for some $b>1$.
\end{proof}

\begin{proof}[Proof of Proposition \ref{PrSmVarBlock}]
The proof of the proposition uses ideas from \cite{HK} and \cite{DH2}.
We provide most of the details
for the sake of completeness.

First, the norm of the operator $\cA_n^m$ is  invariant with respect to replacing $v_j$ with $v_j-c_j$ for a constant $c_j$. Thus,  
it is enough to prove the proposition when $\mu_j(v_j)=0$ for all $j$. 
Moreover,  setting $v_j=0$ for $j\not\in\{n,m+1,...,n+m+1\}$ we can always assume that $v_j$ is defined for all $j$.

Next, we recall a few analytic tools that are crucial for the proof of both lower and upper bounds. 
Denote by $\cH_j$ the space of H\"older continuous functions $u_j$ on 
$X_j$ with exponent $\al$. Let $\cH_j^0\subset \cH_j$ be the subset of functions  such that $\mu_j(u_j)=0$.
Consider the operators 
    $\cR_{j,z,u}, z\in\bbC$ given by
$$
\cR_{j,z,u}(h)=\cL_j(he^{zu_j})
$$
where $u=(u_j)_{j=0}^\infty$ considered as a point in the 
Banach space 
$\DS 
\prod_{j=0}^\infty\cH_j^0$ equipped with the norm 
$\DS \|u\|\!\!=\!\!\sup_j\|u_j\|_\al$. 
These operators are analytic in both $z$ and $u$.
Thus  combining \eqref{ExpConvSTF0} and
\footnote{Note that in the case of a two sided sequence $(T_j)_{j\in\bbZ}$, by applying \cite[Theorem 3.3]{Nonlin} we get \eqref{RPFabove} with any fixed sequence $u$ with $\|u\|<\infty$ with constants depending also on $\|u\|$. In this case the result will follow from the analysis below with the choice $u_j=v_j/\|v\|$ (s.t. $\|u\|=1$). }  \cite[Theorem D.2]{DH2}, there are constants $\del_0$, $C>0$ and $\delta\in(0,1)$ which depend only on the maps $T_j$ such that for every complex $z$ with $|z|\leq \del_0$ and  a sequence $u$ with $\|u\|\leq \del_0$ the following holds: there are uniformly bounded and analytic in $z$ (and $u$) triplets $\la_j(z)\!\!=\!\!\la_j(z;u)\!\in\!\bbC$, $h_j^{(z)}\!\!=\!\!h_j^{(z;u)}\!\in\! \cH_j$ and $\nu_j^{(z)}\!\!=\!\!\nu_j^{(z;u)}\!\in\!\cH_j^*$ 
so that for all $n$ and $m$,
\begin{equation}\label{RPFabove}
\left\|\cR_{n,z,u}^m/\la_{n,m}(z)-\nu_{n}^{(z)}\otimes h_{n+m}^{(z)}\right\|_\al\leq C\del^m
\end{equation}
where $\DS \la_{n,m}(z)=\prod_{k=n}^{n+m-1}\la_k(z)$ and $\nu\otimes h$ is the operator $g\to \nu(g)h$. 
Moreover 
$$\nu_n^{(z)}(h_n^{(z)})=\nu_n^{(z)}(\textbf{1})=1\quad
\la_n(0)=1, \quad h_n^{(0)}=\textbf{1}\quad\text{and} \quad\nu_n^{(0)}=\mu_j.$$
  Henceforth we omit the subscript $u$.

Next,  let $\Pi_j(z)=\ln \la_j(z)$,\;  $\DS\Pi_{n,m}(z)=\sum_{j=n}^{n+m-1}\Pi_j(z)$,
and
$
\bar\Lambda_{n,m}(z)=\ln\bbE_{\mu_n}[e^{zS_{n,m}u}].
$
Then by \cite[Corollary 7.5]{DH2}  there exists $r_0>0$  and $Q<\infty$  such that if $|z|\leq r_0$ 
then for all  $n$ and $m$, 
 $$
\left|\Pi_{n,m}^{''}(z)-\bar\Lambda_{n,m}''(z)\right|\leq Q.
 $$
 Taking $z=0$ we see that 
\begin{equation}\label{2 der est}
\left|\Pi_{n,m}^{''}(0)-\text{Var}(S_{n,m}u)\right|\leq Q.
 \end{equation}
Applying \cite[Corollary 7.5]{DH2}, now with the third derivatives,  we see that if $r_0>0$ is small enough and $|z|\leq r_0$ then for all $n$ and $m$ we have
 \begin{equation}\label{3 der est}
 \left|\Pi_{n,m}^{'''}(z)-\bar\Lambda_{n,m}'''(z)\right|\leq Q_3    \end{equation}
 for some constant $Q_3$. Next, we claim that 
  there are constants $\del_1>0$ and $C_3>0$ such that if $|t|\leq \del_1$ then
\begin{equation}\label{3 der bound}
 |\Pi_{n,m}'''(it)|\leq C_3(1+\|S_{n,m} u\|_{L^2}^2).   
\end{equation}

 If $\|S_{n,m} u\|_{L^2}^2\geq A$ for a sufficiently large constant $A$ then \eqref{3 der bound} is proven exactly like \cite[(7.2)]{DH2}. 
Namely, we decompose the set $\{n,n+1,...,n+m-1\}$ into blocks such that the variance of the sums along the indexes in each one of the individual blocks is bounded above and below by two sufficiently large positive constants, and then repeat the proof of \cite[(7.2)]{DH2} with this block partition instead of the block partition of $\{0,1,2,...,n-1\}$ that resulted in \cite[(7.2)]{DH2} 
(the only difference here is that in  \cite[(7.2)]{DH2} we had $m=0$).

 It remains to prove \eqref{3 der bound} with some  
 $\del_1\!=\! \del_1(A)$ and $C_3\!=\!C_3(A)$ when $\|S_{n,m}u\|_{L^2}^2\!\leq\! A$.
In this case, using \eqref{3 der est} it is enough to bound $|\bar\Lambda_{n,m}'''(it)|$ by some constant uniformly in $n,m$ and $t\in[-\del_1,\del_1]$ for an appropriate $\del_1$. 
Using the formula 
$$
(\ln f)'''=\frac{f'''}{f^2}-\frac{2f'f''}{f^3}-\frac{2f'f''}{f^2}+
\frac{2(f')^3}{f^3}
$$
with the function $f(t)=\bbE[e^{it S_{n,m} u}]$ and noticing that 
for $|t|\leq \del_1(A)$ and $\del_1(A)$ small enough we have $|f(t)-1|\leq \frac12$, we see that there is a constant $C=C(A)$ such that
$$
|\bar\Lambda_{n,m}'''(it)|\leq C.
$$

Next, using \eqref{3 der bound} and the Lagrange form of the second order Taylor remainder of the function $\Pi_{n,m}$ around the origin we see that there are constants
$C_4,C_5$ such that
\begin{equation}\label{Pi Taylor}
\left|\Pi_{n,m}(t)+\frac{t^2}2 \text{Var}(S_{n,m} u)\right|\leq C_4t^2+C_5|t|^3(1+\text{Var}(S_{n,m} u))
\end{equation}
where 
 we used that $\Pi_{n,m}(0)\!\!=\!\!\ln\la_{n,m}(0)\!\!=\!\!0$  and 
 $\Pi_{n,m}'(0)\!\!=\!\!\mu_{n}(S_{n,m} u)\!\!=\!\!0$ (see \cite[Theorem~4.1(b)]{Nonlin}).
 \vskip1mm

 We can now complete the proof of the proposition.
    We start with the upper bound in \eqref{AUpperLower}.
 Without loss of generality we may assume that 
$\text{Var}(S_{n,m}u)\geq C_0$ where $C_0$ is a sufficiently large constant since
when the variance is smaller the required estimate could be always ensured by taking sufficiently large $C_2.$

    By \eqref{RPFabove} and the uniform boundedness of the triplets,  there is a constant $C_2>0$ such that if $|t|\leq r_0$ then
    $\DS
\|R_{n,it}\|\leq C_2|\la_{n,m}(it)|=C_2|e^{\Pi_{n,m}(it)}|.
    $
    Finally, by \eqref{Pi Taylor}
    there 
   exists $0<\del_2<\del_1$ such that if $|t|\leq \del_2$ then (for $C_0$ large enough),
    $$
|e^{\Pi_{n,m}(it)}|\leq e^{-\frac14 t^2 \text{Var}(S_{n,m}u)}.
    $$
  
Next, given  a sequence $(v_j)$ we take $u_j=v_j/r_0$. If $\sup_j\|v_j\|_\al\leq r_0 \del_0$ then $\|u\|\leq\del_0$ and so the above estimates hold. Moreover,  
    $\DS
\cA_{n}^m=R_{n,ir_0}^n
    $
    and $\text{Var}(S_{n,m}u)=r_0^{-2}\text{Var}(S_{n,m}v)$.
  Therefore, the upper bound holds with the above $C_2$ and $c_2=\frac14r_0^{-2}$.

Next, we prove the lower bound. 
Note first  that 
$
\left(\!h_{n+m}^{(z)}\otimes\nu_{n}^{(z)}\!\right)\!\!(\textbf{1})\!\!=\!\!\nu_{n}(\textbf{1})h_{n+m}^{(z)}\!\!=\!\!h_{n+m}^{(z)}.
$
Since $h_{n+m}^{(0)}=\textbf{1}$,
 the uniform boundedness and analyticity of the triplets gives
$$
\left|\left(h_{n+m}^{(z)}\otimes\nu_{n}^{(z)}\right)(\textbf{1})-1\right|\leq A|z|
$$
for some constant $A$. Therefore, there exists $\del_3>0$ such that if $|z|\leq \del_3$ then 
$$
\left|\left(h_{n+m}^{(z)}\otimes\nu_{n}^{(z)}\right)(\textbf{1})\right|\geq \frac12.
$$
Hence, if $m$ is large enough to ensure that 
\begin{equation}
\label{MSmall14}
C\del^m<\frac14,
\end{equation}
 then by \eqref{RPFabove}
\begin{equation}\label{Lower}
\|\cR_{n,z}^m\textbf{1}\|_\al\geq \frac14|\la_{n,m}(z)|.
\end{equation}
Now, by \eqref{Pi Taylor}, if $|t|\leq\del_4$ and $\del_4$ is small enough then
\begin{equation}
\label{LambdaLower}
|\la_{n,m}(it)|=|e^{\Pi_{n,m}(it)}|\geq e^{-t^2[C''+\frac{1}4\text{Var}(S_{n,m}u)]}
\end{equation}
where $C''>0$ is a positive constant.
Now the lower  bound in \eqref{AUpperLower} in case \eqref{MSmall14}
  is obtained as follows. Suppose that $\DS \sup_j\|v_j\|_\al\leq \delta_4 \delta_0$.
  Then the lower bound is obtained by taking 
  $t=\delta_4$, $u_j=v_j/\delta_4$ and using
\eqref{Lower}, \eqref{LambdaLower} and that $\|R_{n,it}^m\|_\al\geq \|R_{n,it}^m\textbf{1}\|$.

The lower bound when $C\del^m\!\!\geq\!\! \frac14$ (with $C$ sufficiently small) follows from Lemma~\ref{LB lemma}, taking into account that in this case $\|S_{n,m} v\|\leq a m\leq m$ (assuming $a\leq 1$).
\end{proof}

\subsection{Corange}
 Here we prove Theorem \ref{Reduce thm}. To simplify the proof we assume that $\mu_k(f_k)=0$ for all $k$.

Recall the definition of the set $\mathsf{H}$ in Theorem \ref{Reduce thm}. 
Taking into account Remark \ref{Rem Red} 
and Lemma \ref{al be},  
$\mathsf{H}$ is the set of all real numbers $t$ such that for all $n$ we have
\begin{equation}
\label{TimesTIntRed}
tf_n=h_n-h_{n+1}\circ T_n+g_n+Z_n
\end{equation}
where $h_n,g_n$ and $Z_n$ are functions such that $\DS \sup_n\|h_n\|_\al\!\!<\!\!\infty$, 
$\DS \mu_n(g_n)\!\!=\!\!0$, $\DS \sup_n\text{Var}(S_n g)\!\!<\!\!\infty$, 
$\DS \|g_n\|_\al  \to 0
$
and $Z_n$ is integer valued.
 It is clear that $\mathsf{H}$ is a subgroup of $\bbR$.

 Theorem \ref{Reduce thm} follows from the following corollary of Proposition \ref{PrSmVarBlock}.

\begin{corollary}\label{H corr}
\,

(i) If $\|S_nf\|_{L^2}\not\to \infty$ then $\mathsf{H}=\bbR$.
\vskip0.2cm

(ii) If $(f_j)$ is irreducible then $\mathsf{H}=\{0\}$.
\vskip0.2cm

(iii) If $(f_j)$ is reducible and $\|S_nf\|_{L^2}\to \infty$    then 
$\DS
\mathsf{H}=t_0\bbZ
$
for some $t_0>0$. 

As a consequence, the number $h_0=1/t_0$ is the largest positive number such that $(f_j)$ is reducible to an $h_0\bbZ$-valued sequence. 
\end{corollary}
\begin{proof}
(i) If the variance does not diverge to $\infty$ then by \cite[Theorem 6.5]{DH2}  (applied with with $m_0=\mu_0$) we see that $(f_j)$ is decomposed as a sum of a center tight sequence and a coboundary. Thus for every real $t$ the function $tf_j$ has decomposition \eqref{TimesTIntRed} (with $Z_j=0)$ with $g_j$ being the martingale part. Since 
$\DS
\sum_j\text{Var}(g_j)<\infty
$
we see that $\|g_n\|_{L^2(\mu_j)}\!\to\! 0$ and so by Lemma \ref{al be} we have $\|g_n\|_\al\to 0$. We thus conclude that $\mathsf{H}=\bbR$.
\vskip0.2cm
(ii) If $t\not=0$ belongs to $\mathsf{H}$ then $f$ must be reducible to an $h\bbZ$-valued sequence, with $h=1/|t|$. 
\vskip0.2cm
(iii) 
We claim first that if $t\in \mathsf{H}$ then for all  $j$ large enough  the norms $\|\cL_{j,t}^n\|_\al$ do not converge to $0$ as $n\to\infty$.  Since $\DS \sup_{n\geq j}\|g_j\|_\al\to 0$ as $n\to\infty$, for $j$ large enough  we can apply  Proposition \ref{PrSmVarBlock} with these functions to conclude 
that 
$$
\|\cA_{j,t}^n\|_\al\geq C_1e^{-c_2\text{Var}(S_{j,n}g)}
$$
where $\cA_{j,t}^n$ is the operator given by
$$
\cA_{j,t}^n(q)=\cL_j^n(e^{it S_{j,n}g}q).
$$
Since $\DS \sup_n\text{Var}(S_{j,n}g)<\infty$ we get
$\DS
\inf_{n}\|\cA_{j,t}^n\|_\al>0.
$
To finish the proof of the claim, note that 
$\DS
\cA_{j,t}^n(q)=e^{ih_{j+n}}\cL_{j,t}^n(qe^{-ih_j})
$
and so
$\DS
 \|\cA_{j,t}^n\|_\al\leq C\|\cL_{j,t}^n\|_\al
$
for some constant $C>0$.

Now, in order to complete the proof of the corollary, it is enough to show that there exists $\del_1>0$ such that for every $t\in \mathsf{H}$ and all $w\in[-\del_1,\del_1]\setminus\{0\}$ for all $j$ we have 
\begin{equation}\label{this will}
\lim_{n\to\infty}\|\cL_{j,t+w}^n\|_\al=0.
\end{equation}
\eqref{this will} shows that $t+w\not\in \mathsf{H}$. It  follows that $\mathsf{H}$ is a non empty discrete subgroup of $\mathbb{R}$,   whence $\mathsf{H}=t_0 \bbZ$
for some $t_0\in \bbR_+$,
completing the proof of the proposition.

In order to prove \eqref{this will}, let $\DS \|f\|=\sup_j\|f_j\|_\al$. Define $\del_1=\frac{\del_0}{2\|f\|+2}$, where $\del_0$ comes from Proposition \ref{PrSmVarBlock}. Thus, when $j$ is large enough and $|w|\leq \del_1$ we can  apply Proposition \ref{PrSmVarBlock} with the functions $u_j=g_j+wf_j$ to conclude that 
the operator 
$$
\cB_{j,t,w}^n(q)=\cL_j^n(qe^{it S_{j,n}g+iwS_{j,n}f})
\quad\text{satisfies}\quad
\|\cB_{j,t,w}^n\|_\al\leq C_2'e^{-c_2w^2\text{Var}(S_{j,n}f)}
$$
where  we have used that the variance of $S_{j,n}g$ is bounded in $n$. The desired estimate
\eqref{this will} follows since
$\DS
\lim_{n\to\infty}\text{Var}(S_{j,n}f)=\infty
$
and 
$\DS
\cL_{j,t+w}^n(q)=e^{-ih_{j+n}}\cB_{j,t,w}^n(qe^{ih_j}).
$
  \end{proof}









\section{Reduction lemmas}\label{Sec 5}
 Fix an integer $m_0\geq 0$. In all of the paper we will take $m_0=0$ except for 
 \S \ref{SS2SideIrr} where we take $m_0$ to be sufficiently large.
Next, take some $j$ and $\ell$ such that $\ell=k+m$ for some integers $k\geq 0$ and 
$0\leq m\leq m_0$. Let $y_1$ and $y_2$ be  two inverse branches of $T_j^k$. Let $w_1, w_2, v_1,v_2$ be inverse branches of $T_{j+k}^m$. 
Define the
{\em temporal distance function with gap $m$ by }
$$\Delta_{j,\ell,k,m}(x', x'', y_1, y_2,w_1,w_2,v_1,v_2)
$$
$$
=S_{j,\ell}(y_1\circ w_1(x'))+S_{j,\ell}(y_2\circ v_1 (x''))-S_{j,\ell}(y_1\circ w_2 (x''))-S_{j,\ell}(y_2\circ v_2(x')). $$
Note that this function is defined only for choices of $x',x'',y_i, v_i, w_i$  for which the above compositions are well defined.
 Namely we have four orbits such that 
the first and the third orbits as well
as the second and the fourth orbits  have the same itineraries up to time $j+k$,
while
the first and the fourth orbits as well
as the second and the third orbits  have the same itineraries after time $j+\ell$. During $m$ 
iterations between times $j+k$ and $j+\ell$ we do not impose any restrictions, hence the 
word {\em gap $m$} in the above definition.

\begin{lemma}
\label{LmTDSmall} 
For every $0<\delta<T$
there exist  constants $\gamma_1, \te_2\in(0,1)$ such that for every positive integer $L$ and $\eps\in(0,\frac12)$ such that 
\begin{equation}
\label{L-Epsilon}
\eps\leq \gamma_1^L 
\end{equation}
 for every $j$ and $\ell\leq L+1$ such that $\ell=k+m$, $0\leq m\leq m_0$ we have the following.
If  for some nonzero real $t$ such that $\del \leq |t|\leq T$ there exists $h$ with 
$\|h\|_{\beta, T}\leq 1$ and
\begin{equation}
\label{NormAlmostOne}
|\cL_{j,t}^{\ell} h (x)|\geq 1-\eps \quad \forall x\in X_{j+\ell}
\end{equation}
then \,
 (a)  $\forall x', x'', y_i,w_i,v_i$ as above
\begin{equation}
\label{TempDecay}
\left|\text{dist}\left(\Delta_{j,\ell,k,m}{(x', x'', y_1, y_2,w_1,w_2,v_1,v_2)},\frac{2\pi}{t}\bbZ\right)\right|\leq C_2 \theta_2^\ell;
\end{equation}
(b) Fix $\bar l\leq \ell/2$. If $T_{j+\ell-\brl}^{\brl}(y')=T_{j+\ell-\brl}^{\brl}(y'')=x$
and $z$ is an inverse branch of $T_{j}^{\ell-\brl}$ (with 
 both $y'$ and $y''$ belonging to its domain) then
\begin{equation}
\label{Cycle}
\left|\text{dist}\left(S_{j,\ell}(z(y'))-S_{j,\ell}(z(y'')),\frac{2\pi}{t}\bbZ \right)\right|\leq C_2\theta_2^\ell.
\end{equation}

\end{lemma}
Note that \eqref{TempDecay} can be written as
\begin{equation}\label{TempDecay1}
\Delta_{j,\ell,k,m}(x', x'', y_1, y_2,w_1,w_2,v_1,v_2)=
\end{equation}
$$ (2\pi/t) \mathfrak{m}_{t,j,\ell,k,m}(x', x'', y_1, y_2,w_1,w_2,v_1,v_2)+O(\theta_2^\ell)
$$
for some integer valued function $\mathfrak{m}_{t,j,\ell,k,m}$, while \eqref{Cycle} can be written as
\begin{equation}\label{Cycle1}
S_{j,\ell}(z(y'))-S_{j,\ell}(z(y''))=(2\pi/t) \mathfrak{m}_{t,j,\ell}(x,z,y',y'')+O(\theta_2^\ell)
\end{equation}
for some integer valued function $\mathfrak{m}_{t,j,\ell}$.

We will also need the following result.
\begin{lemma}
\label{LmTDVar}
There is a sequence of functions $H_k:X_k\to\bbR$ such that  
$\DS \sup_{k}\|H_k\|_{\alpha}<\infty$, constants $C>0$, $\theta<1$ and
$ c(\ell)\to 0$ as $\ell\to \infty$
with  the following properties.

Suppose that  either \eqref{Cycle} holds   for some $t$ with $\del\leq |t|\leq T$ and 
some  $j$, $\ell\geq 2n_0$ and all relevant choices of $\bar l, x,y'$ and $y''$ as described in Lemma \ref{LmTDSmall}(b), or \eqref{TempDecay} holds for some  $t$ with $\del\leq |t|\leq T$, some $j$, $\ell\geq 2n_0$, $m=n_0+1$ and all possible choices of $x',x'', w_i,v_i, w_i$. Then
\begin{equation}\label{Deccomp}
f_{j+\ell-1}=g_{j,t,\ell}+H_{j+\ell-1}-H_{j+\ell}\circ T_{j+\ell-1}+(2\pi/t) Z_{t,j,\ell}
\end{equation}
where $Z_{t,j,\ell}$ is integer valued, 
$\DS \sup_{t,j,\ell}\|g_{t,j,\ell}\|_{ \infty}\leq C\te^\ell$, and
$\DS \sup_{t,j,\ell}\|g_{t,j,\ell}\|_{\al}\leq c(\ell)$.

Moreover, the image of the function $Z_{t,j,\ell}$  is contained in either  the image of the function $\mathfrak{m}_{t,j,\ell}$ from \eqref{Cycle1} or  the image of the function $\mathfrak{m}_{t,j,\ell,k,n_0+1}, k=\ell-n_0-1$ from \eqref{TempDecay1}, depending on the case.
\end{lemma}

\begin{remark}\label{Rem conn}
If the spaces $X_k$ are connected we can take $Z_{t,j,\ell}$ to be a constant. Indeed, since all the functions $f_{j+\ell-1},g_{j,t,\ell},H_{j+\ell-1}, H_{j+\ell}\circ T_{j+\ell-1}$ are continuous we see that $Z_{t,j,\ell}$ is continuous and thus constant.
\end{remark}

\begin{remark}
It is a natural question whether in general one can arrange that $g_{j,t,\ell}$ and $Z_{j,t,\ell}$ 
will depend only on $j+\ell$ and $t$. However, the important part of the lemma is that the coboundary terms $H$ depend only on $j+\ell$, which will allow us to take the same coboundary parts 
 when \eqref{Deccomp} holds for both $j$ and $j+1$. 
This will yield the desired cancellation for  the sums $f_{j+\ell-1}+f_{j+\ell}\circ T_{j+\ell-1}$. Similarly, this will enable us to obtain an appropriate cancellation for ergodic sums $S_{j,m}$ when \eqref{Deccomp} holds 
with $j+s$ for all $0\leq s<m$. Such cancellations will be crucial for decomposing the summands inside
 such blocks into three components: coboudnaries, small terms  and a lattice valued variable.
The lattice valued variables will disappear after multiplication by $t$ and taking the exponents. The sum with small terms will be dealt with similarly to the case of small $t's$. 
The heart of the proof of Theorem \ref{LLT1} is to execute this idea precisely.
\end{remark}


\begin{proof}[Proof of Lemma \ref{LmTDSmall}]
Let
 $\DS  \fm=\sup_j\sup|g_j|$ and take $\ve$ and $L$ such that 
\begin{equation}\label{eps L rel}
 \ve\leq e^{-(3\fm+3)L}.
\end{equation}
 That is, we take $\gamma_1$ in the statement of the lemma equal to $e^{-3\fm-3}$.
Fix $h$ and $\ell\leq L+1$ such that \eqref{NormAlmostOne} holds. 
We claim that
\begin{equation}\label{h bd}
    \inf_{y\in X_j} |h(y)|\geq 1-e^{-\fm\ell-2}.
\end{equation}
To prove \eqref{h bd}, suppose that  $\exists y\in X_{j}$ such that
$|h(y)|<1-e^{-\fm\ell-2} $. Let $x=T_j^\ell y$. Then 
$$
|\cL_{j,t}^\ell h(x)|\leq \left(\cL_{j}^\ell|h|\right)(x)
\leq 1-e^{S_{j,\ell}g(y)}+(1-e^{-\fm\ell-2})e^{S_{j,\ell}g(y)}\leq
 1-e^{-\fm\ell-2} e^{-\fm\ell}<1-\ve
$$
where in the second inequality we have used \eqref{Up} with $k=\ell$. However, the latter estimate contradicts  \eqref{NormAlmostOne}.
Write $h(y)=r(y) e^{i \phi(y)}$ with $\DS \inf_{y\in X_j}r(y)\geq 1-e^{-\fm\ell-2}$. 
 Notice  that each inverse 
branch $y$
of $T_j^\ell$ has a Lipschitz constant non-exceeding $C\te^\ell$ for some constants $C>0$ and $\te\in(0,1)$. Thus,  
 $\phi$ is H\"older continuous and
$G(\phi\circ y)\leq C' \theta^{\al\ell}$ for some constant $C'$.

Next, we have
$$ (\cL_{j,t}^\ell h)(x)=\sum_{T_j^\ell y=x} e^{[S_{j,\ell}g+it S_{j,\ell}f+i\phi](y)} r(\phi(y)). $$
Note that for any probability measure $\nu$ on a probability space $\Om$ and  measurable functions $q:\Om\to \bbR$  and $r:\Om\to [0,\infty)$ 
$$
\left|\int_{\Om} r(\om)e^{iq(\om)}d\nu(\om)\right|^2=1-2\int_{\Om}\int_{\Om}\sin^2\left(\frac12(q(\om_1)-q(\om_2)\right)r(\om_1)r(\om_1)d\nu(\om_1)d\nu(\om_2).
$$
Since 
$\DS \sum_{y:\, T_j^\ell y=x} e^{S_{j,\ell}g(y)}=1$, we can define a probability measure 
$\nu=\nu_x$ on $X_j$ by\\
$\DS
\nu_x(A)=\sum_{y\in A:\, T_j^\ell y=x}e^{S_{j,\ell}g(y)}.
$
Then 
$\DS
\cL_{j, t}^\ell h(x)=\int re^{i (tS_{j,\ell}f+\phi)}d\nu_x.
$
Fix some $x\in X_{j+\ell}$. 
Since
$|\cL_{j, t}^\ell h(x)| \geq 1-\ve$ and $r=|h|\geq 1-\ve_\ell$, where $\ve_\ell=e^{-\fm\ell-2}\leq \frac12$, we see that 
$$
\sum_{y',y''}e^{S_{j,\ell}g(y')+S_{j,\ell}g(y'')}\sin^2\left(\frac12\left(tS_{j,\ell}f(y')-tS_{j,\ell}f(y'')+\phi(y')-\phi(y'')\right)\right)\leq \frac{\ve(2-\ve)}{2(1-\ve_\ell)^2}\leq 4\ve
$$
where the sum it taken over all points $y',y''$ in $X_j$ such that $T_{j}^\ell y'=T_{j}^\ell y''=x$.
Using also that $e^{S_{j,\ell}g}\geq e^{-\fm\ell}$  we see that 
for every $y',y''$  as above we have
$$
\sin^2\left(\frac12\left(t S_{j,\ell}f(y')-t S_{j,\ell}f(y'')+\phi(y')-\phi(y'')\right)\right)\leq 4\ve e^{2\fm\ell}\leq 4e^{-\fm\ell -2}\leq e^{-\fm\ell}
$$
where in the  second inequality we have  used \eqref{eps L rel}.
Therefore, with $\theta=e^{-\fm}$, for any two inverse branches $y'(\cdot)$ and $y''(\cdot)$ of $T_{j}^\ell$  (with $x$ belonging to their image) we have
\begin{equation}
\label{Colin}
t S_{j,\ell}f(y'(x))+\phi(y'(x)) -t S_{j,\ell}f(y''(x))-\phi(y''(x))\in 2\pi \integers+ O(\theta^\ell). 
\end{equation}

The above arguments show that \eqref{Colin} holds  uniformly in $x$ because of our assumption \eqref{NormAlmostOne}.
Now let $y',\tilde y'$ and $y'',\tilde y''$ be two pairs of inverse branches of $x'$ and $x''$ of the form $y'=y_1\circ w_1$, $\tilde y'=y_1\circ w_2$, $y''=y_2\circ v_2$ and $\tilde y''=y_2\circ v_1$, with $y_i,w_i$ and $v_i$ like in the definition of $\Delta_{j,\ell,k,m}$.
Thus, 
$$ 
\Delta_{j,\ell,k,m}(x', x'',y_1,y_2,w_1,w_2,v_1,v_2)+
$$
$$
[\phi(y_1\circ w_1(x'))-\phi(y_1\circ w_2 (x''))]-[\phi(y_2\circ v_2(x'))-\phi(y_2\circ v_1 (x''))]\in \frac{2\pi \integers}{t}+ 
O(\theta^\ell). $$ 
Since 
$$\phi(y_1\circ w_1(x'))-\phi(y_1\circ w_2 (x''))=O(\theta^{\beta \ell}), \quad
\phi(y_2\circ v_2(x'))-\phi(y_2\circ v_1 (x''))=O(\theta^{\beta k})=O(\theta^{\beta \ell})$$
we conclude that 
$\Delta_{j,\ell,k,m}(x', x'', y', y'')$ 
is $O(\theta^{\beta \ell})$ close to $\frac{2\pi \integers}{t}.$
This proves (a).

To prove (b) we use \eqref{Colin} with $y', y''$ replaced by $z\circ y'$ and $z\circ y''$ to get
$$ t S_{j,\ell}f(z(y'(x)))+\phi(z(y'(x))) -t S_{j,\ell}f(z(y''((x)))-\phi(z(y''(x)))\in 2\pi \integers+ O(\theta^\ell). $$
Therefore
$$ t S_{j,\ell}f(z(y'(x)))-t S_{j,\ell}f(z(y''(x)))\in 2\pi \integers+ O\left(\theta^{\beta(\ell-\brl)}\right). $$
In other words there is an integer valued function 
$m(\bry, \brry)$ such that 
\begin{equation}\label{HInteger}
    S_{j,\ell}f(z(y'(x)))-S_{j,\ell}f(z(y''(x)))
=\frac{2\pi m(y', y'')}{t}+ O\left(\theta^{\beta(\ell-\brl)}\right).
\end{equation}
This proves (b).
\end{proof}

\begin{proof}[Proof of Lemma \ref{LmTDVar}.]
In order to  present the idea of the proof in the simplest possible setting we
first  prove the lemma for sequential topologically mixing subshift of finite type.  
Then we explain the modifications needed in the general case.

In the case of  a sequential subshift we  take  $\xi$ in Assumption \ref{AssPairing} such that $\mathsf{d}_j(x,y)<\xi$ is equivalent to $x_0=y_0$ (e.g. any $1/4<\xi<1/2$ will do).
Take an arbitrary point $a=(a_0,a_1,....)$. 

 Write $q=j+\ell.$
For every two symbols $u\in\cA_{q-n_0}$ and $v\in\cA_{q}$ 
choose some admissible path $P_{q,u,v}$ of length $n_0$ from $u$ to $v$ (not including $u$ and $v$).
Define a function $\al'_{q}:X_{q}\to X_0$ by  
$$
\al'_{q}(x)=[a_0,a_1,...,a_{q-n_0},P_{q, a_{q-n_0}, x_{q}}, x]:=
(a_0,a_1,...,a_{q-n_0},P_{q, a_{q-n_0}, x_{q}}, x_{q}, x_{q+1},\dots)
$$
where $x=(x_{q},x_{q+1}...)\in X_{q}$.
Then the function $x\to\al'_{q}(x)$ is Lipschitz  and 
\begin{equation}\label{al prime lip}
\mathsf{d}_0(\al'_{q}(x),\al'_{q}(x'))\leq C2^{-q}\mathsf{d}_{q}(x,x')
\end{equation}
for some constant $C=C_{n_0}$ which depends only on $n_0$. Indeed, if for some $s$ we have $x_{q+k}=x'_{q+k}$ for all $k\leq s$ then 
$\DS
P_{q, a_{q-n_0}, x_{q}}=P_{q, a_{q-n_0}, x_{q+\ell}'}
$
and so the first $q+s$ coordinates of $\al'_{q}(x)$ and $\al'_{q}(x')$ coincide. On the other hand, if $x_{q}\not=x_{q}'$ then we have
$\mathsf{d}_0(\al'_{q}(x),\al'_{q}(x'))=2^{-(q-n_0)}=2^{n_0}2^{-q}
\mathsf{d}_q(x,x')$ (as $\mathsf{d}_q(x,x')=1$).
Next, let
$$
\al_{j,\ell}=T_0^j\circ \al'_{j+\ell},\quad R_{j,\ell}=S_{j,\ell}\circ  \al_{j,\ell}
$$
and 
$$
H_{j,\ell}=R_{j,\ell}-S_{j,\ell}\circ T_0^{j}(a)=S_{j,\ell}\circ T_0^j\circ \al'_{j+\ell}-S_{j,\ell}\circ T_0^{j} (a).
$$
Let the functions $R_ k$ and $H_k$ be given by
$$
R_{k}=S_{0,k}\circ \al'_{k}
\quad\text{and}\quad
H_{k}=R_{k}-S_{0,k}(a)=S_{0,k}\circ\al'_{k}-S_{0,k}(a).
$$
Then $\DS \sup_k\|H_k\|_{\beta}<\infty$ since the first $k-n_0$ coordinates of $\al'_{k}(x)$ coincide with those of 
$a$ and $\DS \sup_k \|f_k\|_{\be}<\infty$, and if for some $s$ and points $x,x'\in X_{j+\ell}$ we have $x_{j+\ell+m}=x_{j+\ell+m}',\,m\leq s$ then $\al_{j+\ell}'(x)$ and $\al_{j+\ell}'(x')$ have the same $j+\ell+s$ first coordinates. \\

We  claim next that
\begin{equation}\label{C1}
\left\|\left(H_{j,\ell+1}\circ T_{j+\ell}-H_{j,\ell}\right)-\left(H_{j+\ell+1}\circ T_{j+\ell}-H_{j+\ell}\right)\right\|_{\beta}=O(\te_3^\ell)
\end{equation}
for some $\te_3\in(0,1)$.
Indeed, since
\begin{equation}\label{Cocycle}
S_{j,\ell}\circ T_0^j=S_{0,j+\ell}-S_{0,j}
\end{equation}
we have
\begin{equation}\label{est1}
R_{j,\ell+1}\circ T_{j+\ell}-R_{j,\ell}=\left(S_{0,j+\ell+1}-S_{0,j}\right)\circ\al'_{j+\ell+1}\circ T_{j+\ell}-\left(S_{0,j+\ell}-S_{0,j}\right)\circ\al'_{j+\ell}
\end{equation}
$$
=R_{j+\ell+1}\circ T_{j+\ell}-R_{j+\ell}+\left(S_{0,j}\circ\al'_{j+\ell}-S_{0,j}\circ\al'_{j+\ell+1}\circ T_{j+\ell}\right).
$$
Since the points $\al'_{j+\ell}(x)$ and $\al'_{j+\ell+1}\circ T_{j+\ell}(x)$ have the same $j+\ell-n_0$ first coordinates and $\DS \sup_k\|f_k\|_\beta<\infty$ we see that 
$$
\sup|S_{0,j}\circ\al'_{j+\ell}-S_{0,j}\circ\al'_{j+\ell+1}\circ T_{j+\ell}|=O(2^{-\beta\ell}).
$$
Using also \eqref{al prime lip} we get
\begin{equation}\label{est2}
\|S_{0,j}\circ\al'_{j+\ell}-S_{0,j}\circ\al'_{j+\ell+1}\circ T_{j+\ell}\|_{\be}=O(2^{-\beta\ell}).
\end{equation}
In order to complete the proof of \eqref{C1}, we note that
\begin{equation}\label{H to R }
H_{j,\ell+1}\circ T_{j+\ell}-H_{j,\ell}=R_{j,\ell+1}\circ T_{j+\ell}-R_{j,\ell}+(S_{j,\ell}T_0^{j}a-S_{j,\ell+1}T_0^{j}a)
\end{equation}
$$
=R_{j,\ell+1}\circ T_{j+\ell}-R_{j,\ell}-f_{j+\ell}(T_0^{j+\ell}a)
\,\,\,\text{ and }
$$
$$
H_{j+\ell+1}\circ T_{j+\ell}-H_{j+\ell}=R_{j+\ell+1}\circ T_{j+\ell}-R_{j,\ell}+(S_{0,j+\ell}a-S_{0,j+\ell+1}a)$$
$$=R_{j+\ell+1}\circ T_{j+\ell}-R_{j+\ell}-f_{j+\ell}(T_0^{j+\ell}a).
$$
 Hence the expression inside the absolute value on the LHS of \eqref{C1} equals to
 
$\DS \left[R_{j,\ell+1}\circ T_{j+\ell}-R_{j,\ell}\right]-\left[R_{j+\ell+1}\circ T_{j+\ell}-R_{j+\ell}\right] $
which together with \eqref{est1} and \eqref{est2} yields \eqref{C1}. \\

In view of \eqref{C1} and \eqref{H to R } it is enough to prove \eqref{Deccomp} with $H_{j,\ell-1}$ and $H_{j,\ell}$ instead of $H_{j+\ell-1}$ and $H_{j+\ell}$, respectively. To prove \eqref{Deccomp} with these functions we write 
\begin{equation}\label{Start}
R_{j,\ell+1}\circ T_{j+\ell}-R_{j,\ell}=S_{j,\ell+1}\circ \al_{j,\ell+1}\circ T_{j+\ell}-S_{j,\ell}\circ \al_{j,\ell}
\end{equation}
$$
=S_{j,\ell+1}\circ \al_{j,\ell}-S_{j,\ell}\circ \al_{j,\ell}+D_{j,\ell}=f_{j+\ell}\circ T_j^{\ell}\circ \al_{j,\ell}+D_{j,\ell}=f_{j+\ell}+D_{j,\ell}
$$
where 
$$
D_{j,\ell}:=S_{j,\ell+1}\circ \al_{j,\ell+1}\circ T_{j+\ell}-S_{j,\ell+1}\circ \al_{j,\ell}
$$
and in the last equality he have used that $T_j^\ell\circ \al_{j,\ell}=\mathrm{id}$.

Next, we claim that 
\begin{equation}\label{C2}
D_{j,\ell}=\frac{2\pi Z_{j,\ell}}{t}+ \fR_{j,\ell}
\end{equation}
where $Z_{j,\ell}$ is an integer valued function and  $\fR_{j,\ell}$ is a function such that 
$\sup|\fR_{j,\ell}|\!\!=\!\!O(\theta_2^\ell)$.
Let us complete he proof of  \eqref{Deccomp} based on the validity of \eqref{C2}. 
By \eqref{Start} and \eqref{C1} it is enough to show that 
\begin{equation}\label{c3}
    \|\fR_{j,\ell}\|_\al\leq c(\ell)
\end{equation} 
where $c(\ell)$ satisfies $c(\ell)\to 0$ as $\ell\to \infty$.

Since $\al_{j,\ell}=\al'_{j+\ell}\circ T_0^j$, \eqref{al prime lip}  gives 
$\DS \sup_{j,\ell}\|D_{j,\ell}\|_\beta<\infty$. Since
$$
\sup|t\fR_{j,\ell}|=\sup|tD_{j,\ell}-2\pi Z_{j,\ell}|=O(\te_2^\ell).
$$
and $\DS \sup_{j,\ell}\|D_{j,\ell}\|_\beta<\infty$, we see that if $\ell$ is large enough then  $Z_{j,\ell}$ must be
 constant on balls or radius $r$ for some positive constant $r$. (Indeed if $x_1$ and $x_2$ are close then 
 $|t||D_{j,\ell}(x_1)-D_{j,\ell}(x_2)|<C\theta_2^\ell$, so $|Z_{j,\ell}(x_1)-Z_{j,\ell}(x_2)|<1$, 
 meaning $Z_{j,\ell}(x_1)=Z_{j,\ell}(x_1)$.) 
 We conclude that $\|Z_{j,\ell}\|_{\beta}\leq C$ for some constant $C$ which does not  depend on $j$ or $\ell$. Therefore 
 $\DS \sup_j\sup\|\fR_{j,\ell}\|_\beta<\infty$. Since  $\sup|\fR_{j,\ell}|\to 0$ as $\ell\to \infty$, 
Lemma \ref{al be} implies
    that  
 $\|\fR_{j,\ell}\|_\al\leq c(\ell)$  for some sequence $c(\ell)$ so that $c(\ell)\to 0$, and \eqref{c3} follows. 

 In order complete
the proof of the lemma for subshifts it remains to prove \eqref{C2}. Set
 $$
z(y)=[a_j,a_{j+1},...,a_{j+\ell-n_0-2},y]=(a_j,...,a_{j+\ell-n_0-2},y_{j+\ell-n_0{-1}},y_{j+\ell-n_0},...)
 $$
 which is defined for all words $y=(y_k)_{k\geq j+\ell-n_0}$ such that $a_{j+\ell-2-n_0}$ and $y_{j+\ell-n_0-1}$ are linked.  
 Let us also set 
 $$
y'(x)=(a_{j+\ell-n_0-1},P_{j+\ell-1, a_{j+\ell-n_0-1},x_{j+\ell-1}},x_{j+\ell-1}, x_{j+\ell},...)
 $$
 and 
 $$
y''(x)=(a_{j+\ell-n_0-1},a_{j+\ell-n_0},P_{j+\ell, a_{j+\ell-n_0},x_{j+\ell}}, x_{j+\ell},x_{j+\ell+1},...).
 $$
 Then 
 $$
\al_{j,\ell-1}(x)=z(y'(x))
 \quad
\text{and} \quad
\al_{j,\ell}(T_{j+\ell-1}x)=z(y''(x)).
 $$
 Notice that $z(\cdot)$ is an inverse branch of $T_{j}^{\ell-n_0-2}$ and that 
 $$T_{j+\ell-n_0-1}^{n_0+1}y'(x)=T_{j+\ell-n_0-1}^{n_0+1}y''(x)=T_{j+\ell-1}x.$$
Thus under \eqref{Cycle} (applied with the point $x'=T_{j+\ell-1}x$, 
 and $\brl=n_0+1$) we have
\begin{equation*}
D_{j,\ell-1}=S_{j,\ell}\circ \al_{j,\ell}\circ T_{j+\ell-1}-S_{j,\ell}\circ \al_{j,\ell-1}=\frac{2\pi Z}{t}+\fR
\end{equation*}
where $Z=Z_{j,\ell-1}$ is an integer valued function and  $\fR=\fR_{j,\ell-1}$ is a function such that 
$\sup|\fR|=O(\theta_2^\ell)$.
This completes the proof of \eqref{C2}  under \eqref{Cycle} for sequential subshifts. 
If instead \eqref{TempDecay} holds with $m=n_0+1$, 
then for every point $x$ we have
$$
D_{j,\ell-1}(x)=\Delta_{j,\ell,k,m}(x', x'', y_1, y_2,w_1,w_2,v_1,v_2)
$$
where $x'=T_{j+\ell}x$, $x''=T_{j+\ell}a$, $y_1=y_2$ coincide with the inverse branch corresponding to the cylinder $[a_j,...,a_{j+\ell-n_0-1}]$,  $w_1$ is the inverse branch corresponding to the cylinder $[a_{j+\ell-n_0}P_{j+\ell-n_0,a_{j+\ell-n_0},x_{j+\ell+1}}]$, $v_2$ is inverse branch corresponding to the cylinder $[P_{j+\ell-n_0-1,a_{j+\ell-n_0-1},x_{j+\ell}},x_{j+\ell}]$ and $w_2=v_1$ coincide with the inverse branch corresponding to the cylinder $[a_{j+\ell-n_0},...,a_{j+\ell}]$.
Indeed, we have that $y_1\circ w_2=y_2\circ v_1$. This
finishes the proof of the lemma for sequential subshifts.
\\


The proof for the more general maps proceeds as follows. First, by Lemma \ref{InvBranch} for every $j, n$ and $y\in X_j$ there is an inverse branch $Z_{j,y,k}:B_{j+k}(T_j^k y,\xi)\to X_{j+k}$ of $T_j^j$ such that 
$$
\text{dist}\left(T_j^{s}(Z_{j,y,k}x), T_j^s y\right)<\xi
$$
for all $s\leq k$ and $x\in B_{j+k}(T_j^k y,\xi)$. The map
$Z_{j,y,k}$ corresponds to the inverse of the map $x\to [y_k,...,y_{k+n-1},x]$ in the case of a subshift, where $y=(y_k,y_{k+1},...)$ and $x=(x_{k+n},x_{k+n+1},...)$. 
Let us  define 
$$
\al'_{j+\ell}=Z_{0,x_0,j+\ell-n_0}\circ S_{j+\ell-n_0,T_0^{j+\ell-n_0}x_0, n_0}=Z_{0,x_0,j}\circ Z_{j, T_0^j x_0, \ell-n_0}\circ W_{j+\ell-n_0,T_0^{j+\ell-n_0}x_0}
$$
where  $W_{j+\ell-n_0,T_0^{j+\ell-n_0}x_0}$ is the right inverse of $T_{j+\ell-n_0}^{n_0}$ from Assumption \ref{Ass n 0}.
Then the proof of the lemma proceeds like in the case of a subshift of finite type with the above definition of the function $\al'_{j+\ell}$, using the  properties of the inverse branches from Lemma~\ref{InvBranch}.
\end{proof}

\section{Local limit theorem in the irreducible case}\label{SSKeyPropZN}

 Here we prove Theorems \ref{LLT1} and \ref{LLT Latt}.  To simplify the proofs we will assume that $\ka_0(f_j\circ T_0^j)=0$ for all $j$, that is, $\bbE[S_n]=0$ for all $n$. This could always be achieved by subtracting a 
constant from $f_j.$


\subsection{Contracting blocks}
\label{SSCBDef}
Fix some $T>\del$ and partition $\{|t|: \del\leq |t|\leq T\}$ into intervals  of small length $\del_1$ (yet to be determined).
 In order to prove \eqref{Suff} it is sufficient to show that if $\del_1$ is small enough then 
for every interval $J$ whose length is smaller than $\delta_1$ we have 

\begin{equation}\label{Suff J}
 \int_{J}\|\cL_{0, t}^n \|_{*}dt
=o(\sig_n^{-1}).
\end{equation}

Let us introduce  a simplifying notation. Given an interval 
of positive integers \\$I=\{a,a+1,...,a+d-1\}$  we write
$\DS
\cL_{t}^{I}=\cL_{a,t}^{d}=\cL_{a+d-1,t}\circ\cdots\circ\cL_{a+1,t}\circ \cL_{a,t}
$
and
$
\DS S_I=S_If=\sum_{j\in I}f_j\circ T_0^j.
$
Henceforth we will refer to a finite interval in the integers as a ``block". The length of a block is the number of integers in the block.

 Fix a small $\ve\in(0,1)$ (that will be determined latter).
We say that a block $I$ is {\em contracting} if 
\begin{equation}
\label{DefContrBl}
\DS \sup_{t\in J}\|\cL_{t}^{I}\|_{*}\leq 1-\eta(\ve)
\end{equation}
where $\eta(\ve)>0$ comes from Lemma \ref{ll2}.
\begin{lemma}\label{SubLemm}
If $I=\{a+1,a+2,...,a+d\}$ is  a non contracting block of size larger than $2k_0$ 
 (where $k_0$ comes from \eqref{Norm bound}) 
and $I''\subset I$ is a sub-block such that $I\setminus I''$ is composed of a union of two disjoint blocks whose lengths not less than $k_0$ 
then $I''$ is a non-contracting block.
\end{lemma}

\begin{proof}
Decompose $I=I'\cup I''\cup I'''$ where the blocks $I',I'',I'''$ are disjoint and are ordered so that $I'$ is to the left of $I''$ and $I''$ is to the left of $I'''$.  Since $I$ is non-contracting there exists $t\in J$ such that $1-\eta(\ve)<\|\cL_{t}^{I}\|_{*}$. On the other hand, by sub-multiplicativity of operator norms we have 
\begin{equation}\label{BBY}
1-\eta(\ve)<\|\cL_{t}^{I}\|_{*}=\|\cL_{t}^{I'''}\circ\cL_{t}^{I''}\circ\cL_t^{I'}\|_*\leq 
\|\cL_{t}^{I'''}\|_{*}\|\cL_{t}^{I''}\|_{*}\|\cL_{t}^{I'}\|_{*}.
\end{equation}
Since the lengths of $I'$ and $I'''$  are at least $k_0$ by \eqref{Norm bound} 
we have
$\DS
\|\cL_{t}^{I'}\|_{*}\!\leq\! 1$ and $\|\cL_{t}^{I'''}\|_{*}\!\leq\! 1 .$
Thus by \eqref{BBY} we have 
$\DS
1-\eta(\ve)<\|\cL_{t}^{I''}\|_{*}.
$
Therefore the block $I''$ is non-contracting.
\end{proof}

Combining Lemmata  \ref{LmTDVar} and \ref{SubLemm} together with Corollary \ref{MainCor} we
obtain the following result.

\begin{corollary}\label{CutLemm}
 Let $H_{k}$ be the functions   from Lemma \ref{LmTDVar}. Let $k_0$ be 
 from
 \eqref{Norm bound} and $k_2(\cdot)$ be  from Corollary \ref{MainCor}.
If $\ve$ is small enough then 
there exists $L=L(\ve)\geq \max(k_0,2n_0+1)$ 
such that $L(\ve)\to \infty$ as $\ve\to 0$ with the following properties. 
If $I=\{a,a+1,...,a+d-1\}$ 
is  a non contracting block such that $d>2k_0+k_2(\ve)+L(\ve)$ 
then for every $s\in I$ with 
$$
a+k_0+L(\ve)+k_2(\ve)\leq s\leq a+d-k_0-1
$$
and all $t\in J$ we can write
\begin{equation}\label{Reduc}
tf_{s}=tg_s+tH_{s}-tH_{s+1}\circ T_s+2\pi Z_s
\end{equation}
where $\DS \|g_{s}\|_{ \al}=O(\del_1)+ c(L)$ where $c(L)\to 0$  as $L\to \infty$ (and $\del_1$ is the length of the interval $J$  in \eqref{Suff J}). 
\end{corollary}
\begin{remark}
The functions $g_s$ and $Z_s$ can also depend on $t,\ve$ and $I$, but it is really important  for the next steps that the functions $H_s$ do not depend on $I$. 
\end{remark}

\begin{proof}
For each $\ve$,  take $L=[a|\ln \ve|]$ where $a$ is sufficiently large so that \eqref{L-Epsilon} holds.
In the following arguments, in order to simplify the notations we write $L=L(\ve)$ and $k_2=k_2(\ve)$.

Take $l\in I$ with $a+k_0\leq l$ and $l+k_2+L-1<a+d-1-k_0$. By Lemma \ref{SubLemm} applied with the sub-block $I''=\{l,l+1,...,l+k_2+L-1\}$ we have 
$\DS
\sup_{t\in J}\|\cL_{l,t}^{k_2+L}\|_*>1-\eta(\ve).
$
Fix some $t'\in J$ such that 
$\DS
\|\cL_{l,t'}^{k_2+L}\|_*>1-\eta(\ve).
$
By Corollary \ref{MainCor} applied with $m=L$ there is a function $h$ with $\|h\|_*=1$ and 
$$
\min_{x\in X_{l+k_2+L}}|\cL_{l+k_2,t'}^{L}h(x)|> 1-\ve.
$$
 Therefore, by Lemma \ref{LmTDVar}, with $j=l+k_2$ and $\ell=L=L(\ve)$ we have 
$$
t'f_{l+k_2+L-1}=t'f_{j+L-1}=
t'g_{t',j,L}+t'H_{j+L-1}-t'H_{j+L}\circ T_{j+L}+2\pi Z_{t',j,L}
$$
where $Z_{t',j,L}$ is integer valued and $\|g_{t',j,L}\|_\al\leq c(L)$ with $c(L)\to 0$ as $L\to \infty$. Now, if $t\in J$ then we can write
$$
tf_{j+L-1}=t'f_{j+L-1}+(t-t')f_{j+L-1}=
g_{t,j,L}+tH_{j+L-1}-tH_{j+L}\circ T_{j+L-1}+2\pi Z_{t',j,L}
$$
where 
$$
g_{t,j,L}=t'g_{t',J,L}+(t-t')\left(f_{j+L-1}+H_{j+L}\circ T_{j+L-1}-H_{j+L-1}\right).
$$
Since the length of the interval $J$ does not exceed $\del_1$ and the $\|\cdot\|_\al$ norms of the functions $f_k$ and $H_{k}$ are uniformly bounded we have 
$\DS
\|g_{t,j,L}\|_{\al}\leq c(L)+C\del_1
$
for some constant $C$.  

To finish the proof, note that any $s\in I$ can be written as $s=j+L-1=l+k_2+L-1$ for some $l$ with the above properties when $a+k_0+L+k_2\leq s\leq a+d-k_0-1$.
\end{proof}

The last  key tool 
needed for the proof of \eqref{Suff J}  is the following  simple fact. 
\begin{lemma}\label{Qqad L}
 Let $Q(h)=ah^2+bh+c$ be a quadratic function with $a>0$ 
  and $\cJ$ be an interval.
  Then there is an absolute constant $C>0$ such that
 $$
 \int_{\cJ} e^{-Q(h)}dh\leq 
 \frac{C}{\sqrt{a}} \exp\left[-{  \min_{\cJ}} Q(h)\right].
 $$
\end{lemma}
\begin{proof}
By linear change of variables we can reduce the problem to the case $Q(h)=h^2.$
Now there are two cases:

(1) If $[-1, 1]\cap \cJ\neq\emptyset$ 
then the result follows because $\DS \int_{\mathbb{R}} e^{-h^2} dh<\infty.$

(2) If $[-1, 1]\cap \cJ=\emptyset$ then the result follows since for $A\geq 1$

\hskip2cm
$\DS \int_A^{\infty} e^{-h^2} dh<\int_A^\infty \frac{2h}{2A} e^{-h^2} dh=\frac{e^{-A^2}}{2A}<e^{-A^2}. $
\end{proof}

Next, let 
$$
D(\ve)=4\left(L(\ve)+n_0+k_0+k_2(\ve)\right)
$$
where $L(\ve)$ comes from Corollary \ref{CutLemm}, $k_2(\ve)$ comes from Corollary \ref{MainCor}, $n_0$ comes  from  Assumption \ref{Ass n 0} and $k_0$ comes from \eqref{Norm bound}. 
Let $L_n$ be the maximal number of contracting blocks contained in 
$I_n=\{0,1,...,n-1\}$, 
 such that 
the distance between consecutive blocks is at least $k_0$ and the 
length of each block is between $D(\ve)$ and $2D(\ve)$.

 Let $\fB\!\!=\!\!\{B_1, B_2, \dots, B_{L_n}\}$ be a corresponding set of contracting blocks
separated by at least $k_0$, 
and $\fA\!\!=\!\!\{A_1, \!\dots\!, A_{\tilde L_n}\}$ be a partition into intervals of the compliment of the union of the blocks $B_j$ in $I_n$, ordered so that $A_j$ is to the left of $A_{j+1}$. 
Thus
$\tilde L_n\!\!\in\!\! \{L_n\!-\!1, L_n, L_n\!+\!1\}.$

We will  prove of \eqref{Suff J} by considering three cases  depending on the size of $L_n.$

\subsection{Large number of contracting blocks}
\label{SSLrgeCB}

The first case is when $L_n$ is at least of logarithmic order in $\sig_n$. More precisely, we have the following result.

\begin{proposition}\label{case 1}
    There is a constant $c=c(\ve)>0$ such that \eqref{Suff J} holds if $L_n\geq c\ln\sig_n$. 
\end{proposition}

\begin{proof}
By the submultiplicativity of operator norms, 
Corollary \ref{CorNonExpanding} and 
\eqref{DefContrBl},
 $\forall t\in J$ we have
$$
\|\cL_{0,t}^n\|_*\leq \left(\prod_{k}\|\cL_{t}^{A_{k}}\|_*\right)\cdot \left(\prod_{j}\|\cL_{t}^{B_{j}}\|_*\right)\leq 
\left(1-\eta(\ve)\right)^{L_n}.
$$
Hence \eqref{Suff J} holds when $L_n\geq c\ln\sig_n$  for $c$ large enough. 
\end{proof}

\subsection{Moderate number of contracting blocks}

It remains to consider the case where $L_n\leq c\ln \sig_n$ for $c=c(\ve)$ from Proposition \ref{case 1}.

\begin{proposition}\label{case 2}
There is $\ve_0>0$ such that   \eqref{Suff J} holds if $\ve<\ve_0$ and $L_n\to\infty$ but
$L_n\leq c\ln \sig_n$,
   where $c=c(\ve)$ comes from Proposition \ref{case 1}.
\end{proposition}

We first need the following result.

\begin{lemma}\label{A lemma}
Let
$A=\{a,a+1,...,b\} \subset\{0,1,...,n-1\}$ be a block of length greater or equal to $4D(\ve)+1$, which 
does not intersect contracting blocks from $\mathfrak{B}$. Define $a'=a+2 k_0+L(\ve)+k_2(\ve)$ and 
$b'=b-2k_0-1$, and set $A'=\{a',a'+1,...,b'\}$. Then for every 
$s\in A'$ and all $t\in J$ we can write 
\begin{equation}\label{deccomp}
    tf_s=g_{t,s}+H_s-H_{s+1}\circ T_s+2\pi Z_s
\end{equation}
where $\|g_{t,s}\|_\al \leq C(\ve)+c_0\del_1$ for some $C(\ve)$ such that $C(\ve)\to 0$ as $\ve\to 0$,  $c_0$ is a constant (here $\del_1$ is the length of the underlying interval $J$) and $Z_s$ are integer valued.
\end{lemma}

\begin{proof}
 Let $s\in A'.$ Then, since $\mathfrak{B}$ is maximal, the block of length $D(\ve)$ ending at $s$
is non contracting. Therefore the result follows from Corollary \ref{CutLemm}.
\end{proof}

We also need the following result. 
\begin{lemma}\label{CombLemma}
Suppose that $L_n=o\left(\sig_n^2\right)$. Then 
 there exists $ m=m_n\in \{1, \dots, \tilde L_n\}$ such that
\vskip0.2cm

(i) For all $n$ large enough we have 
$\DS
\|S_{A_{m}}\|_{L^2}\geq \frac{\sig_n}{4\sqrt{L_n}}$ $\DS \left(\text{and so } |A_{m}|\geq \frac{\sig_n}{4\sqrt{L_n}\|f\|_\infty} \right)
$
where $\DS \|f\|_\infty=\sup_j\|f_j\|_\infty$ and $|A_m|$ is the size of $A_m$.
\vskip0.2cm

(ii)
Write $A_{m_n}=\{a_n,a_{n}+1,..., b_n\}$ and set 
$A_{m_n}'=\{a_n',a_n'+1,...,b_n'\}$, where 
$$a_n'=a_n+ 2 k_0+L(\ve)+k_2(\ve)\,\, \text{ and }\,\,
b_n'=b_n- 2k_0-1.
$$
Then, if $n$ is large enough, \eqref{deccomp} holds for every $s\in A_{m_n}'$ and all $t\in J$.

\end{lemma}

\begin{proof}
Let $\fC=\fA\cup \fB.$ Then 
$\DS
S_n=\sum_{ C\in \fC}S_{C}=\sum_{k=1}^{L_n}S_{B_k}+\sum_{l=1}^{\tilde L_n}S_{A_l}.
$
Thus
$$ 
\sig_n^2=\|S_n\|_{L^2}^2=\sum_k \|S_{B_k}\|_{L^2}^2+\sum_l \|S_{A_l}\|_{L^2}^2+2\sum_{l_1<l_2} \mathrm{Cov}(S_{C_{l_1}}, S_{C_{l_2}}).
$$ 
Now, since the size of each block $B_j$ is at most $2D(\ve)$ and  $\DS \|f\|_\infty=\sup_k\|f_k\|_\infty<\infty$,
$$
\sum_{1\leq k\leq L_n}\|S_{B_k}\|_{L^2}^2\leq \left(2D(\ve)\|f\|_\infty\right)^2 L_n
=o(\sig_n^2).
$$
Next,  for every sequence of random variables $(\xi_j)$ such that $|\text{Cov}(\xi_n,\xi_{n+k})|\leq C\del^k$ for $C>0$ and $\del\in(0,1)$ we have the following.  For all $a<b$ and $k,m>0$,
$$
|\text{Cov}(\xi_{a}+...+\xi_{b}, \xi_{b+k}+...+\xi_{b+k+m})|\leq \sum_{j=a}^b\sum_{s=k}^{\infty}|\text{Cov}(\xi_j,\xi_{b+s})|\leq C\sum_{j=a}^b\sum_{s\geq k}\del^{b+s-j}
$$
$$
\leq C_\del\sum_{j=a}^b\del^k\del^{b-j}=C_{\del}\del^k\sum_{j=0}^{b-a}\del^j\leq C_\del(1-\del)^{-1}
$$
where $C_\del=C/(1-\del)$.
Applying this with $\xi_j=f_j\circ T_0^j$ (and using the exponential decay of correlations, see \cite[Theorem 3.3]{Nonlin} or \cite[Remark 2.6]{DH2}), we see that there is a constant $R>0$ such that
for each $l_1$, 
$$\left| \sum_{C_{l_2}: l_2>l_1} \mathrm{Cov}\left(S_{C_{l_1}}, S_{C_{l_2}}\right)\right|\leq R.$$
Therefore
$$
\sig_n^2=\sum_{1\leq k\leq  \tilde L_n}\|S_{A_k}\|_{L^2}^2+o(\sig_n^2)+O(L_n)=
\sum_{1\leq k\leq  \tilde L_n}\|S_{A_k}\|_{L^2}^2+o(\sig_n^2).
$$
Thus, if $n$ is large enough then 
there is at least one index $m$ such that 
$$
\|S_{A_m}\|_{L^2}\geq \frac{\sig_n}{4\sqrt{L_n}}.
$$
Next,   by the triangle inequality
$
\DS \|S_{A_{m}}\|_{L^2}\leq \
\sup_{j}\|f_j\|_{L^2(\mu_j)}|A_m|\leq \|f\|_\infty |A_m|
$
and so
$\DS
|A_m|\!\!\geq\!\! \frac{\sig_n}{4{\sqrt{L_n}}\|f\|_\infty}.
$
Thus property (i) holds.
Property (ii) follows from  Lemma~\ref{A lemma}.
\end{proof}

   To complete the proof of Proposition \ref{case 2}, we will prove the following result.

  \begin{lemma}\label{Norm Lem}
  There is $\ve_0>0$ such that if the length $\del_1$ of $J$ satisfies $\del_1<\ve_0$ and if $\ve<\ve_0$ then
      $\DS
\int_{J}\|\cL_{t}^{A_{m_n}}\|_*dt=O\left(\sqrt{L_n}\sig_n^{-1}\right).
      $
  \end{lemma}

\begin{proof}[Proof of Lemma \ref{Norm Lem}]
Let $A_{m_n}'$ be defined in Lemma \ref{CombLemma}. 
Then we can write 
$$
A_{m_n}=U_n\cup A_{m_n}'\cup V_n
$$
for blocks $U_n$ and $V_n$ such that  $U_n, A_{m_n}', V_n$ are disjoint, $U_n$ is to the left of $A_{m_n}'$ and $V_n$ is to its right. Moreover, $V_n$ is of size $2k_0+1$ and $U_n$ is of size $2k_0+L(\ve)+k_2(\ve)$. Thus by \eqref{Norm bound},
$$
\sup_{t\in J}
\max\left(\|\cL_t^{U_n}\|_*,\|\cL_t^{V_n}\|_* \right)\leq 1.
$$
Since 
$$
\cL_{t}^{A_{m_n}}=\cL_{t}^{V_n}\circ\cL_{t}^{A_{m_n}'}\circ\cL_t^{U_n}
$$
we conclude that 
$$
\|\cL_{t}^{A_{m_n}}\|_*\leq \|\cL_{t}^{A_{m_n}'}\|_*.
$$
Thus, its enough to show that 
\begin{equation}\label{suff A prime}
\int_{J}\|\cL_{t}^{A_{m_n}'}\|_*dt=O\left(\sqrt{L_n}\sig_n^{-1}\right).
\end{equation}

Next, let us write $A_{m_n}'=\{a'_n,a'_{n}+1,...,b'_n\}$. 
Then,  by Lemma \ref{CombLemma}(ii), for all $s\in A'_{m_n}$ and all $t\in J$ we can write 
$\DS
    tf_s=g_{t,s}+H_s-H_{s+1}\circ T_s+2\pi Z_{t,s}
$
where $\|g_{t,s}\|_{\al} \leq C(\ve)+c_0\del_1$, for some $C(\ve)$ such that $C(\ve)\to 0$ as $\ve\to 0$, and $c_0$ is a constant. 
In particular, if $t_0$ is the center of $J$ and $t=t_0+h\in J$, then for all $s\in A_{m_n}'$ we have 
$$
tf_s=t_0f_s+hf_s=(g_s+hf_s)+t_0H_s-t_0H_{s+1}+2\pi Z_s
$$
where $g_s=g_{t_0,s}$ and $Z_s=Z_{t_0,s}$.
Therefore, for any function $u$ we have
$$
\cL_{t}^{A_{m_n}'}u=e^{-it_0H_{b'_{n}}}\cL_{a'_{n}}^{b'_{n}-a'_{n}}(e^{iS_{a'_{n},b'_{n}}g+it_0H_{a'_{n}}+
ihS_{a'_{n},b'_{n}}f}u).
$$
Let 
$\DS 
\cA_{t}(u):=\cL_{a'_{n}}^{b'_{n}-a'_{n}}(e^{iS_{a'_{n},b'_{n}}g+ihS_{a'_{n},b'_{n}}f}u).
$
Then since $\DS \sup_j\|H_j\|_{\al}<\infty$ there is a constant $C>0$ such that 
$$
\|\cL_{t}^{A_{m_n}'}\|_{*}\leq C\|\cA_t\|_*.
$$
 Now,  by Proposition \ref{PrSmVarBlock} 
there exist
  constants $C_0>0$ and $c>0$ such that  if  $\del_1$ (and hence $|h|$) and $\ve$ are small enough
then 
$$\|\cA_t\|_*\leq C_0 e^{-cV_{n}(h)}$$
where $V_n(h)=\|S_{a_n',b_n'}(g+hf)\|_{L^2}^2$.
Thus
$$
\int_{J}\|\cL_{t}^{A_{m_n}'}\|_{*}dt=\int_{-\del_1/2}^{ \del_1/2}
\|\cL_{t_0+h}^{A_{m_n}'}\|_{*}dh\leq CC_0''\int_{-\del_1/2}^{ \del_1/2}
e^{-cV_{n}(h)}dh.
$$
Applying Lemma  \ref{Qqad L} with $Q(h)\!=\!V_n(h)\!=\!\|S_{a_{n},b_{n}}(g+hf)
\|_{L^2}^2$ and 
using Lemma~\ref{CombLemma}(i) we conclude that there  are constant $C',C''>0$ such that
$$
\int_{-\del_1/2}^{\del_1/2}\|\cL_{0,t_0+h}^{A_{m_n}'}\|_{*}dh\leq C'\left(\text{Var}(S_{a'_n,b'_n})\right)^{-1/2}\leq C''L_n^{{1/2}}\sig_n^{-1}
$$
where we have used that the quadratic form $\DS Q(h)$ is nonnegative, and \eqref{suff A prime} follows.
\end{proof}

\begin{proof}[Proof of Proposition \ref{case 2}]
By the submultiplicativity of the norm and Corollary \ref{CorNonExpanding} 
$$ \|\cL_{0,t}^{n}\|_*\leq \left(\prod_{k=1}^{\tilde L_n} \|\cL_t^{A_k}\|_*\right)\left(\prod_{j=1}^{L_n} \|\cL_t^{B_j}\|_*\right)
\leq (1-\eta(\ve))^{L_n} \|\cL_t^{A_{m_n}}\|_*.
$$
  Thus Lemma \ref{Norm Lem} gives
  $\DS
\int_{J}\|\cL_{0,t}^{n}\|dt\leq (1-\eta(\ve))^{L_n} {\sqrt{L_n}}\sig_n^{-1}
  $
  which is indeed $o(\sig_n^{-1})$ if $L_n$  diverges to infinity. 
\end{proof}

\subsection{Small number of contracting blocks}
\label{SSFewBlocks}
The third and last case we need to cover to complete the proof of \eqref{Suff J} is when $L_n$ is bounded.
\begin{proposition}\label{case 3}
 If $L_n$ is bounded then either 
\vskip0.1cm  
(i) $(f_j)$ is reducible  and, moreover, one can decompose
$$ f_{j}=g_{j}+H_{j-1}-H_{j}\circ T_{j-1}+(2\pi/t) Z_{t,j} $$
 with $t\in J$,\;\; $Z_{t,j}$ integer valued and  $g_j\circ T_0^j$ is a reverse martingale difference satisfying
 $\DS \sum_j \mathrm{Var}(g_j)<\infty$;
 or
 \vskip0.1cm  
(ii) 
$\DS
\int_{J}\|\cL_{0,t}^n\|_*dt=o(\sig_n^{-1}).
$
\end{proposition}

\begin{proof}

Suppose that $L_n$ is bounded, and let $N_0$ be the right end point of the last contracting block $B_{L_n}$. If $L_n=0$ we set $N_0=-1$. Set $N(\ve)=N_0+k_0+L(\ve)+k_2(\ve)+1$.
Then, for every $t\in J$ and $n\geq N(\ve)$ we have
$$
\|\cL_{0,t}^n\|_*=\|\cL_{N(\ve),t}^{n-N(\ve)}\circ\cL_{0,t}^{N(\ve)}\|_*\leq \|\cL_{N(\ve),t}^{n-N(\ve)}\|_*\|\cL_{0,t}^{N(\ve)}\|_{*}\leq \|\cL_{N(\ve),t}^{n-N(\ve)}\|_*
$$
where in the second inequality we have used \eqref{Norm bound}.

Now, by Lemma \ref{A lemma} applied with $A=\{N_0,...,n-1\}$ we see that  
there  are functions  $g_j,\tilde H_j, \tilde Z_j$ such that for all $j\geq N(\ve)$ we have
\begin{equation}\label{Redd}
t_0f_j=g_j+\tilde H_j-\tilde H_{j+1}\circ T_j+2\pi \tilde Z_j
\end{equation}
 where $t_0$ is the center of $J$ and $\DS \sup_j\|g_j\|_{\alpha}\leq C(\ve)$, with $C(\ve)\to 0$ as $\ve\to 0$. Hence, like in the previous cases, if we take $\ve$ and $\del_1$ small enough, then 
 it is enough  to bound the norm of the  operator $\cA_{n,t}$ given by
$$ 
\cA_{n,t}u=\cA_{n,t_0+h}u=\cL_{N(\ve)}^{n-N(\ve)}(e^{iS_{N(\ve),n-N(\ve)}(g+hf)}),
\text{ where } h=t-t_0. 
$$
Let $\del_0$  be the constant from Proposition \ref{PrSmVarBlock}.
Take $\ve$ and $\del_1$ small enough so that $\DS \sup_{|h|\leq\del_1}\sup_j\|g_j+hf_j\|_{\alpha}<\del_0$. 
 Applying Proposition \ref{PrSmVarBlock}  we have 
$$
\|\cA_{n,t_0+h}\|_*\leq C e^{-cQ_n(h)}$$
for some constant $c>0$
where, as before 
$$
 Q_n(h)=\text{Var}(\tilde S_n f)h^2+2h\text{Cov}(\tilde S_nf, \tilde S_n g)+\text{Var}(\tilde S_ng)
$$
where $\tilde S_nf=S_{N(\ve),n-N(\ve)}f$ and $\tilde S_n g=S_{N(\ve),n-N(\ve)}g$. Then $\tilde\sig_n:=\|\tilde S_nf\|_{L^2}\geq \sig_n-CN(\ve)$ for some constant $C$.
Thus, by Lemma \ref{Qqad L} there is a constant $A>0$ such that
\begin{equation}\label{l64}
    \int_{-\del_1/2}^{\del_1/2}\|\cA_{n,t_0+h}\|_*\leq 
A\sig_n^{-1} \exp\left(-c\min_{[-\delta_1/2, \delta_1/2]} Q_n(h)\right).
\end{equation}
Note that
$$
m_n=\min Q_n=\text{Var}(\tilde S_n(g+h_nf)) \geq 0
$$
where $h_n:=\mathrm{argmin 
 }Q_n=-\frac{\fb_n}{2\fa_n}$.
Thus, if $m_n\!\!\to\!\!\infty$ then \eqref{Suff J} holds. Let us suppose that 
$\DS \liminf_{n\to\infty} m_n<\infty.$ 
We claim that in this case either \eqref{Suff J} holds or $(f_j)$ is reducible to a lattice valued sequence of functions.
Before proving the claim, let us simplify the notation and write
$$
Q_n(h)=\sig_n^2(h-h_n)^2+m_n=\fa_n h^2+\fb_n h+\fc_n
$$
where $\fa_n=\tilde\sig_n^2.$ Thus $h_n=\mathrm{argmin 
 }Q_n=-\frac{\fb_n}{2\fa_n}$. 

We now consider two cases.
\vskip0.1cm
(1) For any subsequence  with $\DS \lim_{j\to\infty} m_{n_j}<\infty$ we have
 $|h_{n_j}|\geq \del_1$. Then 
 $$ \min_{[-\delta_1/2, \delta_1/2]} Q_{n_j}(h)\geq \frac{\fa_{n_j}\delta_1^2}{4}
  =\frac{
  \tilde\sig_{n_j}^2\delta_1^2}{4}
 $$
and so \eqref{Suff J} holds  by \eqref{l64}.
\vskip0.1cm
(2) It remains to consider the case when there is a subsequence 
$n_j$ such that $|h_{n_j}|\leq \delta_1$ and $\DS \bQ:=\lim_{j\to\infty} m_{n_j}<\infty.$
By taking further subsequence if necessary we may assume that the limit
$\DS \lim_{j\to\infty} h_{n_j}= h_0$ exists. Then for all $n$,
$$ Q_n(h_0)=\lim_{j\to\infty} Q_{n}(h_{n_j})=
\lim_{j\to\infty} \Var(\tilde S_{n}(g+h_{n_j} f))
$$$$
=\lim_{j\to\infty} \left(\Var(\tilde S_{n_j}(g+h_{n_j} f))-\Var(S_{n, n_j-n}f)
-2\text{Cov}(\tilde S_n (g+h_{n_j} f), S_{n, n_j-n}(g+h_{n_j} f))\right)
$$$$
\leq \lim_{j\to\infty} m_{n_j}-2\lim\inf_{j\to\infty} \text{Cov}(\tilde S_n (g+h_{n_j} f), S_{n, n_j-n}(g+h_{n_j} f))\leq \bQ+2C
$$ 
for some constant $C>0$, where the last inequality uses that $\left|\text{Cov}(f_j, f_{j+k}\circ T_j^k)\right|\leq c_0\del_0^k$ for some constants $c_0>0$ and $\del_0\in(0,1)$.
Since $Q_n(h_0)\leq \bQ+2C$ for all $n$ we obtain
$$ \limsup_{n\to\infty} \Var(\tilde S_n(g+h_0 f))\leq \bQ+2C.$$
Hence $\DS \sup_n\Var(S_n(g+h_0 f))<\infty$.
Thus, by \cite[Theorem 6.5]{DH2}   we can write
$$
h_{0}f_j+g_j=\mu_j(h_{0}f_j+g_j)+M_j+u_{j+1}\circ T_j-u_j
$$
with functions $u_j$ and  $M_j$ such that $\DS \sup_j\|u_j\|_{\al}$ and $\DS \sup_{j}\|M_j\|_{\al}$ are finite, $M_j\circ T_0^j$ is a reverse martingale difference with 
$\DS \sum_j\text{Var}(M_j)<\infty$. Combining this with \eqref{Redd} we conclude that 
$(t_0+h_0)(f_j)$ is reducible to a $2\pi \integers$ valued sequence of functions,
and the proof of Proposition \ref{case 3} is complete.
\end{proof}

\subsection{Proof of the main results in the irreducible case}
\label{Edge1}

 Combining the results of \S\S \ref{SSLrgeCB}--\ref{SSFewBlocks} we obtain \eqref{Suff} completing the proof of  Theorem \ref{LLT1}. \vskip1mm

To prove Theorem \ref{LLT Latt} we note that the analysis 
of \S\S \ref{SSLrgeCB}--\ref{SSFewBlocks}  (in particular the proof of 
Proposition \ref{case 3}) also shows that if $\cJ$ is an interval such that 
 $\int_{\cJ} \|\cL_{0, t}^n \|_{*}dt \neq 0$ then $(f_j)$ is reducible to $h\bbZ$
valued sequence for some $h$ with $\frac{2\pi}{h} \in \cJ.$
By the assumption of Theorem \ref{LLT Latt} such a reduction is impossible for $|h|>1$
(see Theorem \ref{Reduce thm}) and therefore \eqref{SuffLat} holds implying Theorem 
\ref{LLT Latt}.
\vskip1mm

 Theorem \ref{ThEdge1} follows by 
combining
  \cite[Proposition 7.1]{DH2}, the estimate \eqref{Suff} and \cite[Proposition 25]{DH} with $r=1$.

\section{Local limit theorem in the reducible case}\label{Red sec}

\subsection{The statement of the general LLT}

Let $\ka_0$ be a probability measure on $X_0$ which is absolutely continuous with respect to $\mu_0$ and $q_0=\frac{d\ka_0}{d\mu_0}$ is H\"older continuous with exponent $\al$.
Suppose that $(f_j)$ is a reducible sequence such that $\sig_n\to\infty$. 
Let
$a=a(f)$ be the largest positive number such that $f$ is reducible to an $a\bbZ$ valued sequence
 (such $a$ exists by Theorem \ref{Reduce thm}).  Let $\del=2\pi/a$ and  write 
\begin{equation}
\label{RedF}
\del f_j=\del\mu_j(f_j)+M_j+g_j-g_{j+1}\circ T_j+2\pi Z_j
\end{equation}
with $(Z_j)$ being an integer valued irreducible sequence, $g_j,M_j$ are functions 
such that
$\DS \sup_{j}\|g_j\|_\al<\infty$, $\DS \sup_j\|M_j\|_\al<\infty$, and $(M_j\circ T_0^j)$ is a reverse martingale difference with respect to the reverse filtration $(T_0^j)^{-1}\cB_j$ on the probability space $(X_0,\cB_0, \ka_0)$. 
Moreover, we have 
$\DS
\sum_{j}\text{Var}(M_j)<\infty.
$
Then by the martingale convergence theorem the limit $\DS \textbf{M}=\lim_{n\to\infty}S_n M$ exists. Set $\textbf{A}=\textbf{M}+g_0$. 
\begin{theorem}\label{LLT RED}
If \eqref{RedF} holds then
for every continuous function $\phi:\bbR\to\bbR$ with compact support, 
$$
\sup_{u\in a\bbZ}\left|\sqrt{2\pi}{ \sig_n}\bbE_{\ka_0}[\phi(S_n\!-\!u)]
\!\!-\!\!
\left(\!\! a\sum_{k}\int_{X_n}\!\!\bbE_{\ka_0}[\phi(k a+\textbf{A}\!-\!g_n(x))]
d\mu_n(x)\!\!\right)\!e^{-\frac{(u-\bbE[S_n])^2}{2\sig_n^2}}\right|
\!\!=\!\! o(1).
$$
\end{theorem}
\begin{remark}
(i) To demonstrate the roles of $\textbf{A}$ and $g_n$ define $Q_n(x,y)=\textbf{A}(y)-g_n(x)$. To simply the notation we assume that $a=2\pi$. Now, if $Q_n$ converges to the  uniform distribution on $[0,2\pi]$ with respect to $d\ka_0(y)d\mu_n(x)$
then 
$$
a\sum_{k}\int_{X_n}\!\!\bbE_{\ka_0}[\phi(k a+\textbf{A}-g_n(x))]
d\mu_n(x)\to \int_{-\infty}^\infty g(x)dx.
$$
Thus, even though the sequence is reducible, the local law is Lebesgue and we get the non-lattice LLT. This is similar to the fact that if we add a sequence $\mathfrak{U_n}$ which converges in distribution to the uniform distribution on $[0,1]$   to a sum of iid integer valued random variables $S_n$ and if  $\mathfrak{U_n}$ and $S_n$ are asymptotically independent then $Y_n=S_n+\mathfrak{U}_n$ obeys the non-lattice LLT.

 On the other hand, if $\textbf{A}(y)-g_n(x)$ converges in distribution to a constant $b$ then we get an LLT similar to Theorem \ref{LLT Latt}, but with the local law being the counting measure supported on $b+a\bbZ$. In general, the set of possible limits in distribution of the sequence $Q_n$ determines the possible local laws along the appropriate subsequences.
\vskip0.1cm

(ii) Applying the theorem with $\tilde f_j=f_j+g_{j+1}\circ T_j-g_j$ instead of $f_j$ we obtain 
that 
$$
\sup_{u\in a\bbZ}\left|\sqrt{2\pi}{ \sig_n}\bbE_{\ka_0}
[\phi(S_n+g_n\circ T_0^n\!\!-\!\!g_0\!\!-\!\!u)]
-\left(a\sum_{k}\bbE_{\ka_0}[\phi(k a+\textbf{A}_1)]\right)e^{-\frac{(u-\bbE[S_n])^2}{2\sig_n^2}}\right|= o(1)
$$
where $\textbf{A}_1=\textbf{M}\text{ mod }2\pi$ 
(or $\textbf{A}_1=\textbf{M}$). Now, the sequence $Q_n$ in part (i) becomes the single random variable $\textbf{A}_1$, and the same discussion applies with the distribution of $\textbf{A}_1$ determining the local law after subtracting a coboundary.
\vskip0.1cm
(iii) Similarly, 
it will also follow that 
$$
\sup_{u\in a\bbZ}\left|\sqrt{2\pi}{ \sig_n}\bbE_{\ka_0}[\phi(S_n-u+g_n\circ T_0^n)]
\!\!-\!\!\left(a\sum_{k}\bbE_{\ka_0}[\phi(ka+\textbf{A})]\right)e^{-\frac{(u-\bbE[S_n])^2}{2\sig_n^2}}\right|
= o(1)
$$
and the same discussion applies.
\end{remark}

\begin{remark}
Note that Theorem \ref{LLT RED} implies Theorem \ref{LLT Latt} since
in that case \eqref{RedF} holds with $g_j=M_j=0.$ We gave a different proof
in  Section \ref{SSKeyPropZN} since the computations in the general case are significantly
more complicated.
\end{remark}

\subsection{Proof of Theorem \ref{LLT RED}}\label{Bp}
 It suffices to prove the theorem in the case $\ka_0(f_j\circ T_0^j)=0$ for all $j$, that is, $\bbE[S_n]=0$ for all $n$ because this could always be achieved by subtracting a 
constant from $f_j.$

The first step of the proof is standard. 
In view of \cite[Theorem 10.7]{Bre} (see also  \S 10.4 there
and Lemma IV.5 together with arguments of Section VI.4 in \cite{HH}), it is enough to prove the theorem for functions $\phi\in L^1(\bbR)$ whose Fourier transform has compact support. In particular, 
 the inversion formula holds 
 \begin{equation}\label{FI}
 2\pi\phi(x)=\int_{-\infty}^{\infty} e^{itx}\hat\phi(t)dt
 \quad\text{where}\quad
\hat \phi(t)=\int_{-\infty}^\infty e^{-itx}\phi(x)dx.  
\end{equation}

Next, let $L>0$ be such that $\hat\phi$ vanishes outside $[-L,L]$. Then by \eqref{FI}
we have 
$$
\sqrt{2\pi}\bbE_{\ka_0}[\phi(S_n-u)]
=\frac{1}{\sqrt{2\pi}}\int_{-\infty}^\infty e^{-itu}\hat\phi(t)\bbE_{\ka_0}[e^{itS_n}]dt=
\frac{1}{\sqrt{2\pi}}\int_{-L}^L e^{-itu}\hat\phi(t)\bbE_{\ka_0}[e^{itS_n}]dt.
$$

Divide $[-L,L]$ into intervals $J$ of length $\del_1$ for some small $\del_1$, such that each interval $J$ which intersects $\del\bbZ$ is centered at some point in $\del\bbZ$ (recall that 
$\del=\frac{2\pi}{a}$). Then 
$$
\sqrt{2\pi}\bbE_{\ka_0}[\phi(S_n-u)]=
\sum_{J}\frac{1}{\sqrt{2\pi}}\int_{J}e^{-itu}\hat\phi(t)\bbE_{\ka_0}[e^{itS_n}]dt.
$$
Now, because of Theorem \ref{Reduce thm}, the arguments in the irreducible case show that
if $J$ does not intersect $\del\bbZ$ and $\del_1$ is small enough then  
$$
\sup_u\left|\int_{J}e^{-itu}\hat\phi(t)\bbE_{\ka_0}[e^{itS_n}]dt\right|\leq\|\hat\phi\|_\infty \int_{J}\|\cL_t^n\|_*dt=o(\sig_n^{-1}).
$$
Thus, denoting $J_k=[k\del -\del_2,k\del +\del_2]$, 
where  $\del_2=\frac12\del_1$, we see that
\begin{equation}\label{GetR}
\sqrt{2\pi} \sig_n\bbE_{\ka_0}[\phi(S_n-u)]=\sum_{k}\frac{\sig_n}{\sqrt{2\pi}}\int_{J_k}e^{-itu}\hat\phi(t)\bbE_{\ka_0}[e^{itS_n}]dt+o_{n\to\infty}(1).
\end{equation}

The proof of Theorem \ref{LLT RED} is based on the following result.
\begin{proposition}
    \label{Lemma Claim}
    For each $k$ we have 
\begin{equation*}
\sup_{u\in \frac{2\pi}{\del}\bbZ}\left|\frac{\sig_n}{\sqrt{2\pi}}\int_{J_k}e^{-itu}\hat\phi(t)\bbE[e^{itS_n}]dt-e^{-\frac12 u^2/\sig_n^2}\mu_n(e^{-i k g_n})\ka_0(e^{ik\textbf{A}})\hat\phi(k\del)\right|=o_{n\to\infty }(1).   
\end{equation*}
\end{proposition}

\noindent
Let us first complete the proof of the theorem relaying on Proposition \ref{Lemma Claim}. Since there are finitely many intervals $J_k$  inside $[-L, L]$, using \eqref{GetR} and 
 the proposition we get
$$
\frac{\sig_n}{\sqrt{2\pi}}\int_{J_k}e^{-itu}\hat\phi(t)\bbE_{\ka_0}[e^{itS_n}]dt=e^{-\frac12 u^2/\sig_n^2}\sum_{k}
\mu_n(e^{-ikg_n})\ka_0(e^{ik\bA})\hat\phi(k\del)+o_{n\to\infty}(1)
$$
uniformly in $u$. 
Next, notice that, for all $k$, 
$\DS
\mu_n(e^{-i\del k g_n})\ka_0(e^{ik\del \textbf{A}})\hat\phi(\del k)=\widehat{C_n}(\del k)
$
where 
$\DS
C_n(t)=\int_{X_0}\int_{X_n}\phi(t+\textbf{A}(x)-g_n(y)) d\ka_0(x) d\mu_n(y).
$

To complete the proof we use the Poisson summation formula to derive that
$$
\sum_k\widehat{C_n}(k\del)= a \sum_k C_n\left(k a\right).
$$

\subsection{Proof of Proposition \ref{Lemma Claim}}
Fix some $k$ and  denote $J=J_k$ and $t_0=k\del$.
Now, using \eqref{RedF} for $t=t_0+h=k\del+h\in J$, we have
$$
tf_j=k(M_j+g_{j}-g_{j+1}\circ T_j+2\pi  Z_j)+hf_j.
$$
Thus, 
$$
e^{itS_n}=e^{-ikg_n\circ T_0^n}e^{ik(S_n M+g_0)}e^{ih S_n f}.
$$
Next, take some $\ell<n$ and write 
$$
e^{itS_n}=e^{-ikg_n\circ T_0^n}e^{ik(S_n M-S_\ell M)}e^{ih(S_nf-S_\ell f)}
H_{ih,\ell} 
$$
where for all $z\in\bbC$,
$$
H_{z,\ell}=e^{ik(g_0+S_\ell M)+zS_\ell f}.
$$
Notice that for every function $q:X_0\to\bbR$,
$$
\cL_0^n(e^{itS_n}q)=e^{-ikg_n}\cL_{\ell}^{n-\ell}\left(e^{ik(S_{\ell,n-\ell}M+ihS_{\ell,n-\ell}f)}\cL_0^\ell (H_{ih,\ell} 
q)\right)
$$
where we recall that $\DS S_{\ell,n}f=\sum_{j=\ell}^{n-1}f_{j}\circ T_\ell^j$ ($S_{\ell,k}M$ is defined similarly). Notice that for every function $G$ we have 
$$
\cL_{\ell}^{n-\ell}(e^{ik(S_{\ell,n}M+ihS_{\ell,n}f)}G)=\cL_{\ell,n}^{(ih;t_0)}G
$$
 where $\DS
\cL_{s,n}^{(z;t_0)}=\cL_{s+n-1}^{(z;t_0)}\circ\cdots\circ\cL_{s+1}^{(z;t_0)}\circ\cL_{s}^{(z;t_0)}
$ and
$\cL_{s}^{(z;t_0)}(g)=\cL_s(e^{i k M_s+zf_s}g)$. 

Thus, recalling that  $\ka_0=q_0 d\mu_0$ we have
\begin{equation}\label{cf form}
 \bbE_{\ka_0}[e^{itS_n}]=\mu_0[e^{itS_n}q_0]=\mu_n[\cL_0^n(e^{itS_n}q_0)]
= \mu_n[e^{-ik g_n}\cL_{\ell,n-\ell}^{(ih;t_0)}G_{\ell,z}]
\end{equation}
where $G_{\ell,z}=\cL_0^\ell (H_{z,\ell}q_0)$.

 Next, consider the  
 Banach space $\cB_1$ of all sequences $u=(u_j)_{j\geq 0}$ of H\"older continuous functions $u_j:X_j\to\bbC$ such that 
 $\DS \|u\|:=\sup_j\|u_j\|_\al<\infty$. 
 Let the operator $\cA_j^{(u,z)}$ be defined by $\cA_j^{(u,z)} g=\cL_j(e^{ik u_j+zf_j})$.  We view these operators as  perturbations of the operators $\cL_j$. Then these operators are analytic in $(u,z)$ and are uniformly bounded in $j$. Moreover, $\cL_s^{(z;t_0)}=\cA_s^{(M,z)}$, where $M=(M_j)$.
This means that we can view the operators $\cL_s^{(z;t_0)}$ as analytic in $(M,z)$ perturbations of the operators $\cL_s$ (the perturbation is small if $s$ is large and $|z|$ is small). Thus, if $\ell$ is large enough (so that $\DS \sup_{s\geq \ell}\|M_s\|_\al$ is small) and $|z|$ is small enough by applying 
\cite[Theorem D.2]{DH2} with the operators $\cL_{s}^{(z;t_0)},s\geq \ell$, considered as small perturbations of  of the operators $\cL_s$, we get the following.
There are triplets consisting of a non-zero complex number $\la_{t_0,s}(z)$ a H\"older continuous function
$\eta_{t_0,s}^{(z)}$ and a complex bounded linear functional $\nu_{t_0,s}^{(z)}$ (on the space of H\"older functions) which are uniformly bounded, analytic in $z$ and
\begin{equation}\label{SG}
\cL_{\ell,n-\ell}^{(z;t_0)}=\la_{t_0,n-1}(z)\cdots \la_{t_0,\ell+1}(z)\la_{k,\ell}(z)\nu_{t_0,\ell}^{(z)}\otimes \eta_{t_0,n}^{(z)}
+O(\theta^n), \,0< \theta<1.
\end{equation}
Moreover, $\nu_{t_0,s}^{(z)}(\eta_{t_0,s}^{(z)})=\nu_{t_0,s}^{(z)}(\textbf{1})=1$.
Furthermore,
since $\DS \lim_{s\to\infty, z\to 0}\|kM_s+zf_s\|_\al=0$,
\begin{equation}\label{lam close}
\lim_{s\to\infty, z\to 0}|\la_{t_0,s}(z)-1|=0,
\end{equation}
\begin{equation}\label{h close}
\lim_{s\to\infty, z\to 0}\|\eta_{t_0,s}^{(z)}-1\|_\al=0
\end{equation}
and 
\begin{equation}\label{nu close}
\lim_{s\to\infty, z\to 0}\|\nu_{t_0,s}^{(z)}-\mu_s\|_\al=0.
\end{equation}

Setting $\la_{t_0,\ell,n}(z)=\la_{t_0,n-1}(z)\cdots
\la_{t_0,\ell+1}(z)\la_{t_0,\ell}(z)$,
we conclude that 
\begin{equation}\label{CharRep}
\bbE_{\ka_0}[e^{(ik\del+z)S_n}]=\mu_n(e^{-ikg_n}\eta_{t_0,n}^{(z)})
\nu_{t_0,\ell}^{(z)}(G_{\ell,z})\la_{t_0,\ell,n}(z)+O( \theta^n).
\end{equation}
From now on we will only consider complex parameters of the form $z=ih, h\in\bbR$.

\begin{lemma}
If $\ell$ is large enough and $|h|$ is small enough then for all $n$ large enough we have
\begin{equation}\label{uuup}
  |\la_{t_0,\ell,n}(ih)|\leq Ce^{-ch^2\sig_n^2}   
\end{equation}
for some constants $C,c>0$ and all $n$.
\end{lemma}
\begin{proof}
\eqref{uuup} 
follows from plugging in the function $\textbf{1}$ in both sides of \eqref{SG}, using \eqref{h close}, Proposition \ref{PrSmVarBlock} and that 
$\DS \sup_j\!\text{Var}(S_j M)\!\!<\!\!\infty$.    Note that we can absorb the term $O(\te^n)$ in $e^{-ch^2\sig_n^2}$ since  by  the exponential decay of correlations (\cite[Remark 2.6]{DH2}), $\sig_n^2=O(n)$.
\end{proof}

\begin{lemma}\label{Lemma Above}
(i) $\DS
\lim_{\ell\to\infty}\limsup_{h\to 0}\left|\nu_{t_0,\ell}^{(ih)}(G_{\ell,ih})-{\ka_0}[e^{ik(g_0+\textbf{M})}]\right|=0;
$

\vskip0.1cm
(ii)
$\DS
\lim_{\ell\to\infty}\limsup_{h\to 0}\sup_{n\geq \ell}\left|\mu_n(e^{-ikg_n}\eta_{t_0,n}^{(ih)})-\mu_n(e^{-it_0g_n})\right|=0.
$
\end{lemma}
 \begin{remark}
 Note  that 
$\DS
\ka_0[e^{ik(g_0+\textbf{M})}]=\ka_0[e^{ ik\textbf{A}}]
$
where  $\textbf{A}=(g_0+\textbf{M})\text{ mod }2\pi$.    
 \end{remark}

\begin{proof}
(i) In view of the Lasota-Yorke inequality (Lemma \ref{ll1}) we have 
\begin{equation}\label{esstt}
A:=\sup_{\ell}\sup_{|h|\leq1}\|G_{\ell,ih}\|_\al<\infty.  
\end{equation}
Now,
by \eqref{esstt} and \eqref{nu close}
we see that 
$$
\left|\nu_{t_0,\ell}^{(ih)}(G_{\ell,ih})-\mu_\ell(G_{\ell,ih})\right|\leq A\|\nu_{t_0,\ell}^{(ih)}-\mu_\ell\|_\al\to 0\,\,\text{ as }\,(\ell,h)\to (\infty,0).
$$
Next, since $(\cL_0^\ell)^*\mu_\ell=\mu_0$ we have 
$$
\mu_\ell(G_{\ell,ih})=\mu_0\left[e^{ik(g_0+ S_\ell M)+ih S_\ell f}q_0\right]. 
$$
Now, it is clear that  for every $\ell$,
$$
\lim_{h\to 0}\left|\mu_0[e^{ik(g_0+S_\ell M)+ih S_\ell f}q_0]-\mu_0[e^{ik(g_0+S_\ell M)}q_0]\right|=0. 
$$
In view of this estimate, to complete the proof of (i) it is enough to show that 
$$
\lim_{\ell\to \infty}\mu_0[e^{ik(g_0+S_\ell M)}q_0]=\mu_0[e^{ik(g_0+\textbf{M})}q_0],
$$
but this follows from the almost sure convergence of $S_\ell M$ to $\textbf{M}$ and the dominated convergence theorem.

(ii) By \eqref{h close} and since $\DS \sup_n\|g_n\|_{\infty}<\infty$ and 
$\DS \sup_{n}\sup_{|z|\leq r_1}\|\eta_{t_0,n}^{(z)}\|_\al<\infty$ (for some small $r_1$) we see that when $|h|\leq r_1$ we have
$$
\left|\mu_n(e^{-ikg_n}\eta_{t_0,n}^{(ih)})-\mu_n(e^{-ikg_n})\right|
=
\left|\mu_n\left(e^{-ikg_n}(\eta_{t_0,n}^{(ih)}-1)\right)\right|
\leq \|\eta_{t_0,n}^{(ih)}-1\|_\infty. 
$$
Now (ii) follows from \eqref{h close}, and the proof of the lemma is complete.
\end{proof}


Next, 
define $\Pi_{t_0,s}(z)=\ln \la_{t_0,s}(z), s\geq \ell$. Note that  $\Pi_{t_0,z}$ is well defined  when $\ell$  is large enough in view of \eqref{lam close}. Let 
$$
\Pi_{t_0,\ell,n}(z)=\sum_{s=0}^{n-1}\Pi_{t_0,s+\ell}(z).
$$
\begin{proposition}\label{est prop}
There exist constants $r_1,C_1>0$ and $\te\in(0,1)$ such that for every complex number $z$ with $|z|\leq r_1$ and all $\ell$ large enough and $n$ large enough
we have:
\vskip0.2cm
 (i) 
 $\DS
\left|\ln \bbE[e^{ikS_{\ell,n}M+zS_{\ell,n}f}]-\Pi_{t_0,\ell, n}(z)\right|\leq C_1|z|+o_{\ell\to \infty}(1)+O\left(\te^n\right);
 $
\vskip0.2cm

 (ii) 
 $\DS
\Pi_{t_0,\ell,n}(0)=o_{\ell\to\infty}(1)+O(\te^n);
 $
\vskip0.2cm

 (iii)
 $\DS
\frac{\Pi_{t_0,\ell,n}'(0)}{\sig_n}=o_{\ell\to\infty}(1)+O(\te^n);
 $
\vskip0.2cm

 (iv) $\DS\frac{\Pi_{t_0,\ell,n}''(0)}{\sig_n^2}=1+o_{\ell\to\infty}(1)+O(\te^n)$;
\vskip0.2cm

(v)
$
\DS\sup_{t\in[-r_1,r_1]}|\Pi_{t_0,\ell,n}'''(it)|\leq C_1\sig_n^2.
$
\end{proposition}
\begin{proof}[Proof of Proposition \ref{Lemma Claim} relying on Proposition \ref{est prop}]
By \eqref{CharRep}, uniformly in 
 $u\in \frac{2\pi}{\del} \bbZ$,
for all $\ell$ large enough we have
$$
\frac{\sig_n}{\sqrt{2\pi}}\int_{J_k}e^{-itu}\hat\phi(t)\bbE[e^{itS_n f}]dt
$$
$$
=\frac{\sig_n}{\sqrt{2\pi}}\int_{J_k}e^{-i(t-\del k)u}\hat\phi(t)\bbE[e^{itS_nf}]dt=\frac{\sig_n}{\sqrt{2\pi}}I_{k,\ell,n}(\del_2)
+O(\te^n)
$$
where 
$\DS
I_{k,\ell,n,u}(\del_2)=\int_{-\del_2}^{\del_2}
 F(n,k,h,\ell,u)
e^{\Pi_{k\del,\ell,n}(ih)}dh,
$
$$
F(n,k,h,\ell,u)=e^{-iuh}\hat\phi(k\del +h)\mu_n(e^{-ikg_n}\eta_{k\del,n}^{(ih)})\nu_{k\del,\ell}^{(ih)}(G_{\ell,ih}),
\quad\text{and}\quad
\del_2=\frac12\del_1.$$

Since $\sig_n=O(n)$ we have $\te^n=o(\sig_n^{-1})$. So  in order to prove 
 Proposition \ref{Lemma Claim},
it is enough to show that, for every $\ve>0$ there is an $\ell$ and an $N$ such that for all $n\geq N$ and all $u\in \frac{2\pi}{\del}\bbZ=a\bbZ$ we have
\begin{equation}\label{Goal}
\left|\sig_n I_{k,\ell,n,u}(\del_2)
-\sqrt{2\pi}e^{-\frac12 u^2/\sig_n^2}\mu_n(e^{-ik g_n})\mu_0(e^{ik\textbf{A}})\hat\phi(k\del )\right|<\ve.
\end{equation}

By Lemma \ref{Lemma Above} the term 
$F(n,k,h,\ell,u)$
is uniformly bounded in all the parameters $(n,k,h,\ell,u)$ if $\ell$ is large enough and $|h|$ is small enough. 
Now, by 
\eqref{uuup}, for all $\ell$ large enough and $h$ close enough to $0$ we  have
\begin{equation}\label{8.11 app}
\left|e^{\Pi_{k\del,\ell,n}(ih)}\right|\leq Ce^{-ch^2\sig_n^2}
\end{equation}
for some $c,C>0$ and all $n\in\bbN$.
Let $R>0$. Then if  also $|h|\geq R/\sig_n$ we have 
$$
\left|e^{\Pi_{k\del,\ell,n}(ih)}\right|\leq Ce^{-cR^2}.
$$
Thus, using the uniform boundedness of all the factors in $F(n,k,h,\ell,u)$ by taking $R$ and then $\ell$ large enough  we see that \eqref{Goal} will follow if  for all $n$ (large enough) we have
\begin{equation}\label{Goal1}
\sup_{u\in \frac{2\pi}{\del}\bbZ}\left|\sig_n I_{k,\ell,n,u,R}-\sqrt{2\pi}e^{-\frac12 u^2/\sig_n^2}\mu_n(e^{-ik g_n})\mu_0(e^{ik \textbf{A}})\hat\phi(k\del )\right|<\ve   
\end{equation}
where 
$$
I_{k,\ell,n,u,R}=\int_{|h|\leq R/\sig_n}F(n,k,h,\ell,u)dh.
$$
However, using  \eqref{8.11 app},  Lemma  \ref{Lemma Above}   and the continuity of $\hat\phi$ in order to prove \eqref{Goal1}
it is enough to show that for $R$ and $\ell$ large enough we have
\begin{equation}\label{Last est}
\sup_{u\in \frac{2\pi}{\del}\bbZ}\left|\sig_n\int_{|h|\leq R/\sig_n}e^{-iuh}e^{\Pi_{k\del,\ell,n}
(ih)}dh-\sqrt{2\pi}e^{-\frac12 u^2/\sig_n^2}\right|<\ve.   \end{equation}
In order to prove \eqref{Last est},
let us first write
$$
\sig_n\int_{|h|\leq R/\sig_n}e^{-iuh}e^{\Pi_{k\del,\ell,n}(ih)}dh=\int_{-R}^Re^{-iuh/\sig_n}e^{\Pi_{k\del,\ell,n}(ih/\sig_n)}dh.
$$
By Proposition \ref{est prop}(v)  and the Lagrange form of the second order Taylor remainder around $0$ of the function $ \Pi_{k\del,\ell,n}(ih)$
 we have 
$$
\Pi_{k\del,\ell,n}(ih/\sig_n)=\Pi_{k\del,\ell,n}(0)+(ih/\sig_n)\Pi_{k\del,\ell,n}'(0)-\frac{h^2}{2\sig_n^2}\Pi_{k\del,\ell,n}''(0)+O(|h|^3/\sig_n^3)\sig_n^2.
$$
Now, since $|h|\leq R$ the term $O(|h|^3/\sig_n^3)\sig_n^2$ is $o_{n\to\infty}(1)$ and thus it can be disregarded (uniformly in $u$). 
Next, by Proposition \ref{est prop}(iv) 
$$
\frac{\Pi_{k\del,\ell,n}''(0)}{\sig_n^2}=1+o_{\ell\to\infty}(1)+O(\te^n).
$$
Furthermore, by parts (ii) and (iii) of Proposition \ref{est prop}, the 
term $\Pi_{k\del,\ell,n}(0)+(ih/\sig_n)\Pi_{k\del,\ell,n}'(0)$
can be made arbitrarily close to $1$ when $\ell$ and $n$ are large enough. By taking $\ell=\ell(R)$ large enough we conclude that for all $n$ large enough
\begin{equation}
\label{IntR}
\sup_{u\in\frac{2\pi}{\del}\bbZ}\left|\sig_n\int_{|h|\leq R/\sig_n}e^{-iuh}e^{\Pi_{k\del,\ell,n}(ih)}dh-\int_{|h|\leq R}e^{-iuh/\sig_n}e^{-h^2/2}dh\right|<\frac12\ve.
\end{equation}
 Now 
Proposition \ref{Lemma Claim} follows by taking 
 $R$ so large that 
 $\DS
\sqrt{2\pi}\int_{|h|\geq R}e^{-h^2/2}dh<\frac12\ve,
 $
 taking $\ell=\ell(R)$ so large that \eqref{IntR} holds, and using that 
 $$
\int_{-\infty}^{\infty} e^{-i\al h}e^{-h^2/2}dh=\sqrt{2\pi}e^{-\al^2/2}
 $$
 for every real $\al$.
\end{proof}

\begin{proof}[Proof of Proposition \ref{est prop}]
(i) For $|z|$ small enough and $\ell$ large enough  we have 
$$
\bbE[e^{ikS_{\ell,n}M+zS_{\ell,n}f}]=\mu_\ell(\cL_{\ell,n-\ell}^{z;k\del}\textbf{1})=
\mu_n(\eta_{k\del,n}^{(z)})\la_{k\del,\ell,n}(z)+O(\te^n).
$$
By \eqref{h close} and since $\eta_{k\del,n}^{(z)}$ is analytic in both $z$ and $(ik M_j)_{j\geq \ell}$,
we see that 
$$
|\mu_n(\eta_{k\del,n}^{(z)})-1|\leq C|z|+o_{\ell\to\infty}(1)
$$
for some constant $C>0$. Hence we can take the logarithms of both sides to conclude that
$$
\ln \bbE[e^{ik S_{\ell,n}M+zS_{\ell,n}}]=\Pi_{k\del,\ell,n}(z)+O(|z|)+o_{\ell\to\infty}(1)+O(\te^n).
$$

(ii) Plugging in $z=0$ in the above we see that 
$$
\Pi_{k\del,\ell,n}(0)=\ln \bbE[e^{ik S_{\ell,n}M}]+o_{\ell\to\infty}(1)+O(\te^n).
$$
Since $M_j\circ T_0^j$ is a reverse martingale,
\begin{equation}\label{Conv}
\sup_{n\geq\ell}\|S_{\ell,n}M\|_{L^2}^2\leq\sum_{s\geq \ell}\text{Var}(M_s)=o_{\ell\to\infty}(1)
\end{equation}
and so 
$\DS
\lim_{\ell\to\infty}\sup_{n\geq\ell}\left|\ln \bbE[e^{ik S_{\ell,n}M}]\right|=0
$
proving (ii). 

(iii)+(iv)+(v). Let $\Lambda_{\ell,n}(z)=\ln\bbE[e^{ik S_{\ell,n}M+zS_{\ell,n}f}]$. Then by part (i), for every $z$ small enough and all $\ell$ large enough we have
$$
\left|\Lambda_{\ell,n}(z)-\Pi_{k\del,\ell,n}(z)\right|=O(|z|)+o_{\ell\to\infty}(1)+O(\te^n).
$$
Now, because the functions $\Lambda_{\ell,n}(z)$ and $\Pi_{k\del,\ell,n}(z)$ are analytic in $z$, using the Cauchy integral formula we see that  for $s=1,2,3$, in a complex neighborhood of the origin and uniformly in $\ell$ and $n$ we have,
\begin{equation}\label{Log diffs}
\left|\Lambda_{\ell,n}^{(s)}(z)-\Pi_{k\del,\ell,n}^{(s)}(z)\right|=O(|z|)+o_{\ell\to\infty}(1)+O(\te^n)
\end{equation}
where $g^{(s)}(z)$ denotes the  $s$-th derivative of a function $g$. 

To prove (iii), after plugging  in $z=0$ \eqref{Log diffs} with $s=1$  it is enough to show that 
\begin{equation}
\label{NumDenMom}
\left|\frac{\bbE[(S_{\ell,n}f)e^{ik S_{\ell,n}M}]}{\bbE[e^{ik S_{\ell,n}M}]}\right|\leq C\sig_n a_\ell
\end{equation}
for some constant $C>0$, with $a_\ell=o_{\ell\to\infty}(1)$.
By \eqref{Conv} we have
\begin{equation}\label{NumDenMom}
\lim_{\ell\to\infty}\sup_{n\geq \ell}\left|\bbE[e^{ik S_{\ell,n}M}]-1\right|=0.
\end{equation}
 Thus, it is enough to show that 
\begin{equation}\label{Now}
\left|\bbE[(S_{\ell,n}f)e^{ik S_{\ell,n}M}]\right|\leq C\sig_n a_\ell.
\end{equation}
To prove \eqref{Now}, we use that $\bbE[S_{\ell,n}f]=0$ to write 
$$
\bbE[(S_{\ell,n}f)e^{ik S_{\ell,n}M}]=\bbE[(S_{\ell,n}f)(e^{ik S_{\ell,n}M}-1)].
$$
Since $\DS |e^{ik S_{\ell,n}M}-1|\leq k|S_{\ell,n}M|$ we get
$$
\left|\bbE[(S_{\ell,n}f)e^{ik S_{\ell,n}M}]\right|\leq k \bbE[|(S_{\ell,n}f)(S_{\ell,n}M)|]\leq k\sig_n\|S_{\ell,n}M\|_{L^2}
$$ 
where the last step uses the  Cauchy-Schwartz inequality.
  Now \eqref{Now}
 follows from \eqref{Conv}.

Next we prove (iv). Like in the proof of (iii), using \eqref{Log diffs} with $z=0$ and $s=2$ it is enough to show that if $\ell=O(\sig_n)$ then
\begin{equation}\label{E}
|\Lambda_{\ell,n}''(0)\sig_n^{-2}|=o_{\ell\to\infty}(1).   
 \end{equation}
To prove \eqref{E} we first note that 
$$
\Lambda_{\ell,n}''(0)=\frac{\bbE[(S_{\ell,n}f)^2e^{ik S_{\ell,n}M}]}{\bbE[e^{ik S_{\ell,n}M}]}-\left(\Lambda_{\ell,n}'(0)\right)^2. $$
 Now, as shown in the proof of part (iii) we have 
 $$
\left(\Lambda_{\ell,n}'(0)\right)^2=(\sig_n^2)\cdot o_{\ell\to\infty}(1). 
 $$

 To complete the proof split
 $$
\bbE[(S_{\ell,n}f)^2e^{ik S_{\ell,n}M}]=
\bbE[(S_{\ell,n}f)^2(e^{ik S_{\ell,n}M}-1)]+\sig_{\ell,n}^2
 $$
 where $\sig_{\ell,n}^2=\text{Var}(S_{\ell,n}f)$. 
 By
 the exponential decay of correlations (\cite[Remark~2.6]{DH2}), 
 \begin{equation}\label{sig ell}
\sig_{\ell,n}^2= \sig_{n+\ell}^2-\sig_{\ell}^2+O(1)=
O(\sig_n^2)
 \end{equation}
  where the last step uses that $\sig_n\to\infty.$
By \cite[Proposition 3.3]{DH2}, we see that for all $p\geq 1$ 
 \begin{equation}\label{111}
 \|(S_{\ell,n}f)^2\|_{L^p}\leq c_p(1+\sig_{\ell,n})^2=O(\sig_n^2)
 \end{equation}
 where $c_p$ is a constant which does not depend on $\ell$ and $n$. By the 
 Cauchy-Schwartz inequality 
  $$
\left|\bbE[(S_{\ell,n}f)^2(e^{ik S_{\ell,n}M}-1)]\right|\leq  \|(S_{\ell,n}f)^2\|_{L^2}\|e^{ik S_{\ell,n}M}-1\|_{L^2}.
 $$ 
 Since $|e^{ik S_{\ell,n}{M}}-1|\leq k|S_{\ell,n}M|$, applying \eqref{Conv} 
and \eqref{111} with $p=2$ yields that  
 $$
\left|\bbE[(S_{\ell,n}f)^2(e^{ik S_{\ell,n}M}-1)]\right|\leq \sig_n^2o_{\ell\to\infty}(1).
 $$
 To complete the proof of (iv),  we use \eqref{NumDenMom}
  to control the denominator.  

 Finally, let us prove (v). This estimate essentially follows from  the proof of \cite[Proposition 7.1]{DH2}, but for the sake of completeness we will include some details. Let  $n>\ell$. First, like in the proof of \cite[Proposition 7.1]{DH2}
 we decompose $\{\ell,\ell+1,...,n\}$ into a union of disjoint sets $I_1,I_2,...,I_{m_n}$ such that $I_i$ is to the left of $I_{i+1}$, 
 $ m_n=m_n(\ell)\asymp \sig_{\ell,n}^2$ and the variance of 
 $\DS S_{I_m}=\sum_{j\in I_{m}}f_{j}\circ T_0^j,\, 1\leq  m\leq m_n$ is bounded above and below by two positive constants $A_1$ and $A_2$, which can be taken to be arbitrarily large. By taking $n$ large enough and using \eqref{sig ell}, 
we can ensure that $\sig_{n,\ell}^2 \asymp \sig_{n}^2$ and so
 $m_n\asymp \sig_{n}^2$.
 Next,
 set 
 $$
 \Lambda_{I_m}
 (z)=\ln\bbE[e^{ik S_{I_m}M+zS_{I_m}f}].
 $$
 Then,
 using part (i), together with the Cauchy integral formula for the derivatives of analytic functions,
  it is enough to show that there  are $C, \ve_0>0$ 
  such that for all $t\in[-\ve_0,\ve_0]$ and all $1\leq  m\leq  m_n$ we have
 \begin{equation}
 \label{ThirdDer}
\left|\Lambda'''_{I_m}(it)\right|\leq C.
 \end{equation}
  This was done in the proof of \cite[Proposition 7.1]{DH2} in the case $k=0$ 
  (when the term $S_{I_m}Z$ did not appear).
 In the present setting,
 using \cite[Proposition 6.7]{DH2} with the sequence $(M_j)$ we have 
$\DS \sup_n\|S_nM\|_{L^3}<\infty$, and so by the martingale convergence theorem 
$S_nM\to \textbf{M}$ in $L^3$. Consequently,
$\DS \max_{1\leq m\leq m_n(\ell)}\|S_{I_m}M\|_{L^{3}}\to 0$  as $\ell\to\infty$. Using again   \cite[Proposition 6.7]{DH2} but now with the sequence $(f_j)$ we see that 
$\DS \sup_m \|S_{I_m}f\|_{L^3}\!\!<\!\!\infty$.  
Using these estimates the proof of \eqref{ThirdDer} proceeds like in the case $k=0$.
 Namely, we  use the formula
\begin{equation}
\label{LogDer}
(\ln F)'''=\frac{F'''}{F}-\frac{3F'F''}{F^2}+\frac{2(F')^3}{F^3}.
\end{equation}
Taking $F(t)=\bbE[e^{ik S_{I_m}M+it S_{I_m}f}]$ and using that $\|S_{I_m}M\|_{L^3}$ and $\|S_{I_m}f\|_{L^3}$ are bounded by some constant independent of $m$, we see that the numerators in the  RHS of \eqref{LogDer}
 are uniformly bounded above. On the other hand, taking $t$ small enough and $\ell$ large enough we get $|F(t)|\geq \frac12$ and so the denominators are bounded away from $0$.
\end{proof}

\section{Two sided SFT}\label{Sec 9}

\subsection{Preliminaries}
\label{ScSFT-Gibbs}
Let $\tilde T_j:\tilde X_j\to \tilde X_{j+1}$ be a two sided non-autonomous  SFT and let $T_j:X_j\to X_{j+1}$ be the corresponding one sided one.
 We begin with a few remainders from \cite{DH2}.

Let $\pi_j:\tilde X_j\to X_j$ be given by 
$\DS
\pi_j((x_{j+k})_{k\in\bbZ})=(x_{j+k})_{k\geq 0}.
$

\begin{lemma} {\bf (Sequential Sinai Lemma)}[See \cite[Lemma B.2]{DH2}]\label{Sinai}
\\
\,
 Fix $\al\in(0,1]$ and let $\psi_j:\tilde X_j\to\bbR$ be uniformly H\"older continuous with exponent $\al$. Then there are uniformly H\"older continuous functions $u_j:\tilde X_j\to\bbR$ with exponent $\al/2$ and $\phi_j:X_j\to\bbR$  such that 
 $
\DS \psi_j=u_{j}-u_{j+1}\circ \sig_j+\phi_j\circ \pi_j.
 $
 Moreover, if $\|\psi_j\|_{\al}\to 0$ then $\|u_j\|_{\al/2}\to 0$. 
\end{lemma}

\begin{definition}
Let $\phi_j:X_j\to\bbR$ be a sequence of functions such that $\DS \sup_j \|\phi_j\|_\al<\infty$ for some $\al\in(0,1]$.  We say that a sequence of probability measures $(\mu_j)$ on $X_j$ is a sequential Gibbs measure for $(\phi_j)$ if: 

\vskip0.2cm
(i) For all $j$ we have $(T_j)_*\mu_j=\mu_{j+1}$;
\vskip0.2cm

(ii) There is a constant $C>1$ and a sequence of positive numbers $(\la_j)$ such that for all $j$ and every point $(x_{j+k})_k$ in $X_j$ we have 
$$
C^{-1}e^{S_{j,r}\phi(x)}/\la_{j,r}\leq \mu_j([x_j,...,x_{j+r-1}])\leq Ce^{S_{j,r}\phi(x)}/\la_{j,r}
$$
where 
$\DS
S_{j,r}\phi(x)=\sum_{s=0}^{r-1}\phi_{j+s}(T_{j}^s x)
$
and 
$\DS
\la_{j,r}=\prod_{k=j}^{j+r-1}\la_k.
$
\end{definition}

Sequential Gibbs measures on two sided shifts are defined similarly 
(see \cite[Appendix~B]{DH2}).

We say that two sequences $(\al_j)$ and $(\be_j)$
  of positive numbers are \textit{equivalent} if there is a sequence $(\zeta_j)$ of positive numbers which is bounded and bounded away from $0$ such that for all $j$ we have $\al_j=\zeta_j\beta_j/\zeta_{j+1}$.

\begin{theorem}
\label{ThSFT-Gibbs}
[See \cite[Theorem B.5]{DH2}]
    For every sequence of functions $\phi_j:X_j\to\bbR$,
    $j\in \bbZ$,
     (or $\phi_j:\tilde X_j\to\bbR$ for two sided shfits) such that $\DS \sup_j\|\phi_j\|_\al<\infty$ for some $\al\in(0,1]$ there exist  unique Gibbs measures $\mu_j$. Moreover, the sequence $(\la_j)$ is unique up to equivalence.
\end{theorem}
We note that the uniqueness holds when $X_j$ is defined for all $j\in\bbZ$. When it is only defined for $j\geq0$ then there are infinitely many ways to extend $X_j$ and the potentials $\phi_j$  to $j<0$, each of which results in a Gibbs measure.

\subsection{Conditioning}\label{ConditSec}
The proof of the LLT for the two sided shift uses  conditioning.
In this section we explain how this tool works.

Let $\psi_j\!\!:\tilde X_j\!\to\!\bbR$ be a sequence of functions such that 
$\DS \sup_j\!\|\!\psi_j\!\|_\al\!\!<\!\!\infty$ for some $0<\al\leq 1$, and let $\gamma_j$ be the corresponding sequential Gibbs measures, associated with a sequence $(\la_j)$. Let $\phi_j$ be the function from Lemma \ref{Sinai}. Let $\mu_j$ be the sequential Gibbs measure corresponding to $\phi_j$. Then $\mu_j$ is the restriction of $\gamma_j$ to the $\sig$-algebra on $Y_j$ generated by the coordinates indexed by $j+k$ for $k\geq 0$. 
 Let us recall the construction of Gibbs measure described in \S \ref{TPo}.
Define the operators $L_j$ by
$$
L_jg(x)=\sum_{y:\, T_jy=x}e^{\phi_j(y)}g(y).
$$
Then by \cite[Eq. (B3)]{DH2} there is a sequence of positive functions $h_j:X_j\to \bbR$ such that $\DS \inf_j\min_{x\in X_j} h_j(x)>0$ and $\DS \sup_j\|h_j\|_{\al/2}<\infty$, a sequence of probability  measures $\nu_j$ on $X_j$ such that $\nu_j(h_j)=1$ and a sequence of positive numbers $\la_j$ such that  $\DS 0<\inf_j\la_j\leq \sup_j\la_j<\infty$ and the following holds: 
$$
L_j h_j=\la_j h_{j+1}, \quad L_j^*\nu_{j+1}=\la_j \nu_j,
$$
and there are $C>0$ and $\del\in (0,1)$ such that for all $n$ and $j$ and all H\"older continuous functions $g$ with exponent $\al/2$ ,
\begin{equation*}
    \|(\la_{j,n})^{-1}L_j^n g-\nu_j(g)h_{j+n}\|_{\al/2}\leq C_0\|g\|_{\al/2} \del^n.
\end{equation*}
 Here
$$
L_{j}^n=L_{j+n-1}\circ\cdots\circ L_{j+1}\circ L_j, \la_{j,n}=\la_{j+n-1}\cdots \la_{j+1}\la_j.
$$
Then the unique sequential Gibbs measures $(\mu_j)$ corresponding to the sequence of potentials $(\phi_j)_{j\in\bbZ}$ are given by $\mu_j=h_j d\nu_j$ (see \cite[Appendix B]{DH2}), and the transfer operators of $(T_j)$ corresponding to $(\mu_j)$ are given by 
$$
\cL_jg(x)=\sum_{y:\, T_jy=x}e^{\tilde \phi_j(y)}g(y)
$$
where 
$\DS
\tilde \phi_j=\phi_j+\ln h_j-\ln h_{j+1}\circ T_j-\ln \la_j.
$
Then $\cL_j\textbf{1}=\textbf{1}$, $\cL_j^*\mu_{j+1}=\mu_j$ and $\cL_j$ satisfy the duality relation
$$
\int_{X_j}(f\circ T_j)\cdot gd\mu_j=\int_{X_{j+1}}f\cdot (\cL_j g) d\mu_{j+1} 
$$
for all bounded measurable functions $f$ and $g$. 

\begin{lemma}
$\DS \sup_j\|\tilde\phi_j\|_{\al/2}<\infty$.   
\end{lemma}
\begin{proof}
Since $\DS \inf_j\min_{x\in X_j} h_j>0$ and $\DS \sup_j\|h_j\|_{\al/2}<\infty$,  the functions $\ln h_j$ are uniformly H\"older continuous (with respect to the exponent $\al/2$). Since $\DS 0<\inf_j\la_j\leq \sup_j\la_j<\infty$ we conclude that 
$\DS \sup_j\|\tilde\phi_j\|_{\al/2}<\infty$.  
\end{proof}

Next, taking a random point $x$ in $\tilde X_0$ which is distributed according to $\gamma_0$ we get a random sequence of digits. Denote the $j-$th random digit by 
$\mathfrak{X}_j$.
     Our next  result is a non-stationary version of Dobrushin-Lanford-Ruelle equality.
\begin{lemma}\label{DLR}
For every point $x=(x_{j+k})_{k\in\bbZ}\in \tilde X_j$ we have
$$
\gamma_j([x_j,...,x_{j+m-1}]| \mathfrak{X}_{j+m}=x_{j+m},\mathfrak{X}_{j+m+1}=x_{j+m+1},\dots)=e^{S_{j,m}\tilde\phi(x)}.
$$
\end{lemma}
\begin{proof}
We have 
$$
\gamma_j([x_j,...,x_{j+m-1}]|\mathfrak{X}_{j+m}=x_{j+m},\mathfrak{X}_{j+m+1}=x_{j+m+1},...)=
$$
$$
\mu_j([x_j,...,x_{j+m-1}]|\mathfrak{X}_{j+m}=x_{j+m},\mathfrak{X}_{j+m+1}=x_{j+m+1},...).
$$

\noindent
We will
show that for every bounded measurable function $g:X_j\to \bbR$ and every $m\in\bbN$
\begin{equation}\label{Cond0}
    \mu_j(g|(T_j^m)^{-1}\cB_{j+m})=(\cL_{j}^m g)\circ T_j^m 
\end{equation}
where $\cB_{k}$ is the Borel $\sigma$-algebra on $X_{k}$ (the 
$\sigma$-algebra generated by the cylinders).
The desired result follows  from \eqref{Cond0} by taking $g$ to be the indicator 
of the cylinder $[x_{j},...,x_{j+m-1}]$. 

 Next, we prove
\eqref{Cond0}. Using that $(\cL_j^m)^*\mu_{j+m}=\mu_j$ and that $\cL_j^m(g(h \circ T_j^m))=g\cL_j^n h$ for every function $h$, 
we see that
for every bounded measurable function $h:X_{j+m}\to\bbR$ 
$$
\int g (h\circ T_j^m)d\mu_{j}=
\int \cL_{j}^m(g (h\circ T_j^m))d\mu_{j+m}
=\int (\cL_{j}^m g) hd\mu_{j+m}=
\int ((\cL_{j}^m g)\circ T_j^m) h\circ T_j^md\mu_{j}
$$
where in the last inequality we have used that $(T_j^m)_*\mu_j=\mu_{j+m}$. Since the above holds for every function $h$ we conclude that $ \mu_j(g|
(T_j^m)^{-1}\cB_{j+m})=(\cL_{j}^m g)\circ T_j^m$ and \eqref{Cond0} follows.
\end{proof}

A key tool in the reduction of the LLT from the two sided shift to the one sided shift is the following result.

\begin{proposition}[Reguality of densities after conditioning on the past]\label{Cond Dens Lemma}
For every $j$, the conditional distributions (with respect to $\gamma_j$) of the coordinates $y_{j+k}, k\geq 0$  given the coordinates $y_{j+s}, s<0$ (namely, the past) is absolutely continuous with respect to the distribution of $y_{j+k}, k\geq 0$ (i.e. $\mu_j$). Moreover, there is a constant $C\geq 1$ such that the corresponding Radon-Nikodym density $p(y_j,y_{j+1},...|y_{j-1},y_{j-2},...)$ satisfies  
$$
C^{-1}\leq p(y_j,y_{j+1},...|y_{j-1},y_{j-2},...)\leq C
$$
and
\begin{equation}
\label{ConDDHolder}
\|p(\cdot|y_{j-1},y_{j-2},...)\|_{\al/2}\leq C
\end{equation}
for almost every point $y=(y_{j+k})_{k\in\bbZ}$.
\end{proposition}
Note that Proposition \ref{Cond Dens Lemma} means that we can choose a version of the densities  
satisfying \eqref{ConDDHolder}.

\begin{proof}
We prove first that 
 $\gamma_j$ and $\gamma_j(\cdot|y_{j-1},  y_{j-2}, \dots)$
 are equivalent and that the densities are bounded and bounded away from $0$.
For every point $y\in Y_j$ and every cylinder of the form 
$\Gamma=[y_{j},...,y_{j+n-1}]$ and every $r>0$ we have 
$$
\gamma_j(\Gamma|y_{j-1},...y_{j-r})=\frac{\gamma_j([y_{j-r},...,y_{j+n-1}])}{\gamma_j([y_{j-r},...,y_{j-1}])}
=\frac{\gamma_{j-r}([y_{j-r},...,y_{j+n-1}])}{\gamma_{j-r}([y_{j-r},...,y_{j-1}])}.
$$
Applying the Gibbs property with the measure $\gamma_{j-r}$ to both cylinders 
$[y_{j-r},...,y_{j+n-1}]$ and $[y_{j-r},...,y_{j-1}]$,
we see that for some constant $C>1$ we have
$$
\gamma_j(\Gamma|y_{j-1},...y_{j-r})=C^{\pm 1} e^{S_{j,n}\tilde\phi(\pi_j(y))}
$$
where $a=C^{\pm 1}b$ means that $C^{-1}\leq  a/b\leq C$.
Applying again the Gibbs property with the measure $\gamma_j$ and the cylinder $\Gamma$ we also get that 
$$
\gamma_j(\Gamma)=C^{\pm 1}e^{S_{j,n}\tilde\phi(\pi_j(y))}.
$$
Taking $r\to\infty$ we conclude that the densities exist and they are uniformly bounded and bounded away from $0$.

Next, we prove that the densities  are  H\"older continuous functions of $(y_j,y_{j+1},...)$ uniformly in $y_{j-1}, y_{j-2},...$ (namely, the past). 
Fix some $m,r>0$. Then for every point $y\in Y_j$ we have
$$
\gamma_j([y_j,...,y_{j+m-1}]|y_{j-1},...,y_{j-r})=
\frac{\gamma_j([y_{j-r},y_{j-r+1},...,y_j,...,y_{j-m+1}])}{\gamma_j([y_{j-1},...,y_{j-r}])}
$$
$$
=\frac{\gamma_j([y_{j-r},...,y_{j-1}]|y_j,...,y_{j+m-1})\gamma_j([y_j,...,y_{j+m-1}])}{\gamma_j([y_{j-r},...,y_{j-1}])}.
$$
Therefore, 
\begin{equation}\label{Cond cruc}
    \frac{\gamma_j([y_j,...,y_{j+m-1}]|y_{j-1},...,y_{j-r})}{\gamma_j([y_j,...,y_{j+m-1}])}=
\frac{\gamma_j([y_{j-r},...,y_{j-1}]|y_j,...,y_{j+m-1})}{\gamma_j([y_{j-r},...,y_{j-1}])}
\end{equation}
$$
=
\frac{\mu_{j-r}([y_{j-r},...,y_{j-1}]|y_j,...,y_{j+m-1})}{\mu_{j-r}([y_{j-r},...,y_{j-1}])}.
$$
Thus, by Lemma \ref{DLR} 
\begin{equation}\label{T1}
   \lim_{m\to\infty}\frac{\gamma_j([y_j,...,y_{j+m-1}]|y_{j-1},...,y_{j-r})}{\gamma_j([y_j,...,y_{j+m-1}])}=\frac{e^{S_{j-r,r}\tilde \phi(\pi_{j-r}y)}}{\mu_{j-r}([y_{j-r},...,y_{j-1}])}. 
\end{equation}
On the other hand,  
$$
\frac{\gamma_j([y_j,...,y_{j+m-1}]|y_{j-1},...,y_{j-r})}{\gamma_j([y_j,...,y_{j+m-1}])}
\!\!=\!\!
\frac{1}{\gamma_j([y_j,...,y_{j+m-1}])}
\!\!\int_{[y_j,...,y_{j+m-1}]}p_j(y|y_{j-1},...,y_{j-r})d\mu_j(y)
$$
where $p_j(y|y_{j-1},...,y_{j-r})$ is the density of the coordinates indexed by $j+k$ for $k\geq0$ given the ones indexed by $y_{j-1},...,y_{j-r}$. Thus,
\begin{equation}\label{T2}
\lim_{m\to\infty}\frac{\gamma_j([y_j,...,y_{j+m-1}]|y_{j-1},...,y_{j-r})}{\gamma_j([y_j,...,y_{j+m-1}])}=p_j(y|y_{j-1},...,y_{j-r}),\,\,\,\gamma_j\,\,\text{ a.s.}
\end{equation}
Combining \eqref{T1} and \eqref{T2} we see that
$$
p_j(y|y_{j-1},...,y_{j-r})=\frac{e^{S_{j-r,r}\tilde \phi(\pi_{j-r}y)}}{\mu_{j-r}([y_{j-r},...,y_{j-1}])}. 
$$
The above formula shows that the distribution of $x=\pi_j(y)=(x_j,x_{j+1},...)$ given $y_{j-r}^{j-1}=(y_{j-1},...,y_{j-r})$
has density
$$
\frac{e^{S_{j-r,r}\tilde \phi([y_{j-r}^{j-1}, x])}}{\mu_{j-r}([y_{j-r},...,y_{j-1}])}\bbI(A^{(j-1)}_{y_{j-1},x_j}=1)
$$
with respect to $\mu_j$, where $[y_{j-r}^{j-1},x]=(y_{j-r},...,y_{j-1},x_j,x_{j+1},...)$
and $A^{(j)}$ are the incidence matrices of our shift.
Thus, the proof of the proposition will be complete if we prove that there is a constant $C_1>0$ such that for every $j$ and $r$ and all $x\in X_{j-r}$, 
\begin{equation}\label{su}
    \left\|\frac{e^{S_{j-r,r}\tilde \phi((x_{j-r},...,x_{j-1},\cdot))}}{\mu_{j-r}([x_{j-r},...,x_{j-1}])}\right\|_{\al/2}\leq C_1. 
\end{equation}
Indeed, once \eqref{su} is proven we can take $r\to\infty$ to get the result.

In order to prove \eqref{su}, we first notice that by the Gibbs property we have 
\begin{equation}\label{sup est}
       \left\|\frac{e^{S_{j-r,r}\tilde \phi((x_{j-r},...,x_{j-1},\cdot))}}{\mu_{j-r}([x_{j-r},...,x_{j-1}])}\right\|_{\infty}\leq C_2. 
\end{equation}
  Let
  $G_{\al/2}(\Psi)$ denote the H\"older constant of a function $\Psi$ corresponding to the exponent $\al/2$.
Since $\DS\sup_k\|\tilde\phi_k\|_{\al/2}\!\!<\!\!\infty$ we have 
$\DS
\sup_{x_{j-r},...,x_{j-1}}G_{\al/2}\left(S_{j-r,r}\tilde \phi((x_{j-r},...,x_{j-1},\cdot)\right)
\!\!\leq \!\! C_3
$
for some constant $C_3$ (since we ``freeze" the first $r$ coordinates).

Using that $|e^{t}-e^{s}|\leq (e^{t}+e^{s})|t-s|$ for all $t,s\in\bbR$ together with \eqref{sup est}, we obtain \eqref{su} with $C_1=2C_2C_3$, and the proof of the proposition is complete.
\end{proof}

\subsection{Reduciblity in the  two sided case}

Let $\gamma_j$ be  (sequential) Gibbs measures generated by some Gibbs measures $\mu_j$ on the one sided shifts $X_j$.  
Let $\psi_j:\tilde X_j\to\bbR$ be functions such that 
$\DS \sup_j\|\psi_j\|_\al<\infty$ for some $\al\in(0,1]$ and $\gamma_j(\psi_j)=0$ for all $j$. 
Consider the functions 
$\DS
S_n\psi=\sum_{j=0}^{n-1}\psi_j\circ \tilde T_0^j
$
as random variables on the probability space $(\tilde X_0, \text{Borel},\gamma_0)$.

\begin{proposition}\label{Red2sided}
 Let $\phi_j: X_j\to\bbR$ be the functions like in Lemma \ref{Sinai}. Then for every $h>0$ we have that $(\phi_j)$ is reducible to an $h\bbZ $-valued sequence iff $(\psi_j)$ is reducible to an $h\bbZ $-valued sequence.
\end{proposition}
\begin{proof}
First, it is clear that $(\psi_j)$ is reducible if $(\phi_j)$ is. Conversely, suppose that
 $(\psi_j)$ is reducible. 
 Then there are $h\!\!\neq\!\!0$ and 
 functions $H_j\!\!:\!\! \tilde X_j\!\!\to\!\! \mathbb{R}$, 
 $Z_j\!\!:\!\! \tilde X_j\!\!\to\!\! \mathbb{Z}$  
such that $\DS \sup_j
\|H_j\|_{\beta},<\infty$, $(S_n H)_{n=1}^\infty $ is  tight
and $\psi_j=H_j+hZ_j$ for all $j$. 
Applying \cite[Lemma 6.3 and Theorem 6.5]{DH2} 
with the sequence $(H_j)$ on the two sided shift (which is possible in view of Lemma \ref{Sinai})
  we can decompose
 $
 \psi_j=u_j-u_{j+1}\circ \tilde T_j+M_j+h Z_j,
 $
 where $M_j\circ \tilde T_0^j$ is a reverse martingale difference and 
 $\DS \sup_j\max(\|u_j\|_\beta,\|M_j\|_\beta)<\infty$. Moreover 
 $\DS \sum_{j}\text{Var}(M_j)<\infty$.
Now, since $M_j\circ \tilde T_0^j$ is a reverse martingale difference and 
$\DS \sum_{j}\text{Var}(M_j)<\infty$ 
we have that  with probability 1,
$S_{j,n}M$ can be made  arbitrarily small for large $j$. 
 Thus, by the Dominated Convergence Theorem we can ensure that
 $\bbE[e^{it S_{j,n}M }]$ is arbitrarily close to $1$ as $j\to\infty$,
 where $t=2\pi/h$. Now,  assume for the sake of contraction that $(\phi_j)$ is irreducible.  Then,  like in the proof of Corollary \ref{H corr} we see that the $\al/2$ H\"older operator
 norms of the transfer operators $\cL_{j,t}^n$ decays to $0$  as $n\to\infty$
  for every nonzero $t$, where $\cL_{j,t}(h)=\cL_j(he^{it\phi_j})$.
Next, we show that under this assumption for every $j$ we get that  $\bbE[e^{it S_{j,n}M }]\to 0$ as $n\to\infty$, which contradicts that   $\bbE[e^{it S_{j,n}M }]$ 
is close to $1$. This will complete the proof. 
In order to prove that $\bbE[e^{it S_{j,n}M }]\to 0$ as $n\to\infty$, let us first note that $M_j=\phi_j+v_{j+1}\circ \tilde T_j-v_{j}-hZ_j$ for some sequence of functions $v_j$ with 
$\DS \sup_j\|v_j\|_{\alpha/2}<\infty$. 
Conditioning on the past $y_{j-1},y_{j-2},...$ we have
$$
\gamma_j(e^{itS_{j,n} M})=
\gamma_0(e^{it S_{j,n} \phi+it v_{j+n}\circ\tilde T_j^n-it v_j})
$$
$$
=\int\left(\int_{X_0}e^{itS_{j,n}\phi(x)+itv_{j+n}(\tilde T_j^n(y^{-1},x))+itv_j(y^{-},x)}p_j(x|y^{-})d\mu_j(x)\right)d\gamma_j(y^{-})
$$
where $y^{-}=(...,y_{j-2},y_{j-1})$ and $p_j(x|y^{-})$ is the density of $\gamma_j$ 
conditioned on $x_k\!=\! y_k, k\!<\!j$, see \S \ref{ConditSec}.
 Next, since $\mu_j=(\cL_j^n)^*\mu_{j+n}$
for every realization $y^{-}$ we have 
$$
\int_{X_j}e^{itS_{j,n}\phi(x)+itv_{j+n}(\tilde T_j^n(y^{-1},x))+itv_j(y^{-},x)}p_j(x|y^{-})d\mu_j(x)
$$
$$
=\int_{X_{j+n}}e^{itv_{j+n}(y^{-},\cdot)}\cL_{j,t}^n(e^{it v_j(y^{-},\cdot)}p_j(x|y^{-}))d\mu_{j+n}.
$$
By Proposition \ref{Cond Dens Lemma}
we have $\|p_j(\cdot|y^{-})\|_{\al/2}\leq A$ for some constant $A$. Therefore
\begin{equation}\label{CharEst}
 |\gamma_j(e^{itS_{j,n}M})|\leq C\|\cL_{j,t}^n\|_{\al/2}.
\end{equation}
 By the foregoing discussion 
$ \DS \lim_{n\to\infty} \gamma_j(e^{itS_{j,n}M})=0$
  and the proof of the proposition is complete.
 \end{proof}

\subsection{Proof of Theorem \ref{LCLT two sided} in the irreducible case}
By
 Lemma \ref{Sinai} there are sequences of functions $\phi_j:X_j\to\bbR$ and 
 $u_j: \tilde X_j\to\bbR$ such that 
 $\DS \sup_j\|f_j\|_{\al/2}<\infty$,  $\DS \sup_j\|u_j\|_{\al/2}<\infty$ and 
 $\DS 
\psi_j=\phi_j\circ\pi_j+u_{j+1}\circ \sig_j-u_j.
 $
Let $\cL_{j}$ be the transfer operators corresponding to $\mu_j$ and for every $t\in\bbR$ let $\cL_{j,t}(g)=\cL_j(ge^{it \phi_j})$.

 As it was explained in \S \ref{SSInt-LLT},
 the non-lattice LLT 
and the first order expansions follow from  the  two results below.

\begin{lemma}\label{estt1}
There are constants $\del_0,C_0,c_0\!\!>\!\!0$ such that for every $t\!\in\![-\del_0,\del_0]$ we have
$$
|\gamma_0(e^{itS_n\psi})|\leq C_0e^{-c_0\sig_n^2 t^2}
$$
where $\sig_n=\|S_n\psi\|_{L^2}$.
\end{lemma}

\begin{lemma}\label{estt2}
Let $\del_0$ be like in Lemma \ref{estt1}. 
Under the irreducibility assumptions of Theorem \ref{LCLT two sided} for every $T>\del_0$ we have 
$\DS
 \int_{\del_0\leq |t|\leq T} 
|\gamma_0(e^{it S_n\psi})|dt=o(\sig_n^{-1}).
$
\end{lemma}

\begin{proof}[Proof of Lemma \ref{estt1}]
Arguing like in the proof of Proposition \ref{Red2sided}, we see that
there is a constant $C>0$ such that for all $t$ and $n$ we have
\begin{equation}\label{CharEst}
 |\gamma_0(e^{itS_n\psi})|\leq C\|\cL_{0,t}^n\|_{\al/2}.
\end{equation}
Now the result follows from the corresponding result in the one sided case, noting that $\|S_n\phi\|_{L^2}=\|S_n\psi\|_{L^2}+O(1)$.
\end{proof}

\begin{proof}[Proof of Lemma \ref{estt2}]
Since $\psi_j=\phi_j\circ \pi_j+u_{j+1}\circ \tilde T_j-u_j$, by Proposition \ref{Red2sided} 
the sequence of functions $(\phi_j)$ is also irreducible. 
Thus, the lemma follows from \eqref{CharEst} and 
 \eqref{Suff J} which holds in the irreducible case.
\end{proof}

\subsection{LLT in the reducible case}
Using Lemma \ref{Sinai} and the conditioning argument from \S \ref{ConditSec}, in the reducible case we can also prove an LLT similar to the one in Section~\ref{Red sec}. 
By Proposition \ref{Red2sided} we have $R(\phi)=R(\psi)$. In particular $a(\phi)=a(\psi)$, where $a(\cdot)$ was defined at the beginning of Section \ref{Red sec}. 
Now, 
the decay of the characteristic functions at the relevant points needed in the proof of Theorem \ref{LLT RED} can be obtained by repeating the arguments in the proof of Theorem \ref{LCLT two sided}. The second ingredient is to expand the characteristic functions around points in $(2\pi/a(\phi))\bbZ$. This is done by conditioning on the past and using an appropriate Perron-Frobenius theorem 
(see \cite[Theorem D.2]{DH}) for each realization on the past, and then integrating. In order not to overload the paper the exact details are left for the reader.



\section{Irreducible systems}\label{Irr Sec}

\subsection{The connected case}\label{PathConn}

Here we prove Theorem \ref{LLT2}. As we have explained before, it suffices to prove \eqref{Suff}. Hence Theorem \ref{LLT2} follows from the
estimate below.

\begin{proposition}\label{ExpDecP stretched}
 If all the spaces $X_j$ are  connected 
then for every $0<\del<T$ there are constants $c,C>0$ such that for all $n$,
\begin{equation}\label{expp}
\sup_{\del\leq |t|\leq T}\|\cL_{0,t}^n\|_{*}\leq Ce^{-c\sig_n}.
\end{equation}
Moreover, if $\sig_n\to \infty$ then $(f_j)$ is irreducible.
\end{proposition}

\begin{proof}
Like in the previous section, we can assume that $\mu_k(f_k)=0$ for all $k$.

Next,  fix a sufficiently small interval $J$ such that 
$J\cap (-\del_{0},\del_0)\!\!=\!\!\emptyset$ and let $L_n$  be 
 the number of contracting blocks as before. Then, it is enough to prove that 
\begin{equation}
\label{SupTL-Small}
\sup_{t\in J}\|\cL_{0,t}^n\|_{*}\leq Ce^{-c\sig_n}.
\end{equation}
If $L_n\geq  c\ln  \sig_n $ then  \eqref{SupTL-Small}
follows by repeating the proof of Proposition \ref{case 1} (note that the arguments  in \S \ref{SSLrgeCB}
 yield uniform in $t\in J$ bounds on the norms, and not only on average). 

Suppose next that $L_n\leq c \ln \sig_n$. Let us  reexamine to the proof  of Lemma \ref{Norm Lem} (in particular, we will use all the notations from there). Since all the spaces are  connected, combining
Lemma \ref{LmTDVar} and Remark \ref{Rem conn} we see that 
 that all the functions $Z_s$ appearing there are constants. Thus, for all $t=t_0+h\in J$  and $s\in A_{m_n}'$ we have
\begin{equation}\label{Chom}
   tf_s=g_{s,t}+tH_s-tH_{s+1}\circ T_s+z_{s,t}
\end{equation}
with $\|g_{s,t}\|_\al$ arbitrarily small,
and $z_{s,t}$ is a constant. Since $\DS \sup_s \|H_s\|_\al<\infty$ by 
Lemma~\ref{CombLemma} 
 there is a constant $c_1>0$ such that for all $n$ large enough we have 
$$
\|tS_{A_{m_n}'}\tilde g_{t}\|_{L^2}\geq c_1\sig_n L_n^{-1/2}\geq c_1c \sqrt{\sig_n}
$$
where $\tilde g_t=\{g_{s,t}+z_{s,t}: s\in I^{(n)}\}$. Note that $\|tS_{A_{m_n}'}\tilde g_{t}-tS_{A_{m_n}'} f\|_\infty\leq 2\sup_s\|H_s\|_\al$ and so $|\bbE[tS_{A_{m_n}'}\tilde g_{t}]|\leq 2\sup_s\|H_s\|_\al$. We conclude that for all $n$ large enough
$$
\sqrt{\text{Var}(tS_{A_{m_n}'}\tilde g_{t})}=\sqrt{\text{Var}(tS_{A_{m_n}'}g_{t})}\geq c_0 \sig_n-2\sup_s\|H_s\|_\al\geq \frac12  c_0\sig_n
$$
where $c_0=cc_1\delta_0$.
Since $\|g_{s,t}\|_\al$ are small, Proposition \ref{PrSmVarBlock} gives 
$\DS
\|\cL_{t,g_t}^{A_n}\|_*\leq e^{-\frac14 c_2 \sig_n}$ for some $c_2>0
$
where $\cL_{t,g_t}^{A_n'}$ is defined similarly to $\cL_{t}^{A_n'}$ but with $g_t$ instead of $tf$. 
 Now using \eqref{Chom} we obtain that
$\DS
\|\cL_{t}^{A_n'}\|_*\leq Ce^{-\frac14 c_2\sig_n}
$
for some constant $C>0$, and the proof of \eqref{expp} is complete.

Finally, we show that $(f_j)$ is irreducible. At the beginning of the  proof of Corollary~\ref{H corr}(iii) we showed that  if $t\in \mathsf{H}$ then the norms 
$\|\cL_{j,t}^n\|$ do not converge to $0$. On the other hand, by starting from $j$ instead of $0$ and then applying\footnote{Alternatively, note that the above proof of \eqref{expp} proceeds similarly if omit a few first iterates.}
\eqref{expp}  we get that for every given $t\not=0$ these norms must decay to $0$. Thus 
$\mathsf{H}=\{0\}$ and so by Corollary~\ref{H corr} the sequence  $(f_j)$ must be irreducible.
\end{proof}

 Note that the fact that the spaces are connected was only used in the derivation of \eqref{Chom}. 
We thus obtain the following result.

\begin{proposition}
\label{PrCHom}
Suppose that for each non-contracting block $B$ we have that \eqref{Chom} holds for $s\geq s(T), s\in B$.
If $\sig_n\to\infty$
then $(f_n)$ is irreducible.
\end{proposition}

\begin{proof}[Proof of Theorem \ref{Thm small}]
Taking $\alpha<\beta$ we obtain by Lemma \ref{al be} that $\|f_n\|_\alpha\to 0.$
Thus 
\eqref{Chom} holds with $H_s=z_{s,t}=0.$
\end{proof}

\subsection{Non-lattice LLT on the tori}
\label{SS2SideIrr}

\begin{proof}
[Proof of Theorem \ref{ThHyp}(b)]
Since the family $(T_n)$ is conjugated to a constant map it suffices to consider the case where $T_n\equiv T$ for all $n.$ Moreover by Franks-Manning Theorem
(\cite{Fr69, Man}) $T$ is conjugated to a linear map, so it suffices to prove the result when 
$T$ is linear $T(x) =\cA x$ mod 1 
where $\cA$ is a hyperbolic linear map
 (but the functions $(f_n)$ are different for different $n$ and they are only 
H\"older continuous).

Let $\Sigma$ be the symbolic system coding $T$ and define $F_n=f_n\circ \pi$ where
$\pi(\bx)$ is the point having symbolic expansion $\bx.$
(Below we denote by $x_n$ the symbols of $\bx$ and write $\bx_n=\sigma^n \bx.$) 
By Sinai's Lemma \ref{Sinai},
$F_n=\brF_n+\psi_n-\psi_{n+1}$ where $\brF_n$ depends only on 
indices $n+k$ for $k\geq 0.$ We want to show that $(F_n)$ is irreducible, which by 
Proposition~\ref{Red2sided} is equivalent to irreducibility of $(\brF_n).$

Given $\ell$ we say that orbits 
$\bx, \by, \bu, \bv$ form an {\em $(n, \ell)$ rectangle} if
$$x_{n-k}=u_{n-k}, \quad y_{n-k}=v_{n-k}, \quad
x_{n+k}=v_{n+k}, \quad y_{n+k}=u_{n+k}, \quad\text{for}\quad k\geq \ell.$$

Next, we recall the proof scheme of Theorem \ref{LLT1} (the non-lattice LLT). 
We divided the 
interval $[-T,T]\setminus(-\delta_0,\delta_0)$ into small intervals $J$. Then for each small interval the proof involved the number $L_n=L_n(J)$ of contracting blocks corresponding to $J$. More precisely, we had three cases: large, moderate and small number of blocks (i.e. $L_n$ is bounded).
That is, as a consequence of Propositions \ref{case 1}, \ref{case 2} and \ref{case 3} we saw that the non-lattice LLT can  fail only if $L_n$ is bounded (for some $J$) and $(F_n)$ is reducible. 	In particular, there is a constant $M>0$ such that every block of length larger than $M$  whose left end point exceeds $M$ is non-contracting. Henceforth, we suppose for the sake of contradiction that the non-lattice LLT fails. In what follows we will show that this assumption implies irreducibility, which yields the non-lattice LLT. It will follow that the non-lattice LLT holds.

Let us now fix large integers $\ell$ and $\brell$. Take a $\xi\in\bbR$ such that $\delta_0\leq |\xi|\leq T$, where $\delta_0$ and $T$ are two fixed positive numbers.
Let $n$ be large enough. Since $[n-\ell-\brell, n+\ell]$ is not a contracting block, by Lemma \ref{LmTDSmall}(a)  applied to the sum 
$\DS \bar S_N=\sum_{j=1}^N \brF_j$ we get
$$ \sum_{j=n-\ell-\brell}^{n+\ell} \left[\brF_j(\bx_j)+\brF_j(\by_j)-\brF_j(\bu_j)-\brF_j(\bv_j)\right]
=\sum_{j=n-\ell-\brell}^{n+\ell+\brell} \left[\brF_j(\bx_j)+\brF_j(\by_j)-\brF_j(\bu_j)-\brF_j(\bv_j)\right]
$$
\begin{equation}
\label{PreD}
=
h m(\bx, \by, \bu, \bv)+O(\theta^\brell)
\end{equation}
for some $\theta<1.$ Here $\DS h=\frac{2\pi}{\xi}$, 
$m(\cdot,\cdot,\cdot,\cdot)$ is an integer valued function  and the first equality holds
because for $j>n+\ell$ we have 
$\brF(\bx_j)=\brF(\bv_j)$ and $\brF(\by_j)=\brF(\bu_j).$

We claim that 
\begin{equation}
\label{NoWinding}
m(\bx, \by, \bu, \bv)=0.
\end{equation}
By   Lemma \ref{LmTDVar} (using the validity of \eqref{TempDecay} and that $n_0$ can always be increased) this is sufficient to conclude that \eqref{Chom}
holds for all $s$ large enough,
and by Proposition \ref{PrCHom} this is sufficient to prove irreducibility,
which yields the non-lattice LLT.\\ 

The proof of \eqref{NoWinding}
will be divided into several steps.
Given $(\bx, \by, \bu, \bv)$ as above let
$$ D_F(\bx, \by, \bu, \bv)=\sum_{j\in\mathbb{Z}} 
\left[F_j(\bx_j)+F_j(\by_j)-F_j(\bu_j)-F_j(\bv_j)\right].
$$
This series converges since $F_{n+k}(\bx_n)\!-\!F_{n+k}(\bv_n),$
$F_{n+k}(\by_n)\!-\!F_{n+k}(\bu_n)$, 
$F_{n-k}(\bx_n)\!-\!F_{n-k}(\bu_n)$, \\
and
$F_{n-k}(\by_n)-F_{n-k}(\bv_n)$ are exponentially small in $k.$
We note the following properties
$$ D_F(\bx, \by, \bu, \bv)=D_\brF(\bx, \by, \bu, \bv),$$
$$ D_F(\bx, \by, \bu, \bv)
=\sum_{|j-n|\leq \ell+\brell}
\left[F_j(\bx_j)+F_j(\by_j)-F_j(\bu_j)-F_j(\bv_j)\right]+O(\theta^\brell),
$$
$$ D_\brF(\bx, \by, \bu, \bv)
=\sum_{|j-n|\leq \ell+\brell}
\left[\brF_j(\bx_j)+\brF_j(\by_j)-\brF_j(\bu_j)-\brF_j(\bv_j)\right]+O(\theta^\brell). 
$$
Combining this with \eqref{PreD} we see that
\begin{equation}
\label{DF-M}
D_F(\bx, \by, \bu, \bv)=h m(\bx, \by, \bu, \bv)+O(\theta^\brell)
\end{equation}
and we shall use this identity to show that $m(\bx, \by, \bu, \bv)=0.$

Given orbits $\ba, \bb, \bc, \bd$ in $\T^d$ we say that they form 
{\em $(n, R)$ rectangle}
if $$c_n\in W^u(a_n)\cap W^s(b_n),\quad
d_n\in W^s(a_n)\cap W^u(b_n)$$
and the induced distances 
$\fd_u(a_n, c_n), \fd_u(b_n, d_n), \fd_s(a_n, d_n), \fd_s(b_n, c_n)$ are all smaller
than $R.$ We note that given $\ell$ there exists $R$ such that if 
$(\bx, \by, \bu, \bv)$ is an $(n, \ell)$ rectangle then 
$(\pi(\bx), \pi(\by), \pi(\bu), \pi(\bv))$ is an $(n, R)$ rectangle
and moreover
\begin{equation}
\label{DF-Df}
 D_F(\bx, \by, \bu, \bv)=D_f(\pi(\bx), \pi(\by), \pi(\bu), \pi(\bv)) .
 \end{equation}
The converse of the above statement need not be true, that is, if
$(\pi(\bx), \pi(\by), \pi(\bu), \pi(\bv))$ is an $(n, R)$ rectangle,
$(\bx, \by, \bu, \bv)$ is not necessarily an $(n, \ell)$ rectangle, 
 since $\pi^{-1}$ is not continuous (nor well defined). However, we shall use that the
converse
statement 
is close to being correct.

Recall that $Tx=\cA x$ mod 1. Thus given $(n, R)$ rectangle $(\ba, \bb, \bc, \bd)$ we
we can lift $(a_n, b_n, c_n, d_n)$ to get $\ta_n, \tb_n, \tc_n, \td_n, \ta_n^*\in \R^d$
so that
$$ \tc_n\in \ta_n+E^u , \quad \tb_n\in \tc_n+E^s, \quad \td_n\in \tb_n+E^u, 
\quad \ta_n^*\in d_n+E^s \quad \text{and}\quad \ta_n^*=\ta_n+k_n 
$$
with $k_n\in \Z^d,$ $\|k_n\|\leq 4R.$ Here $E^u$ and $E^s$ are expanding and contracting 
eigenspaces of $\cA.$ 
Indeed,  upon shifting the points $b_n, c_n, d_n$ by vectors in $\bbZ^d$ the rectangle, viewed as  a closed path from $a_n$ to $a_n$, can be lifted to a continuous piecewise linear path between $a_n$ and a point $\ta_n^*$ of the form $\ta_n^*=\ta_n+k_n$, with each linear part being in either the stable or the unstable direction.

Since  $E^u$ and $E^s$
are linear subspaces of complementary dimensions, given $a_n, b_n$ and $k_n$, the points
$c_n$ and $d_n$ and, hence, the whole orbits of these points are determined uniquely.
We shall denote the corresponding rectangle $\cR_n(a_n, b_n, k_n).$

\begin{lemma}
There exist $\eta\in(0,1)$ and $C>0$ such that
for all $n$, $L\in\bbN$ and $k$ 
$$\mathrm{mes}\left\{(a, b)\in\bbT^d\times\bbT^2: 
\cR_n(a, b, k) \text{ is not an image of an } (n, L)\text{ rectangle}\right\}\leq C\eta^L$$
\end{lemma}
where $\cR_n(a, b, k)$ denotes the $R$-rectangle at time $n$ formed by $a,b$ and $k$.
\begin{proof}
To simplify notation we prove the lemma for $n=0$. 
Let  $c$ and $d$ be other two  points of the rectangle. Since the set of points with unique coding has full measure, 
see \cite{Bowen} (or Proposition \ref{PrOneToOne}
in the present paper),
we can assume that there are unique points $x$ and $y$ such that  $a=\pi(x)$ and $b=\pi(y)$. Let $j>0$. Note that
the distance between $\cA^{-j}a$  and $\cA^{-j}c$ does not exceed 
$C\delta^j$ for some constants $C>0$ and $\delta\in(0,1)$ since $a-c$ is in the unstable direction (\text{mod }1), and $\cA^{-j}$ contracts this direction exponentially fast in $j$. Reversing the roles of the stable and unstable directions we see that the the distance between $\cA^{j}b$  and $\cA^{j}c$   does not exceed $C\delta^{j}$. 
Thus, if $c$ does not have a coding $c=\pi(u)$ with $u_j=x_j$ for all $j\leq -L$ and $u_j=y_j$ for all $j\geq L$ then either the points  $\cA^{-j}a$  and $\cA^{-j}c$  belong to the  $C\delta^{|j|}$ neighborhood of the boundary $\partial\cP$ of the Markov partition for some $j\leq -L$ or 
 the points  $\cA^{j}b$  and $\cA^{j}c$  belong to the  $C\delta^{j}$ neighborhood of  $\partial\cP$ for some $j\geq L$.
In particular either $\cA^{-j}a$ is exponentially close to the boundary for some $j\leq -L$ or  $\cA^{j}b$ is exponentially close to the boundary for some $j\geq L$. In this case we have
$$
(a,b)\in\left(\bigcup_{j\leq -L}\cA^{j}(B_{C\delta^{|j|}}(\cP))\times\bbT^d\right)\bigcup\left(\bbT^d\times\bigcup_{j\geq L}\cA^{-j}(B_{C\delta^{|j|}}(\cP))\right):=A_L
$$
where  for every measurable set $E$ and $\varepsilon$ the set $B_\varepsilon(E)$ is the $\varepsilon$ neighborhood of $E$ (here we view $\cA$ is acting on $\bbT^d$).
Now, each set $B_{C\delta^{|j|}}(\cP)$ has measure $O(\eta^{|j|}$) for some $\eta\in(0,1)$. Moreover, $\cA^{j}$ is measure preserving. Thus 
$$
\mathrm{mes}(A_L)=O(\eta^{L}).
$$
By reversing the roles of  $a$ and $b$ and the roles of the stable and unstable directions we see that  if $d$ does not have a coding with the same symbols of $y$ with places $j\geq -L$ and the same symbols as $x$ for $x\geq L$ then $(b,a)$ belongs to $A_L$, and the proof of the lemma is complete.
\end{proof}
Now \eqref{DF-Df} shows that if $L=\hell$ is large enough
and $[n-\hell-\brell, n+\hell]$ is not a contracting block,
then  $D_f(\cR_n(a, b, k))$ is close to $h\Z$ 
for $(a, b)$ on a set of large measure. 
On the other hand the map
$(a, b)\mapsto D_f(\cR_n(a, b, k))$ is uniformly H\"older, 
whence
$D_f(\cR_n(a, b, k))$ can not be close to $h\Z$ 
without being close to a fixed
$h \fm_n(k)$ for all $(a, b)$ with $\|a-b\|\leq 4R.$ 

We next claim that 
\begin{equation}
\label{Additive}
 \fm_n(k_1+k_2)=\fm_n(k_1)+\fm_n(k_2) 
 \end{equation}
provided that $\|k_1\|, \|k_2\|, \|k_1+k_2\|\leq 4R.$
To see why this is true, consider the rectangles $\cR_n(0,0, k_1)$ and $\cR_n(0,0, k_2)=\cR_n(k_1,k_1, k_2)$. After lifting these rectangles  to continuous paths on $\mathbb R^d$  
we get a path from the origin with four legs in the stable and unstable directions, alternately. When projected to $\bbT^d$ this path becomes $\cR_n(0, 0, k_1+k_2)=\cR_n(0, k_1, k_1+k_2)$. 
Thus, $\cR_n(0, 0, k_1+k_2)$ is the union of 
$\cR_n(0,0, k_1)$ and $\cR_n(0,0, k_2)$, and so 
$\DS D_f(\cR_n(0,0, k_1))+D_f(\cR_n(0,0, k_2))=D_f(\cR_n(0, 0, k_1+k_2)).$
Since the left hand side is close to $\fm_n(k_1)+\fm_n(k_2)$ while the right hand side is close to 
$\fm_n(k_1+k_2)$, 
\eqref{Additive} follows.

\eqref{Additive} shows that there is an integer vector $\fq_n$ such that
\begin{equation} \label{QnNorm}
\|\fq_n\|<\sC
\end{equation}
 and
$\fm_n(k)=\langle k, \fq_n\rangle$, where $\sC$ is a constant.
We claim that $\fq_n=0$ and so $\fm_n\equiv0$.
Indeed, since $n$ is large enough neither $[n-\ell-\brell, n+\ell]$ nor 
$[n-1-\ell-\brell, n-1+\ell]$ are contracting blocks.
Then taking $(\ba, \bb, \bc, \bd)$ which form both $(n, R)$ and $(n-1, R)$ rectangles
and using that $a_{n+1}^*-a_{n+1}=\cA (a_n^*-a_n)$,
we conclude that
$\DS \fm_{n}(\cA k)=\fm_{n-1}(k)$  and so 
\begin{equation}
\label{QnRec}
\fq_{n-1}=\cA^* \fq_{n}. 
\end{equation}
On the other hand, since $\cA$ induces a hyperbolic automorphism of $\bbT^d$, there is 
$r\!\!=\!\!r(\sC)\!\!\in \!\!\mathbb{N}$ such that if 
$q\!\in\! \Z^d\setminus 0$ satisfies $\|q\|\leq \sC$ then $\|(\cA^*)^r q\|\geq \sC$. 
Indeed, every nonzero integer vector must have a component in the unstable direction since the eignevalues of $\cA$ are irrational. Now  take $n$ large enough so that the above  hold with $n-i$ instead of $n$ for all $0\leq i\leq r$. This is possible if $n-\bar\ell-\ell-r\geq M$, where $M$ was specified at the begging of the proof. 
 Iterating \eqref{QnRec} we get that $\fq_{n-r}=(\cA^*)^r \fq_{n}$ and so 
either $\fq_n=0$ or $\|\fq_{n-r}\|\geq \sC$. Since the second option contradicts
\eqref{QnNorm}, we conclude that $\fq_n\equiv 0$.
 Hence $\fm_n(k)\equiv 0$ for all $k$ with $\|k\|\leq 4R$.
Now \eqref{NoWinding} follows from \eqref{DF-M} and \eqref{DF-Df}.
\end{proof}


\renewcommand{\theequation}{D.\arabic{equation}}

\end{document}